\documentclass[onefignum,onetabnum]{siamart190516}

\usepackage{amsmath}
\usepackage{amssymb}
\usepackage{comment}
\usepackage{graphicx}
\usepackage[caption=false]{subfig}
\usepackage{footnote}
\usepackage{textcomp}
\usepackage{mathrsfs}
\usepackage{mathtools}
\usepackage{epstopdf}
\usepackage{array}
\usepackage[maxfloats=99]{morefloats}
\usepackage{cases}
\usepackage{color}
\usepackage{cite}
\usepackage{bm}
\usepackage{mathtools}
\usepackage{enumitem}
\graphicspath{{./fig/}}

\newsiamremark{remark}{Remark}
\newsiamremark{Result}{Main Result}
\newsiamremark{hypothesis}{Hypothesis}
\crefname{hypothesis}{Hypothesis}{Hypotheses}
\newsiamthm{claim}{Claim}

\headers{Classical limit for VMSE with randomness}{S. Chen, Q. Li, and X. Yang}

\title{Classical limit for the varying-mass Schr\"odinger equation with random inhomogeneities}

\author{Shi Chen\thanks{Department of Mathematics, University of Wisconsin-Madison, Madison, WI, 53706, USA (schen636@wisc.edu).}
\and%
Qin Li\thanks{Department of Mathematics and Wisconsin Institute for Discovery, University of Wisconsin-Madison, Madison, WI, 53706, USA (qinli@math.wisc.edu).}
\and %
Xu Yang\thanks{Department of Mathematics, University of California, Santa Barbara, CA, 93106, USA (xuyang@math.ucsb.edu).}
}

\usepackage{amsopn}

\newcommand{\ri}{{\mathrm{i}}}

\newcommand{\eps}{\varepsilon}

\newcommand{\ue}{u^{\varepsilon}}
\newcommand{\uec}{\overline{u^{\varepsilon}}}
\newcommand{\we}{W^{\varepsilon}}
\newcommand{\invpi}{\frac{1}{(2\pi)^d}}
\newcommand{\bbr}{\mathbb{R}}
\newcommand{\epsh}{\frac{\varepsilon}{2}}
\newcommand{\half}{\frac{1}{2}}
\newcommand{\rmd}{\mathrm{d}}

\newcommand{\frmm}{\tilde{m}_0}

\newcommand{\stfrm}{\hat{m}_1(\omega,p)}
\newcommand{\frmrandalter}{\tilde{m}_1(\tau,p)}

\newcommand{\calo}{\mathcal{O}}
\newcommand{\teps}{t/\varepsilon}
\newcommand{\xeps}{x/\varepsilon}
\newcommand{\schr}{Schr\"odinger }
\newcommand{\nkl}{N_{\text{KL}}^{\varepsilon}}

\ifpdf
\hypersetup{
  pdftitle={Classical limit for VMSE with randomness},
  pdfauthor={S. Chen, Q. Li, and X. Yang}
}
\fi

\begin{document}

\maketitle

% REQUIRED
\begin{abstract}
  The varying-mass Schr\"odinger equation (VMSE) has been successfully applied to model electronic properties of semiconductor hetero-structures, for example, quantum dots and quantum wells. In this paper, we consider VMSE with small random heterogeneities, and derive a radiative transfer equation as its asymptotic limit. The main tool is to systematically apply the Wigner transform in the classical regime when the rescaled Planck constant $\epsilon\ll 1$, and expand the Wigner equation to proper orders of $\epsilon$. As a proof of concept, we numerically compute both VMSE and its limiting radiative transfer equation, and show that their solutions agree well in the classical regime.
\end{abstract}

% REQUIRED

% REQUIRED

\section{Introduction}
Quantum transport in quantum-size structures has become rather important with the recent progress in crystal growth technology and the designing of heterostructure nanoelectronic devices. In these applications, material property is spatially dependent and anisotropic, and accordingly, the Schr\"odinger equation is equipped with an effective mass term to incorporate the spatial variation~\cite{TsOgMi:1991,ScSc:2020}. The simulations have been used to understand, for example, the electron dynamics in a crystal with slowly varying composition, the current-voltage characteristics of quantum-well resonant tunneling diodes, among many others~\cite{ThEiHe:89, MoBr:84, Ei:90, VR:90,BeDu:1966,CoDuMaTi:1966}. The mathematical studies are concentrated on the derivation of the model from the classical Schr\"odinger equation by analyzing the electronic band structure~\cite{PoRi:1996,AlCaPi:2004,AlPi:2005,CaSp:2011,Sp:2006}. In many of these examples, one cannot typically a-priori uniquely determine the effective mass term experimentally, and thus randomness is included to describe the inhomogeneity. In other examples, the effective mass needs to be specifically designed for the device to have certain desirable property, such as cloaking~\cite{ZhGeSuZh:2008}, and for such an inverse problem, a thorough understanding of the forward problem with random media is a necessity.

We are interested in deriving the asymptotic limit of the following varying-mass Schr\"odinger equation
\begin{equation}\label{eqn:effect_mass_schr1}
\ri\varepsilon \partial_t \ue(t,x) + \frac{1}{2}\varepsilon^2 \nabla_x\cdot( m^\eps(t,x) \nabla_x \ue(t,x)) = 0\,,
\end{equation}
where $t>0$, $x \in \bbr ^d$ with $d \in \mathbb{N}$ and $\varepsilon\ll 1$ is the rescale Planck constant. The varying mass $m^\eps$ is assume to be random and highly oscillatory, with a given covariance matrix in time and space. We shall assume that $\ue$ decays fast enough at infinity to validate all the derivations. One goal of the paper is to show that in the $\varepsilon\to0$ regime, the Wigner transform of the solution converges to a special radiative transfer equation.

The problem is motivated by a fact that simulating~\cref{eqn:effect_mass_schr1} is extremely challenging in the classical regime ($\varepsilon \ll 1$), and deriving its asymptotic limit helps in guiding the design of the numerical scheme. The challenges are two-folded. In deterministic regime, $m^\eps(t,x)$ is a deterministic highly oscillatory function in $(t,x)$, and the oscillation is seen in the solution $\ue$ as well. Standard numerical solvers, in order to be accurate, have to resolve the small wavelength in $\ue(t,x)$. A typical example is the standard finite difference method used in~\cite{MaPiPo:99, MaPiPoSt:02} that requires a mesh size and time step of order $o(\varepsilon)$. The time-splitting spectral method~\cite{BaJiMa:02, BaJiMa:03} was developed to solve the Schr\"odinger equation with constant mass and varying potential term, and it fully makes use of the fast Fourier transform (FFT) to enhance computing speed. It does improve the mesh size to be of order $\mathcal{O}(\varepsilon)$ for constant-mass Schr\"odinger equation, but its application to VMSE does not appear to be straightforward. A bigger problem comes from the randomness in $m^\eps$. Since only the covariance of $m^\eps$ is given, numerically one has to find many realizations and compute the deterministic Schr\"odinger equation before calculating the ensemble mean or variance of the solution. The computational cost of each realization, however, increases at least algebraically as $\varepsilon\to 0$, as details in the random fluctuation become more and more important.

Many works are done to overcome the first difficulty, that is to obtain accurate numerical solutions without mesh resolving. One main approach is to explore the WKB-type ansatz, e.g., Gaussian beam methods~\cite{JiWuYa:08, JiMaSp:11} and frozen Gaussian approximation~\cite{HeKl:84, LuYa:12}. To a large extent, one applies the WKB-type ansatz $\ue(t,x) = A(t,x)\exp\left( \frac{i S(t,x)}{\varepsilon} \right)$ and derive the eikonal equation for $S(t,x)$ and transport-like equation for $A(t,x)$, with small scale $\varepsilon$ eliminated from the dynamics of $S(t,x)$ and $A(t,x)$. No such types of methods have been applied to efficiently solve~\cref{eqn:effect_mass_schr1} in the literature yet, not to mention its application to systems that present randomness. Another competing approach is to firstly derive, directly from the equation using the Wigner transform, the asymptotic limit on the theoretical level, and then impose numerical tricks that take advantage of the theoretical understandings. The hope is to develop methods that are ``asymptotic preserving'' (AP), meaning the methods preserve the asymptotic limit automatically with mesh size relaxed from small scale requirement. Methods such as~\cite{BaJiMa:02, BaJiMa:03,ShLuCa:2011,JiLuCa:2014,ChJiLi:2015,ChJiLi:2013,JiNo:2006,JiNo:2007} are all of this type.

We use the second approach in this paper. In particular, we are not yet interested in developing AP schemes, but rather deriving the asymptotic limit of VMSE~\cref{eqn:effect_mass_schr1} first using the Wigner transform~\cite{GeMaMaPo:1997}, a main tool in classical theory. The literature on deriving the asymptotic equations for wave propagation in random media is very rich~\cite{Sp:77,ErYa:2000,BaPaRy:2002,BaPi:2006,BaKoRy:2011, LuSp:07,ChJiLi:2015,ChJiLi:2013,BoGaSo:2019,BoGa:2018,GuRy:2016}. Most of the work starts with the Schr\"odinger equation with constant mass, and the randomness and high oscillations are introduced through the potential term. When it is the effective mass term that is random and highly oscillatory, the process of the derivation is similar but is much more delicate, as will be detailed later in this paper. As a proof of concept, we numerically demonstrates the derived radiative transfer equation by carefully computing and comparing its solution to the one of VMSE~\cref{eqn:effect_mass_schr1}, and show that the two solutions agree.

We note that the computation of the limiting Wigner equation is rather standard: we apply the standard WENO method. To compute VMSE as the reference is significantly harder due to the above listed reasons. To deal with the randomness, we employ the recent development in uncertainty quantification, and utilize KL (Karhunen-Lo\'eve) decomposition~\cite{loeve,Xi:2010} for representing the randomness, upon which, Monte Carlo sampling is used for each random component. For a high accuracy, a large number of random variables are included to present the fine structure of the randomness, and accordingly, fine discretization in the spatial domain is needed. We have not been able to find examples in the literature that study the asymptotic limit numerically except an attempt using Monte Carlo solver for wave equation in~\cite{BaPi:2006}.

The rest of the paper is organized as follows. To better illustrate the derivation, we start with a simpler case where $m^\eps = m_0(t,x)$ is deterministic, and independent of $\varepsilon$, and derive the limiting radiative transfer equation by the Wigner transform in~\Cref{sec:Schr_der}. In~\Cref{sec:Schr_ran}, we systematically introduce the derivation of the limiting equation for the varying-mass Schr\"odinger equation~\cref{eqn:effect_mass_schr1} with random heterogeneities. We present our numerical validation in~\Cref{sec:numerics} and make conclusive remarks in~\Cref{sec:conclusion}.

\section{Wigner transform of VMSE in the deterministic setting}\label{sec:Schr_der}

As a preparation, in this section, we first derive the classical limit for VMSE~\cref{eqn:effect_mass_schr1} with deterministic and slow-varying mass. The extension to incorporate the randomness is left to~\Cref{sec:Schr_ran}. This is to consider:
\begin{equation}\label{eqn:effect_mass_schr_m}
\ri\varepsilon \partial_t \ue(t,x) + \frac{1}{2}\varepsilon^2 \nabla_x\cdot( m_0(t,x) \nabla_x \ue(t,x)) = 0\,,\quad \ue(0,x) = \ue_{\mathrm{I}}(x)\,.
\end{equation}
The varying mass $m_0$ is a real function of $x$ and is independent of $\varepsilon$. It is taken to be deterministic. $\ue$ is a complex function and is associated with some primary physical quantities. The most basic ones are particle density $\rho^\varepsilon$ and current density $J^\varepsilon$, which are calculated by
\[
\rho^{\varepsilon}(t,x) = |\ue(t,x)|^2\,,\quad J^{\varepsilon}(t,x) = \varepsilon\text{Im}\left(m_0(t,x)\overline{\ue(t,x)}\nabla_x\ue(t,x)\right)\,.
\]
They are both quadratic functionals of $\ue(t,x)$. It is straightforward to derive the following conservation law:
\[
\partial_t \rho^{\varepsilon} + \nabla_x\cdot J^{\varepsilon} = 0\,.
\]

A more general definition of physical observables can be given using phase space symbols and Weyl quantization~\cite{Ho:85}. To make it more explicit, for $T>0$, let $\{a(t)\}_{t\in[0,T]}\subset \mathcal{S}'(\mathbb{R}^{2d})$ be a family of tempered distributions serving as symbols, then using Weyl quantization, we can define a family of pseudo-differential operators $\{a^{\mathrm{W}}(t,x,\varepsilon D_x)\}_{t\in[0,T]}$ on Schwartz space $\mathcal{S}(\mathbb{R}^d)$: for each $t\in[0,T]$
\begin{equation}
(a^{\mathrm{W}}(t,x,\varepsilon D_x)f)(x) = \frac{1}{(2\pi)^d}\int_{\mathbb{R}^{2d}}a\left(t, \frac{x+y}{2},\varepsilon k \right) f(y) e^{\ri (x-y)k}\rmd y \rmd k\,,
\end{equation}
where $\varepsilon D_x = -\ri \varepsilon \nabla_x$. The expectation value of the symbol $a(t)$ at time $t$ is then defined as a quadratic functional of wave function $\ue(t)$:
\begin{equation}
a[\ue(t)] = \left\langle \ue(t), a^{\mathrm{W}}(t,x,\varepsilon D_x)\ue(t)\right\rangle_{L^2}\,,
\end{equation}
where $\langle\cdot,\cdot\rangle_{L^2}$ denotes the $L^2$-inner product.

\begin{remark}\label{rem:wellpose}
We consider Schwartz solutions of VMSE~\cref{eqn:effect_mass_schr_m} in the following. Suppose $m_0$ and all its derivatives are bounded, {\it i.e.}, $m_0 \in C^{\infty}([0,T]\times\mathbb{R}^d)$ with
\begin{equation}\label{eqn:m0_H1}
|\partial_t^\beta\partial_x^{\alpha} m_0(t,x)| \leq C_{\alpha,\beta} \quad \forall \alpha \in \mathbb{N}^d\,, \quad \beta\in\mathbb{N}\,, \quad (t,x)\in[0,T]\times\mathbb{R}^d  \,,
\end{equation}
and that $m_0$ satisfies elliptic condition
\begin{equation}\label{eqn:m0_H2}
m_0(t,x)\geq C >0, \quad \forall (t,x)\in[0,T]\times\mathbb{R}^d\,,
\end{equation}
then by the regularity theory of evolution equations~\cite{Ka:1985,Li:1961}, there is a unique Schwartz solution in $C^\infty([0,T];\mathcal{S}(\mathbb{R}^d))$ to the VMSE~\cref{eqn:effect_mass_schr_m} with $\ue_\mathrm{I}\in \mathcal{S}(\mathbb{R}^d)$.
\end{remark}

Wigner transform is a technique explored in~\cite{RyPaKe:1996} for the \schr equation with random potential, and has been demonstrated as a very powerful tool for investigating the classical limit~\cite{GeMaMaPo:1997}. Fixed $\varepsilon>0$, given $\ue(t) \in \mathcal{S}(\mathbb{R}^d)$, it is defined as a function on the phase space:
\begin{equation}\label{eqn:wigner_transform}
\we(t,x,k) = \invpi \int_{\bbr^d}e^{\ri ky}\ue\bigg(t,x-\epsh y\bigg) \uec\bigg(t,x+\epsh y\bigg)\mathrm{d}y\,.
\end{equation}
Here $\uec$ is the complex conjugate of $\ue$. Thus, the Wigner transform maps function on $\mathcal{S}(\mathbb{R}^d)$ to $\mathcal{S}(\mathbb{R}^{2d})$. This definition is essentially the Fourier transform of
\[
\left\langle x-\frac{\varepsilon}{2}y\middle|\ue\right\rangle \left\langle \ue\middle|x+\frac{\varepsilon}{2}y \right\rangle
\]
in the $y$ variable.

The Wigner transform loses phase information: meaning for any $S(t)$, the Wigner transform defined by $\ue(t)$ and that defined by $\ue(t) e^{\ri S(t)}$ are the same, and hence cannot capture the phase difference $S(t)$. Moreover, it is not guaranteed that $\we(t)$ is positive, and thus it does not serve directly as the particle density on the phase space. However, the quantum expectation of physical observables can be easily recovered using the Wigner function, namely,
\begin{equation}
a[\ue(t)]=\langle \ue(t), a^{W}(t,x,\varepsilon D_x) \ue(t) \rangle_{L^2} = \int_{\mathbb{R}^{2d}}a(t,x,k)\we(t,x,k)\rmd x \rmd k\,.
\end{equation}

In particular, the first and second moments of $\we(t,x,k)$ in velocity $k$ provide the particle density $\rho^{\varepsilon}(t,x)$ and the current density $J^{\varepsilon}(t,x)$:
\begin{equation}
\begin{aligned}
\int_{\bbr^d} \we (t,x,k) \rmd k = \rho^{\varepsilon}(t,x)\,,\quad \int_{\bbr^d} m_0(t,x) k \we (t,x,k) \rmd k = J^{\varepsilon}(t,x)\,.
\end{aligned}
\end{equation}
By plugging in the \schr equation, we derives the equation satisfied by $\we$ in the following Lemma.
\begin{lemma}\label{lem:W}
Suppose $\ue\in C^\infty([0,T];\mathcal{S}(\mathbb{R}^d))$ solves the VMSE~\cref{eqn:effect_mass_schr_m} with $m_0$ satisfying~\cref{eqn:m0_H1} and~\cref{eqn:m0_H2}. Then the Wigner transform $\we\in C^\infty([0,T];\mathcal{S}(\mathbb{R}^{2d}))$ of $\ue$ satisfies the Wigner equation
\begin{equation} \label{eqn:wignereqn}
\begin{aligned}
&\partial_t \we (t,x,k) + \frac{1}{\varepsilon}\mathcal{Q}^{\varepsilon}_1 \we (t,x,k) + \mathcal{Q}^{\varepsilon}_2 \we (t,x,k) = \varepsilon\mathcal{Q}^{\varepsilon}_3 \we (t,x,k)\,, \\
&\we(0,x,k) = \we_{\mathrm{I}}(x,k)\,,
\end{aligned}
\end{equation}
where the operators $\mathcal{Q}^\varepsilon_i$'s are given by
\begin{equation}\label{eqn:q1}
\mathcal{Q}^{\varepsilon}_1 \we (t,x,k)
= \frac{|k|^2}{2}\int_{\bbr^d} \frac{e^{\ri px}}{(2\pi)^d}\tilde{m}_0(t,p) \ri \left[\we\left(t,x,k-\epsh p\right)-\we\left(t,x,k+\epsh p\right)\right]\rmd p
\end{equation}
\begin{equation}\label{eqn:q2}
\begin{aligned}
\mathcal{Q}^{\varepsilon}_2 \we (t,x,k)& \\
=& \frac{k}{2}\cdot\int_{\bbr^d} \frac{e^{\ri px}}{(2\pi)^d}\tilde{m}_0(t,p) \left[\nabla_x\we\left(t,x,k-\epsh p\right)+\nabla_x\we\left(t,x,k+\epsh p\right)\right]\rmd p
\end{aligned}
\end{equation}
\begin{equation}\label{eqn:q3}
\begin{aligned}
\mathcal{Q}^{\varepsilon}_3 \we (t,x,k)& \\
=& \frac{1}{8}\int_{\bbr^d} \frac{e^{\ri px}}{(2\pi)^d}\tilde{m}_0(t,p) \ri\left[\Delta_x\we\left(t,x,k-\epsh p\right)-\Delta_x\we\left(t,x,k+\epsh p\right)\right]\rmd p\\
&+\frac{1}{8}\int_{\bbr^d} \frac{e^{\ri px}}{(2\pi)^d}\tilde{m}_0(t,p) \ri |p|^2\left[\we\left(t,x,k-\epsh p\right)-\we\left(t,x,k+\epsh p\right)\right]\rmd p\,.
\end{aligned}
\end{equation}
The spatial Fourier transform of $m_0(t)$ is defined by
\begin{equation}
\tilde{m}_0(t,p) = \int_{\mathbb{R}^d}e^{-\ri pz}m_0(t,z)\rmd z\,.
\end{equation}
In addition, $\we_\mathrm{I}$ is the Wigner transform of initial condition $\ue_\mathrm{I}\in \mathcal{S}(\mathbb{R}^d)$\,.
\end{lemma}
The derivation is rather tedious, and we leave it to~\cref{appendix}. % The extension to time-dependent $m^{\varepsilon}$ is straightforward.

We now turn to the derivation of classical limit. Formally, we let $\varepsilon\to 0$ in the operator $Q_i^\varepsilon$'s and obtain:
\[
\frac{1}{\eps}\mathcal{Q}^{\varepsilon}_1 \we (t,x,k) =-\frac{|k|^2}{2}\nabla_x m_0(t,x)\cdot\nabla_k W^0(t,x,k) + O(\varepsilon^2)\,,
\]
and
\[
\mathcal{Q}^{\varepsilon}_2 \we (t,x,k) = m_0(t,x) k\cdot \nabla_x W^0(t,x,k) + O(\varepsilon^2)\,.
\]
This leads to the Liouville equation as a limit for~\cref{eqn:wignereqn}
\begin{equation}\label{eqn:limiteqn_m}
\partial_t W^0(t,x,k) + m_0(t,x)k\cdot\nabla_x W^0(t,x,k) - \frac{|k|^2}{2}\nabla_x m_0(t,x) \cdot \nabla_k W^0(t,x,k) + O(\varepsilon^2) = 0\,.
\end{equation}
Without the higher order terms, this limiting equation is the push-forward of the initial data
\begin{equation}
W^0(t,x,k) = W_{I}(\theta_{-t}(x,k))\,,
\end{equation}
under the flow that is generated by the Hamiltonian $H(t,x,k) = \frac{1}{2}m_0(t,x)|k|^2$, with the trajectories $\theta_t:\, \mathbb{R}^{2d}\to\mathbb{R}^{2d}$ follow the ODE:
\begin{equation}\label{eqn:hamilton_m}
\dot{x} = m_0(t,x)k\,,\quad \dot{k} = -\frac{1}{2} |k|^2 \nabla_x m_0(t,x)
\end{equation}
equipped with initial data:
\[
x(0,y,p) = y \,,\quad k(0,y,p) = p\,.
\]

The formal derivation above on obtaining asymptotic limit can be made rigorous. Indeed it is a direct consequence of the following classical result~\cite{LiPa:93,GeMaMaPo:1997}:
\begin{theorem}[Modification of Theorem 6.1 in~\cite{GeMaMaPo:1997}]\label{thm:gerard}
Consider the Cauchy problem
\begin{equation}
\ri \varepsilon \partial_t \ue - (P^{\varepsilon})^{\mathrm{W}}(t,x,\varepsilon D_x) \ue = 0, \quad \ue(0,x) = \ue_\mathrm{I}(x)\,,
\end{equation}
for $t>0$, $x\in \mathbb{R}^d$, where the Weyl operator $(P^{\varepsilon})^{\mathrm{W}}(t,x,\varepsilon D_x)$ is associated with the symbol $P^{\varepsilon}(t,x,k)$. Assume the symbol satisfies:
\begin{enumerate}[label=\roman*)]
  \item $\exists \sigma \in \mathbb{R}$, for all $\alpha$, $\beta\in\mathbb{N}_0$, there exists $C_{\alpha,\beta}>0$, such that for all $n,m\in \{1,\dots, d \}$, and for all $\varepsilon\in(0,\varepsilon_0]$, we have
  \begin{equation}
  \left| \frac{\partial^{\alpha+\beta}}{\partial x_n^{\alpha}\partial k_m^{\beta}}P^{\varepsilon}(t,x,k)\right|
  \leq C_{\alpha,\beta}(1+|k|)^{\sigma-\beta}\,,
  \end{equation}
  for all $(t,x,k)\in[0,T]\times\mathbb{R}^{2d}$\,,

  \item $(P^{\varepsilon})^\mathrm{W}(t,x,\varepsilon D_x)$ is essentially self-adjoint on $L^2(\mathbb{R}^d)$\,,

  \item $P^{\varepsilon}(t,x,k) = P^0(t,x,k) + \varepsilon Q^0(t,x,k) + o(\varepsilon)$ uniformly in $C([0,T];C_{\mathrm{loc}}^{\infty}(\mathbb{R}^{2d}))$\,.
\end{enumerate}
Then, if the initial data $\ue_\mathrm{I}$ is bounded in $L^2(\mathbb{R}^d)$, the Wigner transform $\we(t)\in \mathcal{S}'(\mathbb{R}^{2d})$ of $\ue(t)$, as $\varepsilon\rightarrow0$, converges uniformly on $[0,T]$ (in weak-$*$ sense) to the solution of
\begin{equation}
\begin{aligned}
&\partial_t W^0(t,x,k) + \nabla_k P^0(x,k) \cdot \nabla_x W^0(x,k) - \nabla_x P^0(x,k) \cdot \nabla_k W^0(t,x,k) = 0\,, \\
&W^0(0,x,k) = W_\mathrm{I}^0(x,k)\,,
\end{aligned}
\end{equation}
where the initial data $W_\mathrm{I}^0$ is the (weak-$*$) limit of Wigner transform of $\ue_\mathrm{I}$ as $\varepsilon\rightarrow0$. Furthermore, if the initial data $\ue_\mathrm{I}$ is $\varepsilon$-oscillatory, of which we refer to~\cite{GeMaMaPo:1997} for the definition, then the particle density
\[
\rho^{\varepsilon}(t,x) = |\ue(t,x)|^2 \to \rho^0(t,x) = \int_{\mathbb{R}^d} W^0(t,x,k) \rmd k \,,
\]
uniformly on $[0,T]$ as well.
\end{theorem}

Our theorem for VMSE is a direct corollary of the theorem above, applied on the equation with varying mass:
\begin{theorem}\label{prop:limit_wigner_m}
Suppose the mass $m_0$ satisfies~\cref{eqn:m0_H1} and elliptic condition~\cref{eqn:m0_H2}, then the Wigner transform $\we(t)$ of $\ue(t)$, the solution to VMSE~\cref{eqn:effect_mass_schr_m} with initial condition $\ue_\mathrm{I}\in \mathcal{S}(\mathbb{R}^d)$, converges uniformly on $[0,T]$ (in weak-$*$ sense) to the measure $W^0(t)\in\mathcal{S}'(\mathbb{R}^d)$ that solves:
\begin{equation}\label{eqn:liouville_thm}
\begin{aligned}
&\partial_t W^0(t,x,k) + m_0(t,x)k\cdot\nabla_x W^0(t,x,k) - \frac{|k|^2}{2}\nabla_x m_0(t,x) \cdot \nabla_k W^0(t,x,k) = 0\,, \\
&W^0(0,x,k) = W_\mathrm{I}^0(x,k)\,,
\end{aligned}
\end{equation}
with the initial data $W_\mathrm{I}^0$ being the (weak-$*$) limit of Wigner transform of $\ue_\mathrm{I}$.
\end{theorem}
\begin{proof}
This is a direct corollary of the previous theorem. To prove it amounts to deriving the symbol for the equation and justifying the assumptions on the symbols. To derive the symbol, we first compare VSME with
\begin{equation}\label{eqn:VMSE_pseudo}
\ri\varepsilon \partial_t \ue - (P^{\varepsilon})^\mathrm{W}(t,x,\varepsilon D_x)\ue = 0\,.
\end{equation}
Then
\begin{equation}
\begin{aligned}
((P^{\varepsilon})^\mathrm{W}(t,x,\varepsilon D_x)\ue)(x) &= -\frac{1}{2}\varepsilon^2 \nabla_x\cdot( m_0(t,x) \nabla_x \ue(t,x))\\
&= -\frac{1}{2}\varepsilon^2 \nabla_x m_0(t,x)\cdot \nabla_x \ue(t,x) - \frac{1}{2}\varepsilon^2 m_0(t,x) \Delta_x \ue(t,x)\,.
\end{aligned}
\end{equation}
Recall first the connection between the left symbol and differential operator:
\[
(P_\mathrm{l}^{\varepsilon}(t,x,\varepsilon D_x)\ue)(x) = \frac{1}{(2\pi)^d}\int_{\mathbb{R}^{2d}}P_\mathrm{l}^{\varepsilon}\left(t,x,\varepsilon k \right) f(y) e^{\ri (x-y)k}\rmd y \rmd k\,,
\]
we have the left symbol for $(P^{\varepsilon})^\mathrm{W}(t,x,\varepsilon D_x)$ to be
\begin{equation}
P_\mathrm{l}^{\varepsilon}(t,x,k) = -\frac{\varepsilon}{2} \ri k\cdot \nabla_x m_0(t,x)  + \frac{1}{2} m_0(t,x) |k|^2\,.
\end{equation}
We then recall the change of quantization formula~\cite{Ma:2002} that connects the left symbol and the Weyl symbol:
\begin{equation}
P^{\varepsilon}(t,x,k) \sim \sum_{\alpha\in\mathbb{N}^d}\frac{(-1)^{|\alpha|}\varepsilon^{|\alpha|}}{\ri^{|\alpha|}\alpha!}\partial_k^{\alpha}\partial_{\theta}^{\alpha}P_\mathrm{l}^{\varepsilon}\left(t,x+\frac{1}{2}\theta,k\middle)\right|_{\theta=0}\,,
\end{equation}
to finally obtain:
\begin{equation}
P^{\varepsilon}(t,x,k) = \frac{1}{2}m_0(t,x)|k|^2 + \frac{1}{8}\varepsilon^2 \Delta_x m_0(t,x)\,.
\end{equation}
This symbol apparently satisfies the three assumptions in~\cref{thm:gerard} given the condition of $m_0$ in~\cref{eqn:m0_H1}, with $P^0(t,x,k) = \frac{1}{2}m_0(t,x)|k|^2$. This concludes the proof.
\end{proof}

\begin{remark}
The results in~\cite{GeMaMaPo:1997} is vastly general, and its application in our setting can also be extended greatly. Indeed, the derivation for the VMSE with a potential term is also straightforward. Let the VMSE be:
\begin{equation}\label{eqn:VMSE_potential}
\ri\varepsilon \partial_t \ue(t,x) = - \frac{1}{2}\varepsilon^2 \nabla_x\cdot( m_0(t,x) \nabla_x \ue(t,x)) + V(t,x)\ue(t,x)\,,\quad \ue(0,x) = \ue_{\mathrm{I}}(x)\,.
\end{equation}
then the Weyl symbol, according to our derivation is then given by
\begin{equation}
P^{\varepsilon}(x,k) = \frac{1}{2}m_0(t,x)|k|^2 + \frac{1}{8}\varepsilon^2 \Delta_x m_0(t,x) + V(t,x)\,,
\end{equation}
with $P^0(t,x,k) = \frac{1}{2}m_0(t,x)|k|^2 + V(t,x)$. The $P^0$ is indeed the Hamiltonian in kinetic limit. With smoothness condition of both $m_0$ and $V$, the asymptotic limit becomes:
\begin{equation}\label{eqn:liouville_potential}
\partial_t W^0 + m_0k\cdot \nabla_x W^0 -\frac{|k|^2}{2}\nabla_xm_0\cdot\nabla_kW^0 - \nabla_xV\cdot\nabla_kW^0=0\,.
\end{equation}
\end{remark}

\section{Semi-classical limit for VMSE with random perturbation}\label{sec:Schr_ran}

We consider the VMSE where the effective mass involves random perturbation, namely:
\begin{equation}\label{eqn:randpwe}
\ri\varepsilon \partial_t \ue + \frac{1}{2}\varepsilon^2 \nabla_x\cdot( m^{\varepsilon}(t,x)\nabla_x \ue) = 0,
\end{equation}
where the effective mass is
\begin{equation}\label{eqn:mass_rand}
m^{\varepsilon}(t,x) = m_0(t,x) + \sqrt{\varepsilon}m_1(\teps,\xeps)\,.
\end{equation}
While the leading order $m_0$ is assumed to be deterministic and smooth, we  allow the random perturbation $m_1(t,x)$ to present small scales at $\varepsilon$. While the scale for the perturbation is at the order of $\sqrt{\epsilon}$, the oscillation is at the order of $\epsilon$ for both $t$ and $x$. Furthermore we assume $m_1$ is mean-zero and stationary in both $t$ and $x$ with the correlation function $R(t,x)$:
\begin{equation}\label{eqn:correlation}
R(t,x) = \mathbb{E}[m_1(s,z)m_1(t+s,x+z)] \quad \forall x,z\in\bbr^d \text{ and }t,s\in\bbr\,.
\end{equation}

Taking the Fourier transform of the function in both time and space, one has:
\begin{equation}
\hat{R}(\omega,p) = \int_{\bbr^{d+1}} e^{-\ri\omega s - \ri p z}R(s,z)\rmd s \rmd z\,,
\end{equation}
then it is straightforward to show:
\begin{equation}
\mathbb{E}[\frmrandalter \hat{m}_1(\omega,q)] = (2\pi)^d e^{-\ri\omega\tau}\hat{R}(\omega,p) \delta(p+q),
\end{equation}
and
\[
\hat{R}(-\omega,p) = \hat{R}(\omega,p)\,,\quad\text{and}\quad \hat{R}(\omega,-p) = \hat{R}(\omega,p)\,.
\]

We dedicate this section to the derivation of the classical limit of the equation above. We will show that
\begin{Result}\label{thm:W_rand}
In the zero limit of $\varepsilon$, the Wigner transform of $\ue(t)$, which is the solution to the VMSE~\cref{eqn:randpwe} with varying random mass~\cref{eqn:mass_rand}, solves the radiative transfer equation:
\begin{multline}\label{eqn:RTE}
\partial_t W^{0} + m_0k\cdot\nabla_x W^{0} - \frac{k^2}{2}\nabla_x m_0 \cdot \nabla_k W^{0} \\
= \frac{1}{(2\pi)^d} \int_{\bbr^d} \frac{1}{4}(p\cdot k)^2 \hat{R}\left(\frac{m_0}{2}(p^2-k^2),p-k\right)[W^{0}(p)-W^{0}(k)]\mathrm{d}p\,.
\end{multline}
\end{Result}

\begin{proof}[Derivation]
In view of \cref{eqn:wignereqn} in~\cref{lem:W}, noting that $m_0\to m_0+\sqrt{\varepsilon}m_1$, the Wigner equation~\cref{eqn:randpwe} is transformed to
\begin{equation}\label{eqn:wignerrand}
\partial_t \we + \frac{1}{\varepsilon}\mathcal{Q}^{\varepsilon}_1 \we + \mathcal{Q}^{\varepsilon}_2 \we + \frac{1}{\sqrt{\varepsilon}}\mathcal{P}_1 \we + \sqrt{\varepsilon}\mathcal{P}_2 \we = \varepsilon\mathcal{Q}^{\varepsilon}_3 \we + \varepsilon^{3/2}\mathcal{P}_3 \we\,,
\end{equation}
where the operators $\mathcal{Q}^{\varepsilon}_i$'s are defined in \cref{eqn:q1}-\cref{eqn:q3}, and $\mathcal{P}_i$'s are their counterparts defined by $m_1$:

\begin{equation}
\begin{aligned}
\mathcal{P}_1 \we(t,x,k) =& \frac{|k|^2}{2}\int_{\bbr^d} \frac{e^{\ri p\xi}}{(2\pi)^d}\tilde{m}_1(\tau,p) \ri\left[\we\left(k- \half p\right)-\we\left(k+\half p\right)\right]\rmd p\\
&-\frac{1}{8}\int_{\bbr^d} \frac{e^{\ri p\xi}}{(2\pi)^d}\tilde{m}_1(\tau,p) \ri|p|^2\left[\we\left(k-\half p\right)-\we\left(k+\half p\right)\right]\rmd p\,,
\end{aligned}
\end{equation}
\begin{equation}
\mathcal{P}_2 \we(t,x,k) = \frac{k}{2}\cdot\int_{\bbr^d} \frac{e^{\ri p\xi}}{(2\pi)^d}\tilde{m}_1(\tau,p) \left[\nabla_x\we\left(k- \half p\right)+\nabla_x\we\left(k+ \half p\right)\right]\rmd p\,,
\end{equation}
and
\begin{equation}
\begin{aligned}
\mathcal{P}_3 \we(t,x,k) =& \frac{1}{8}\int_{\bbr^d} \frac{e^{\ri p\xi}}{(2\pi)^d}\tilde{m}_1(\tau,p) \ri\left[\Delta_x\we\left(k-\half p\right)-\Delta_x\we\left(k+\half p\right)\right]\rmd p\,,
\end{aligned}
\end{equation}
where we use the fast variables
\[
\tau = \frac{t}{\varepsilon}\,,\quad\xi = \frac{x}{\varepsilon}\,.
\]
Explicitly spelling out the fast variables in $\we$, one has:
\[
\we(t,x,k)\to \we(t,\tau,x,\xi,k)\,,\quad \nabla_x\to\nabla_x+\frac{1}{\varepsilon}\nabla_\xi\,,\quad \partial_t\to\partial_t+\frac{1}{\varepsilon}\partial_\tau\,,
\]
and thus the leading orders in~\cref{eqn:wignerrand} become:
\begin{equation*}
\begin{aligned}
\frac{1}{\varepsilon}\mathcal{Q}^{\varepsilon}_1 \we &= -\frac{|k|^2}{2} \nabla_x m_0 \cdot \nabla_k \we + \calo(\varepsilon^2),\\
\mathcal{Q}^{\varepsilon}_2 \we &= \frac{1}{\varepsilon}m_0 k\cdot\nabla_{\xi} \we + m_0 k\cdot\nabla_{x}\we + \calo(\varepsilon),\\
\varepsilon\mathcal{Q}^{\varepsilon}_3 \we &= -\frac{1}{8}\nabla_x m_0\cdot\nabla_k(\Delta_{\xi}\we) + \calo(\varepsilon),
\end{aligned}
\end{equation*}
and
\begin{equation*}
\begin{aligned}
\sqrt{\varepsilon}\mathcal{P}_2 \we =& \frac{1}{\sqrt{\varepsilon}}\mathcal{P}_2\left(\frac{\partial}{\partial \xi}\right) \we + \calo(\sqrt{\varepsilon}),\\
\varepsilon^{3/2}\mathcal{P}_3 \we =& \frac{1}{\sqrt{\varepsilon}}\mathcal{P}_3 \left(\frac{\partial}{\partial \xi}\right) \we + \calo(\sqrt{\varepsilon})\,.
\end{aligned}
\end{equation*}

To perform the asymptotic expansion of the equation, we first write the ansatz
\begin{equation}\label{eqn:multiscaleexpand}
\we = W^{(0)}+ \sqrt{\varepsilon} W^{(1)} + \varepsilon W^{(2)}+\cdots\,.
\end{equation}
By plugging the expansion above into~\cref{eqn:wignerrand}, we have, at the order of $\calo(1/\varepsilon)$:
\begin{equation}\label{eqn:eqnW0}
\partial_{\tau}W^{(0)} + m_0k\cdot\nabla_{\xi}W^{(0)}=0\,,
\end{equation}
which suggests $W^{(0)}$ having no dependence on $\tau$ and $\xi$, the fast variables. The next order is $\calo(1/\sqrt{\varepsilon})$, and the equation writes:
\begin{equation}\label{eqn:eqnW1}
\begin{aligned}
&\partial_{\tau}W^{(1)} + m_0k\cdot\nabla_{\xi}W^{(1)}\\
=& \frac{1}{\ri}\int_{\bbr^d} \frac{e^{\ri p\xi}}{(2\pi)^d}\frmrandalter \left(\frac{|k|^2}{2}-\frac{|p|^2}{8}\right)\left[W^{(0)}\left(k- \half p\right)-W^{(0)}\left(k+\half p\right)\right]\rmd p\,.
\end{aligned}
\end{equation}
Since the only $\tau$ dependence on the right hand side is in $\frmrandalter$, the equation can be solved explicitly using the Fourier transform:
\begin{equation}\label{eqn:W1}
\begin{aligned}
&\ri(2\pi)^{d+1}W^{(1)}(t,\tau,x,\xi,k)\\
&= \int \frac{e^{\ri p\xi+\ri\omega\tau}\stfrm(4|k|^2 - |p|^2)}{8(\ri\omega + \ri m_0k\cdot p + \theta)}\left[W^{(0)}\left(k- \frac{p}{2}\right)-W^{(0)}\left(k+ \frac{p}{2}\right)\right]\rmd p \rmd \omega\,,
\end{aligned}
\end{equation}
where $\theta$ is a regularization parameter, to be sent to $0$ in the end, and $\stfrm$ is the space-time Fourier transform of $m_1$:
\begin{equation}
\stfrm = \int_{\bbr} e^{-\ri\tau \omega} \frmrandalter \rmd \tau\,.
\end{equation}

The following order is $\calo(1)$ and is the order we use to close:
\begin{equation}\label{eqn:order1}
\begin{aligned}
&\partial_t W^{(0)} + m_0k\cdot\nabla_x W^{(0)} -\frac{k^2}{2}\nabla_x m_0 \cdot \nabla_k W^{(0)} \\
&+\partial_{\tau} W^{(2)} + m_0k\cdot\nabla_{\xi}W^{(2)}+ \mathcal{P}_1 W^{(1)} + \mathcal{P}_2\left( \frac{\partial}{\partial \xi}\right) W^{(1)}
= \mathcal{P}_3 \left( \frac{\partial}{\partial \xi}\right) W^{(1)}\,.
\end{aligned}
\end{equation}
Noticing
\begin{equation}
\mathbb{E}[\partial_{\tau} W^{(2)} + m_0k\cdot\nabla_{\xi}W^{(2)}] = 0\,,
\end{equation}
we eliminate the dependence on $W^{(2)}$ in the equation and arrive at:
\begin{equation}\label{eqn:W0_limit}
\begin{aligned}
&\partial_t W^{(0)} + m_0k\cdot\nabla_x W^{(0)} -\frac{k^2}{2}\nabla_x m_0 \cdot \nabla_k W^{(0)} \\
&= -\mathbb{E}[\mathcal{P}_1 W^{(1)}] - \mathbb{E}\left[\mathcal{P}_2\left( \frac{\partial}{\partial \xi}\right) W^{(1)}\right] + \mathbb{E}\left[\mathcal{P}_3\left( \frac{\partial}{\partial \xi}\right) W^{(1)}\right]\,.
\end{aligned}
\end{equation}

Getting the simplified version of the equation and showing the radiative transfer equation limit amounts to analyzing the terms on the right respectively. Since they are quite similar, we only present the calculation of terms corresponding to $\mathcal{P}_3$ below. It essentially comes from plugging in $W^{(1)}$ formula in~\cref{eqn:W1} into it. For example, $\mathbb{E}\left[\mathcal{P}_3\left( \frac{\partial}{\partial \xi}\right) W^{(1)}\right]$ can be simplified as:
\begin{equation}\label{eqn:W0_average_1}
\begin{aligned}
 &\quad \,\, \mathbb{E}\left[\mathcal{P}_3\left( \frac{\partial}{\partial \xi}\right) W^{(1)}\right] \\
&= \frac{1}{8}\mathbb{E}\left[\int_{\bbr^d} \frac{e^{\ri p\xi}}{(2\pi)^d}\frmrandalter \ri\left[\Delta_{\xi}W^{(1)}\left(k-\half p\right)-\Delta_{\xi}W^{(1)}\left(k+\half p\right)\right]\rmd p \right]\\
&=\frac{1}{8}\frac{1}{(2\pi)^{d+1}}\int_{\bbr^{d+1}} -|p-k|^2(p\cdot k) \hat{R}(\omega,p-k) \frac{\theta}{(\omega-\frac{m_0}{2}(p^2-k^2))^2+\theta^2}\\
& \qquad \qquad \qquad \qquad \qquad \qquad \qquad \qquad \qquad \qquad \qquad \qquad \quad [W^{(0)}(k)-W^{(0)}(p)]\rmd \omega \rmd p\\
&\xrightarrow{\theta\to 0} \frac{1}{(2\pi)^{d}}\int_{\bbr^{d+1}} -\frac{|p-k|^2}{8}\half(p\cdot k) \hat{R}\left(\frac{m_0}{2}(p^2-k^2),p-k\right)[W^{(0)}(k)-W^{(0)}(p)] \rmd p\,,
\end{aligned}
\end{equation}
where we used
\[
\lim_{\theta\to 0}\frac{\theta}{x^2+\theta^2} =\pi \delta(x)\,.
\]

Other terms in~\cref{eqn:W0_limit} can be similarly treated. The term corresponding to operator $\mathcal{P}_1$ becomes:
\begin{equation}
\begin{aligned}
&-\mathbb{E}[\mathcal{P}_1 W^{(1)}]\\
=&-\frac{|k|^2}{2}\mathbb{E}\bigg[\int_{\bbr^d} \frac{e^{\ri p\xi}}{(2\pi)^d}\frmrandalter \ri\left[W^{(1)}\left(k- \half p\right)-W^{(1)}\left(k+\half p\right)\right]\rmd p\bigg]\\
 &+\frac{1}{8}\mathbb{E}\left[\int_{\bbr^d} \frac{e^{\ri p\xi}}{(2\pi)^d}\frmrandalter \ri|p|^2\left[W^{(1)}\left(k-\half p\right)-W^{(1)}\left(k+\half p\right)\right]\rmd p\right]\\
=&-\frac{|k|^2}{2}\frac{1}{(2\pi)^{d}}\int_{\bbr^{d+1}} \half(p\cdot k) \hat{R}\left(\frac{m_0}{2}(p^2-k^2),p-k\right) [W^{(0)}(k)-W^{(0)}(p)] \rmd p\\
& + \frac{1}{8}\frac{1}{(2\pi)^{d}}\int_{\bbr^{d+1}} |p-k|^2\half(p\cdot k) \hat{R}\left(\frac{m_0}{2}(p^2-k^2),p-k\right) [W^{(0)}(k)-W^{(0)}(p)]\rmd p\,,
\end{aligned}
\end{equation}
and the term corresponding to operator $\mathcal{P}_2$ becomes:
\begin{equation}\label{eqn:W0_average_end}
\begin{aligned}
&- \mathbb{E}\left[\mathcal{P}_2\left( \frac{\partial}{\partial \xi}\right) W^{(1)}\right]\\
=&-\frac{k}{2}\cdot\mathbb{E}\left[\int_{\bbr^d} \frac{e^{\ri p\xi}}{(2\pi)^d}\frmrandalter \left[\nabla_{\xi}W^{(1)}\left(k- \half p\right)+\nabla_{\xi}W^{(1)}\left(k+ \half p\right)\right]\rmd p \right]\\
=& -\frac{k}{2}\cdot\frac{1}{(2\pi)^{d}}\int_{\bbr^{d+1}} (p-k)\half(p\cdot k) \hat{R}\left(\frac{m_0}{2}(p^2-k^2),p-k\right) [W^{(0)}(k)-W^{(0)}(p)] \rmd p\,.
\end{aligned}
\end{equation}
Inserting~\cref{eqn:W0_average_1}-\cref{eqn:W0_average_end} into~\cref{eqn:W0_limit} and simplify, we have the leading order asymptotic limit of~\cref{eqn:wignerrand}, concluding the proposition.
\end{proof}
\begin{remark}
With the formal derivation at hand,~\cref{thm:W_rand} can be potentially made rigorous with the perturbed test function technique as shown in~\cite{PoVa:2003,BaPaRy:2002}. This is, however, beyond the focus of this article. According to the conditions listed in~\cite{BaPaRy:2002} for the Schr\"odinger equation with random potential, to have the weak-$\ast$ convergence in $L^\infty([0,T];\mathcal{S}'(\mathbb{R}^{2d}))$ in our case here, it is expected that the Fourier transform of $m_1$ is a Markov process on the space of measures with bounded total variation and uniformly bounded support. Other conditions may be needed as well. We leave the rigorous justification to the future research.
\end{remark}

\section{Numerical result}\label{sec:numerics}
As a proof of concept, we provide some numerical evidences for~\cref{prop:limit_wigner_m} and~\cref{thm:W_rand}, the two results with $m^{\varepsilon}$ being completely deterministic and $m^{\varepsilon}$ having random fluctuations.

\subsection{Illustration of~\cref{prop:limit_wigner_m}}
We present numerical evidence for~\cref{prop:limit_wigner_m} in this subsection.
\subsubsection{Numerical setup}
According to the theorem, the Wigner transform of solution to VMSE satisfies, in the leading order, the Liouville equation.

To compare the wave functions of VMSE and its Wigner limit, we evaluate the following two macroscopic quantities:
\begin{equation}
\begin{aligned}
&\rho^0(t,x) = \int W^0(t,x,k)\rmd k\,, \quad \rho^{\varepsilon}(t,x) = |\ue(t,x)|^2\,, \\
&J^0(t,x) = m_0(t,x) \int k W^0(t,x,k)\rmd k\,, \quad J^{\varepsilon}(t,x) = \varepsilon\text{Im}\left(m_0(t,x)\overline{\ue(t,x)}\nabla_x\ue(t,x)\right)\,.
\end{aligned}
\end{equation}

As a computational setup, we set $\Omega = [0,L]\times[0,T]$, and choose the spatial mesh size $\Delta x = L / M$ with $M$ being an even
integer. The time step is denoted by $\Delta t$. The spatial and temporal grid points are denoted by $x_j = j\Delta x, j = 0,1,\cdots,M$, and $t_n = n\Delta t, n = 0,1,2,\cdots$. The initial data for Schr\"odinger equation has a Gaussian form:
\begin{equation}\label{eqn:initialSchr_deterministic}
\ue_I(x) = \exp\left(-A(x-x_0)^2 + \frac{\ri}{\varepsilon}p_0x\right)\,.
\end{equation}
The periodic boundary conditions are imposed
\begin{equation}
\ue(t,0) = \ue(t,L), \quad \partial_x\ue(t,0) = \partial_x\ue(t,L)\,.
\end{equation}
In computation we will set $L$ to be large enough and the periodic boundary condition plays minimum role.

Correspondingly, the transport equation~\cref{eqn:liouville_thm} has initial data:
\begin{equation} \label{eqn:initialRTE_deterministic}
W^0_I(x,k) = \exp(-2A(x-x_0)^2)\delta(k-p_0)\,.
\end{equation}

For Schr\"odinger equation with potential term~\cref{eqn:VMSE_potential}, we use standard Finite Difference method with the discretization resolved, namely $\Delta x = O(\varepsilon)$ and $\Delta t = o(\varepsilon)$. The Crank-Nicolson is applied in time, and spectral method is applied to treat spacial discretization~\cite{BaJiMa:02}, namely: let $U^{\varepsilon,n}_j$ be the approximation of $u^{\varepsilon}(x_j,t_n)$, then
\begin{equation}
\begin{aligned}
\frac{U^{\varepsilon,n+1}_j-U^{\varepsilon,n}_j}{\Delta t} = &\frac{\ri\varepsilon}{4} \big( D^\mathrm{s}_x(m_0^{n+1/2}D^\mathrm{s}_x U^{\varepsilon,n+1})|_{x_j} + D^\mathrm{s}_x(m_0^{n+1/2}D^\mathrm{s}_x U^{\varepsilon,n})|_{x_j}\big) \\
& - \frac{\ri}{\varepsilon} V_j^{n+1/2}\frac{U_j^{\varepsilon,n+1}+U_j^{\varepsilon,n}}{2}
\end{aligned}
\end{equation}
with
\[
U^{\varepsilon,n+1}_0 = U^{\varepsilon,n+1}_M,\quad U^{\varepsilon,n+1}_1 = U^{\varepsilon,n+1}_{M+1},\quad U^{\varepsilon,0}_j = u^{\varepsilon}_0(x_j),\quad \forall j\,.
\]
Here the super-index $n$ is for time, while the lower-index $j$ is for spatial grid point. We sample $M$ grid points in the domain $[0,L]$. The differential-operator $D^\mathrm{s}_x$ is computed through spectral method:
\begin{equation}
D^\mathrm{s}_x U|_{x=x_j} = \frac{1}{M}\sum_{l = -M/2}^{M/2-1} \ri\mu_l \hat{U}_l e^{\ri \mu_lx_j}\,,
\end{equation}
with
\begin{equation}
\hat{U}_l = \sum_{j = 0}^{M-1} U_j e^{-\ri\mu_l x_j},\quad l = -\frac{M}{2},\dots,\frac{M}{2}-1\,.
\end{equation}

To compute the deterministic Liouville equation~\cref{eqn:liouville_potential}, we use the particle method, that is to compute a large number of ODE systems:
\begin{equation}\label{eqn:hamilton_t}
\begin{cases}
\dot{x} = -km_0(T-t,x)\\
\dot{k} = \frac{|k|^2}{2}\partial_x m_0(T-t,x) + V(T-t,x)\,,
\end{cases}
0\leq t\leq T\,, x(0) = y, k(0) = p\,.
\end{equation}
The final solution is $W^0(T,y,p) = W_I(x(T),k(T))$. The equations~\cref{eqn:hamilton_t} for trajectory can be efficiently solved with typical ODE solvers. See also~\cite{ToEn:2004,ZaTo:2010} for the discussions of the regularized delta function.

\subsubsection{Numerical examples}
We have two examples below. In the first examples, we set $L=1.25$ and $T=0.5$. For the initial data~\cref{eqn:initialSchr_deterministic} and~\cref{eqn:initialRTE_deterministic}, we take $A = 2^7$, $p_0 = 1$ and $x_0 = 0.25$. In the first example, we set
\begin{equation}\label{eqn:m0_x_t}
m_0(t,x) = (1+0.2\sin(2\pi x))(1+0.2\cos(2\pi t))\,, \quad V_0(t,x) = 0\,.
\end{equation}
To compute Liouville equation, we set the spatial size $\Delta x = 2^{-10}$ and the frequency step $\Delta k =  2^{-10}$. The system~\cref{eqn:hamilton_t} are computed using MATLAB adaptive ODE solver with a prescribed error accuracy $10^{-8}$. To compute VMSE, we set $\varepsilon = 2^{-n}$ and we use the discretization:
\begin{equation}
\Delta t = 2^{-1.2n-3}, \quad \Delta x = 2^{-n-2}\,,
\end{equation}
that resolves the scales.

In~\cref{fig:mass_transport_det} we show the solution to the transport equation~\cref{eqn:liouville_thm} at different time snapshots. The results are presented both on the phase space, and on the physical domain, where we plot the density and the flux term. We then compare the Schr\"odinger equation solution and the limiting Liouville equation solution. In~\cref{fig:Liouville_vs_Schr}-\cref{fig:Liouville_vs_Schr_J} we present both the comparison of the density $\rho^0$ and $\rho^\varepsilon$ with different $\varepsilon$, and the comparison of the flux $J^0$ and $J^\varepsilon$ with different $\varepsilon$. The convergence rate is also shown in~\cref{fig:Liouville_conv} and~\cref{fig:Liouville_conv_J}, with the error quantified according to the following:
\begin{equation}
\text{Err}^{\varepsilon}_\rho = \int_{\bbr} |\rho^0 - \rho^{\varepsilon}| \rmd x\,, \quad
\text{Err}^{\varepsilon}_J = \int_{\bbr} |J^0 - J^{\varepsilon}| \rmd x
\end{equation}
where we recall:
\begin{equation*}
J^0(t,x) = m_0(t,x) \int k W^0(t,x,k)\rmd k\,, \quad
J^{\varepsilon}(t,x) = \varepsilon\text{Im}\left(m_0(t,x)\overline{\ue(t,x)}\nabla_x\ue(t,x)\right)\,.
\end{equation*}
According to the numerical solution, the errors decay at a rate of $O(\varepsilon^2)$.

\begin{figure}[tbhp]
  \centering
  \subfloat[$t = 0.05$]{
  \includegraphics[width=0.45\textwidth]{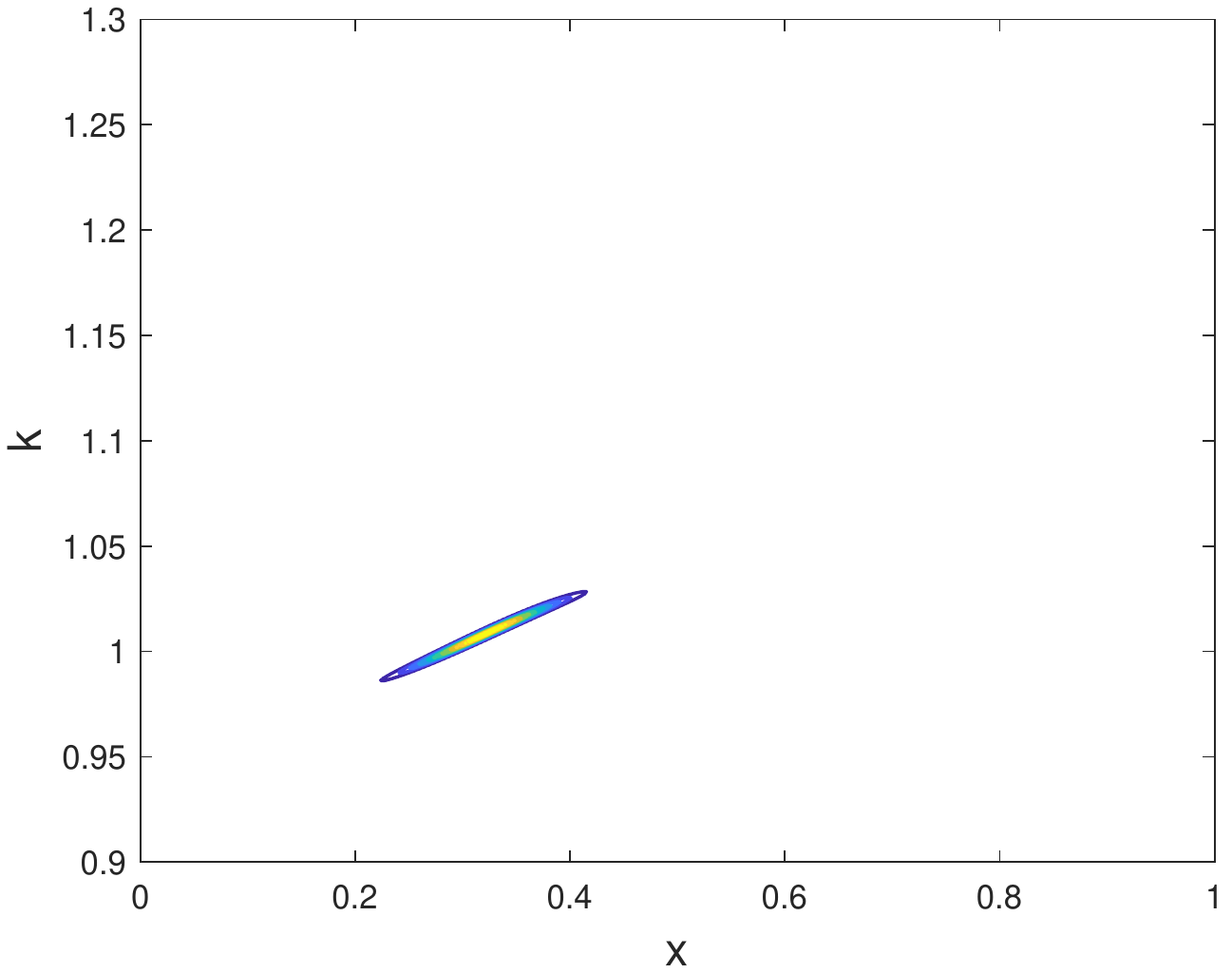}
  \includegraphics[width=0.45\textwidth]{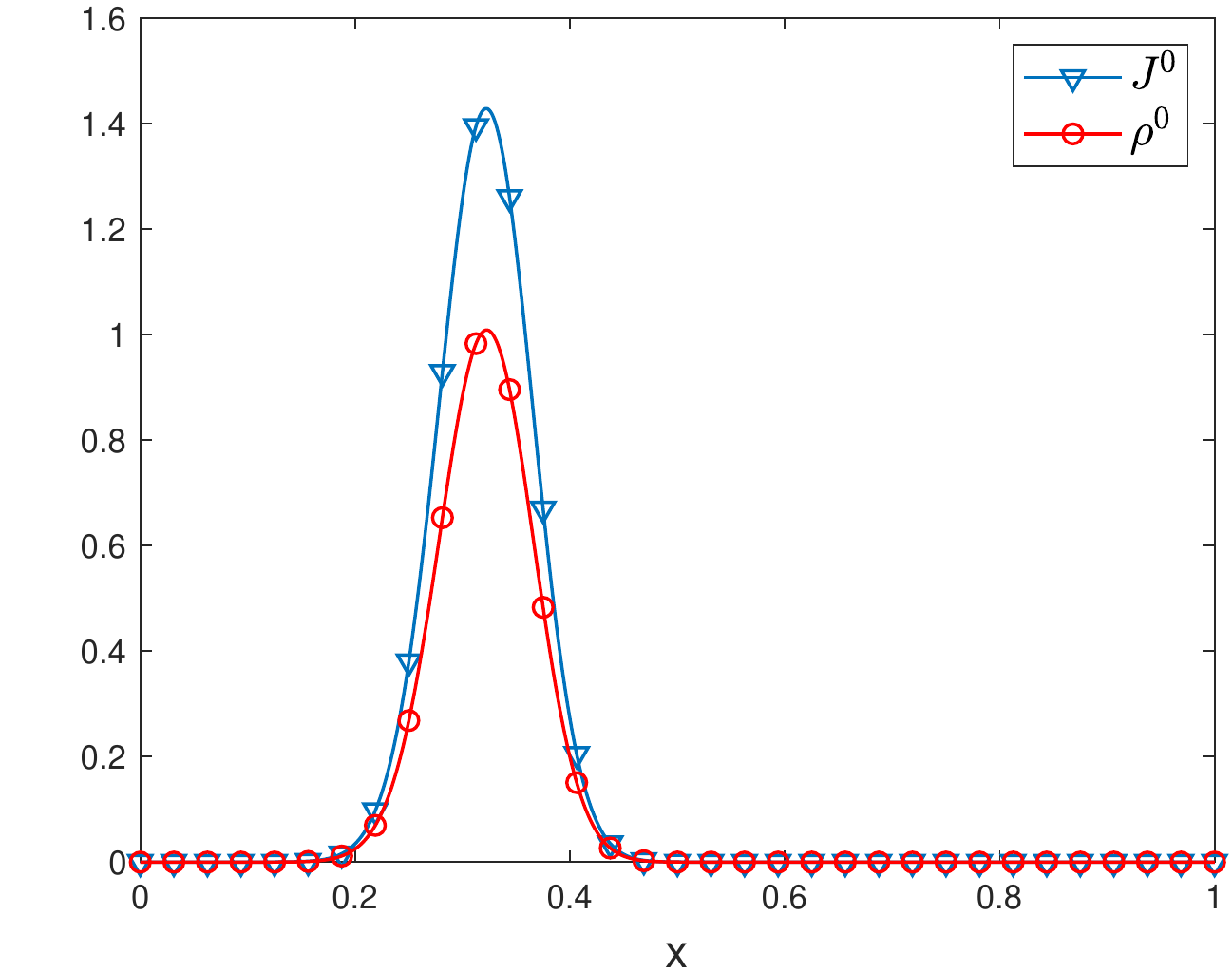}
  }

  \subfloat[$t = 0.15$]{
  \includegraphics[width=0.45\textwidth]{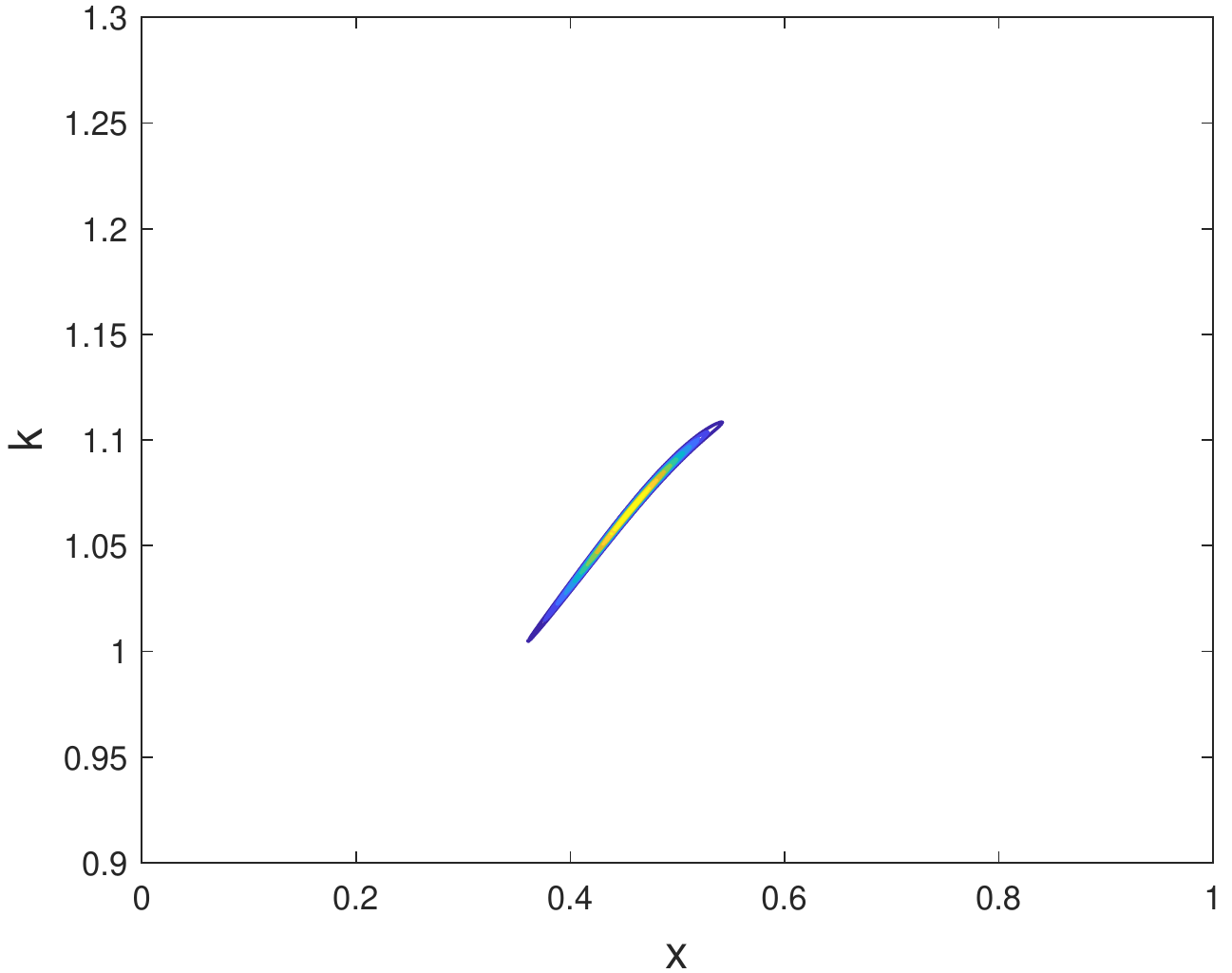}
  \includegraphics[width=0.45\textwidth]{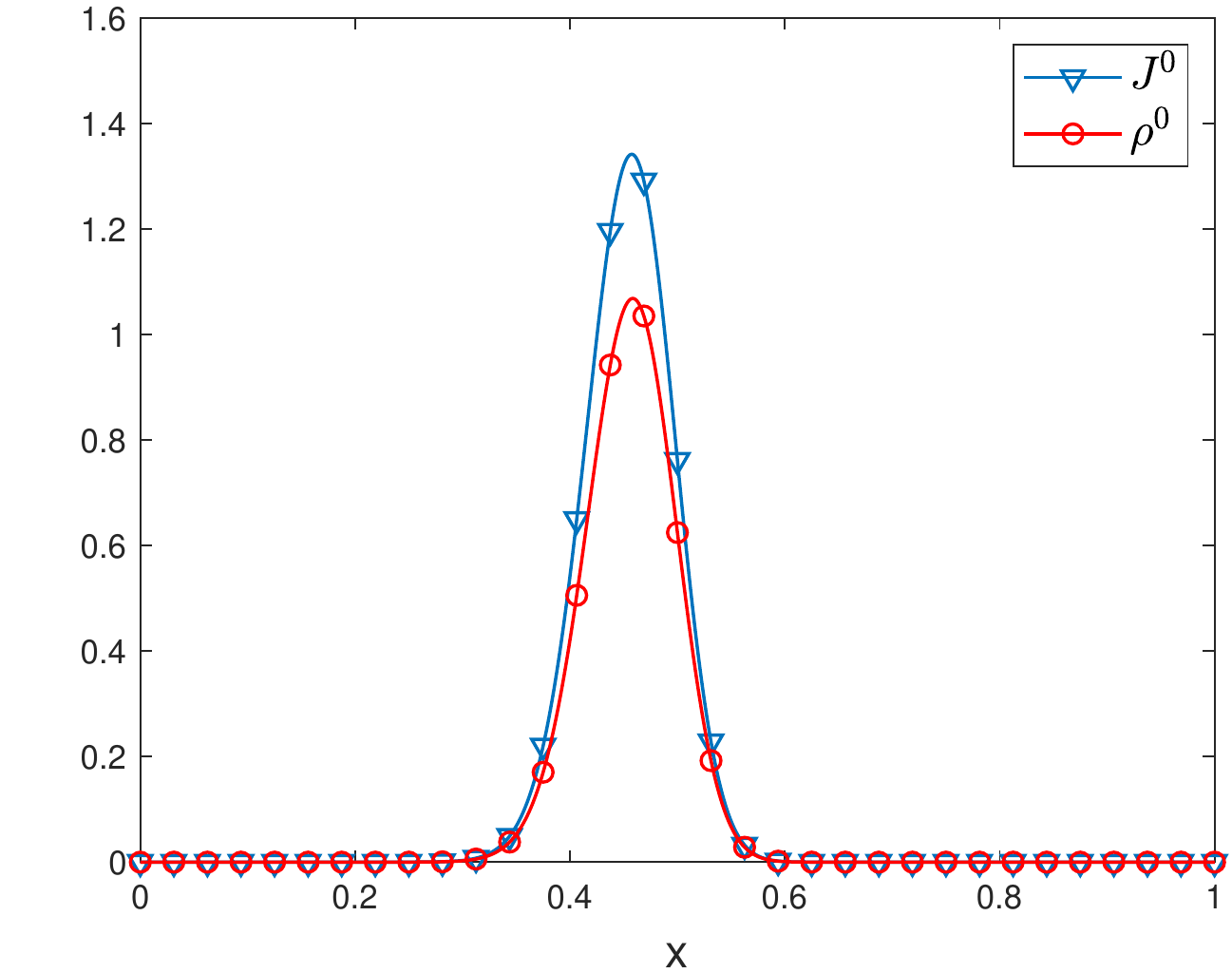}
  }

  \subfloat[$t = 0.4$]{
  \includegraphics[width=0.45\textwidth]{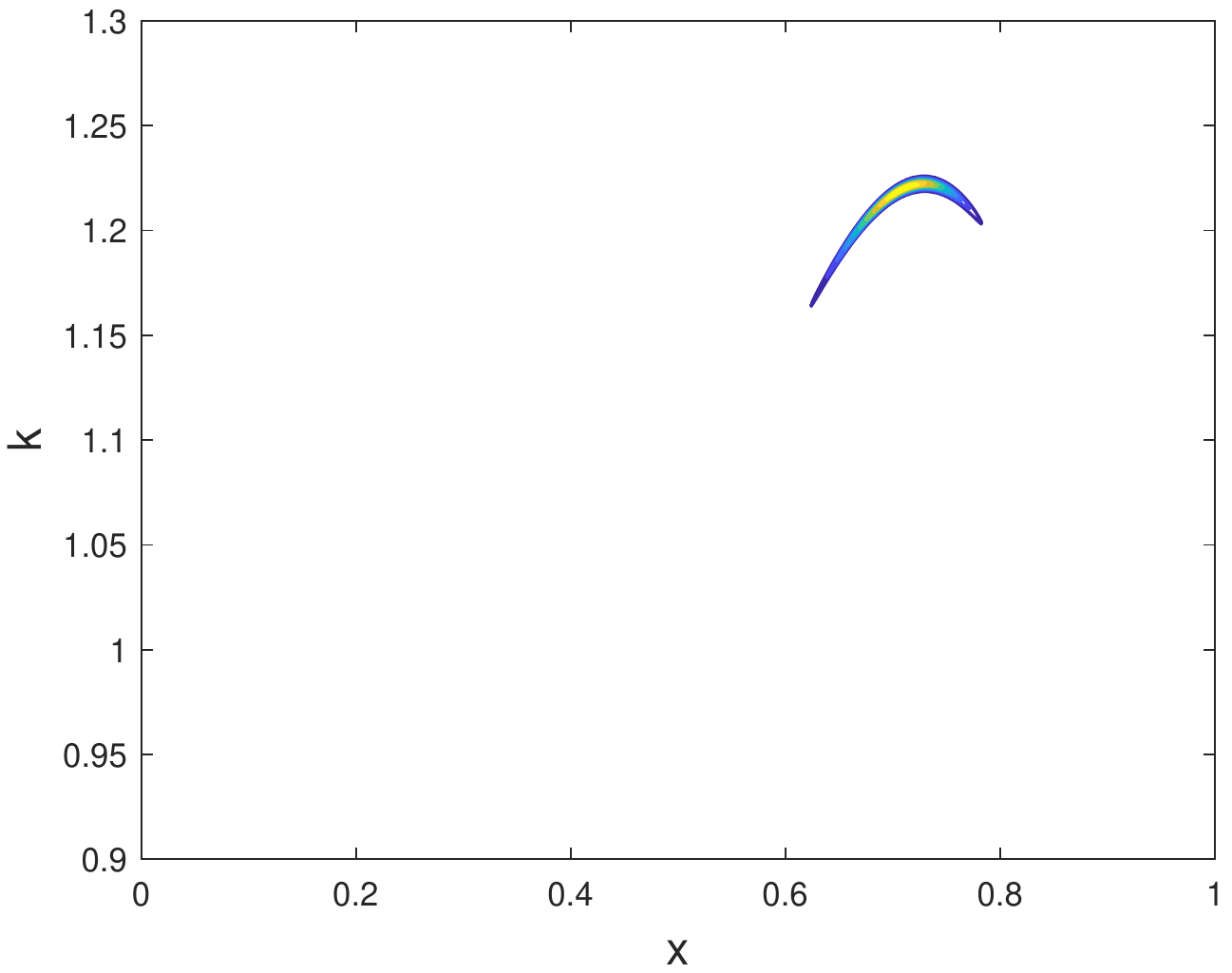}
  \includegraphics[width=0.45\textwidth]{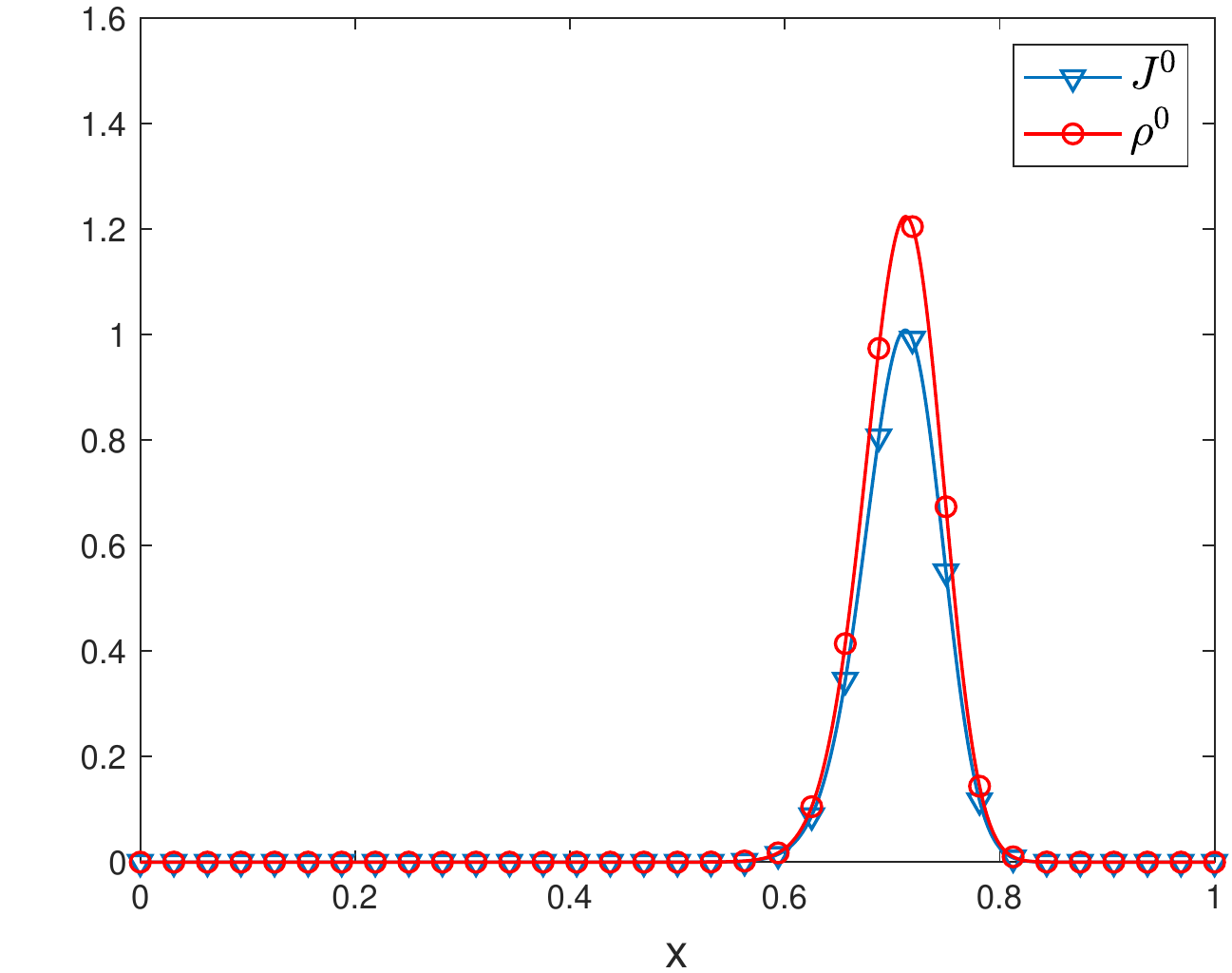}
  }
  \caption{The left column shows the contour of $W^0$ in phase space and the right column shows the particle density $\rho^0 = \int W^0 \rmd k$ and current density $J^0 = m_0 \int  k W^0 \rmd k$. The mass $m_0$~\cref{eqn:m0_x_t} is $t$-dependent.}
  \label{fig:mass_transport_det}
\end{figure}

\begin{figure}[tbhp]
  \centering
  \subfloat[]{\label{fig:Liouville_vs_Schr}\includegraphics[width=0.45\textwidth]{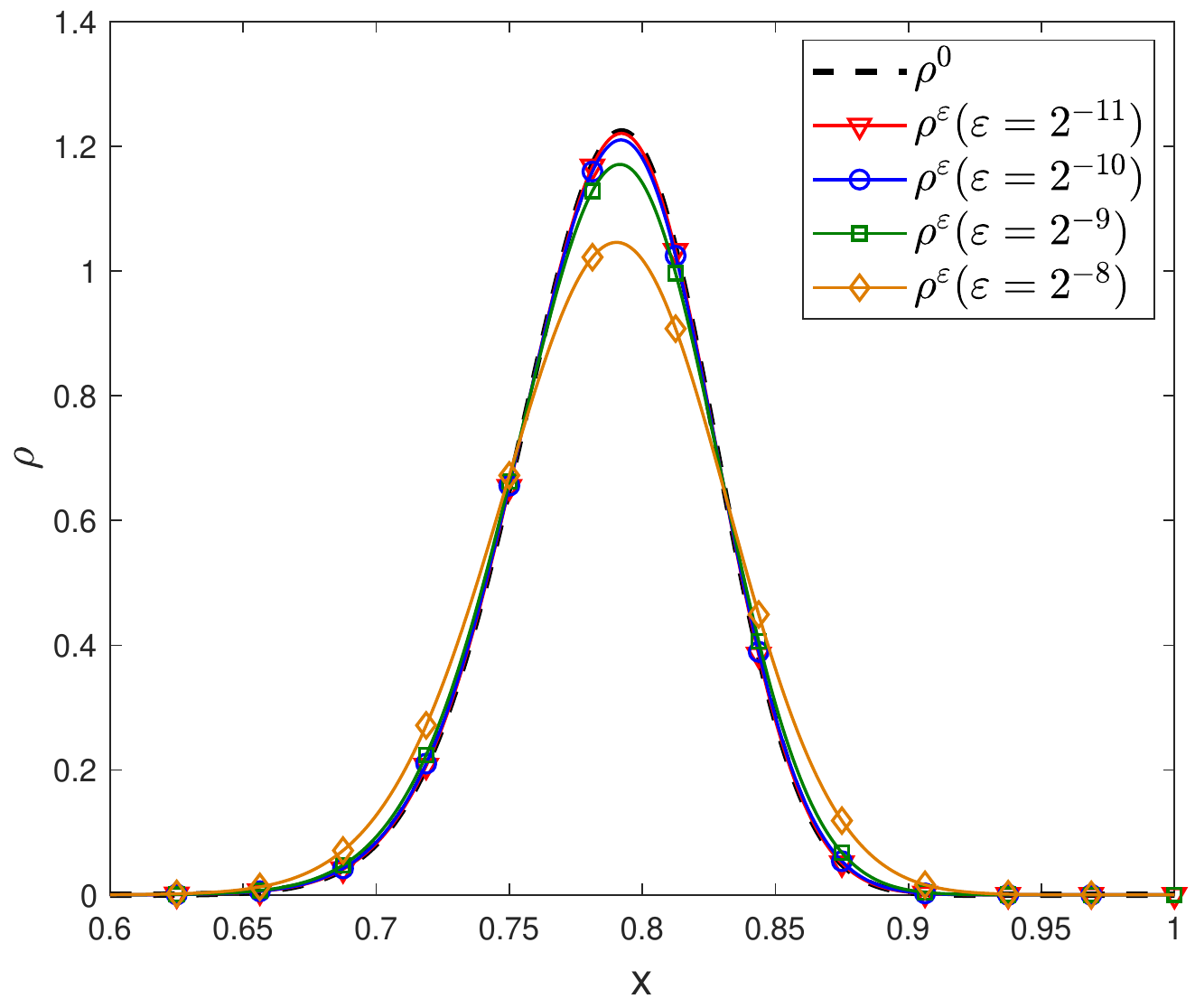}}
  \subfloat[]{\label{fig:Liouville_vs_Schr_J}\includegraphics[width=0.45\textwidth]{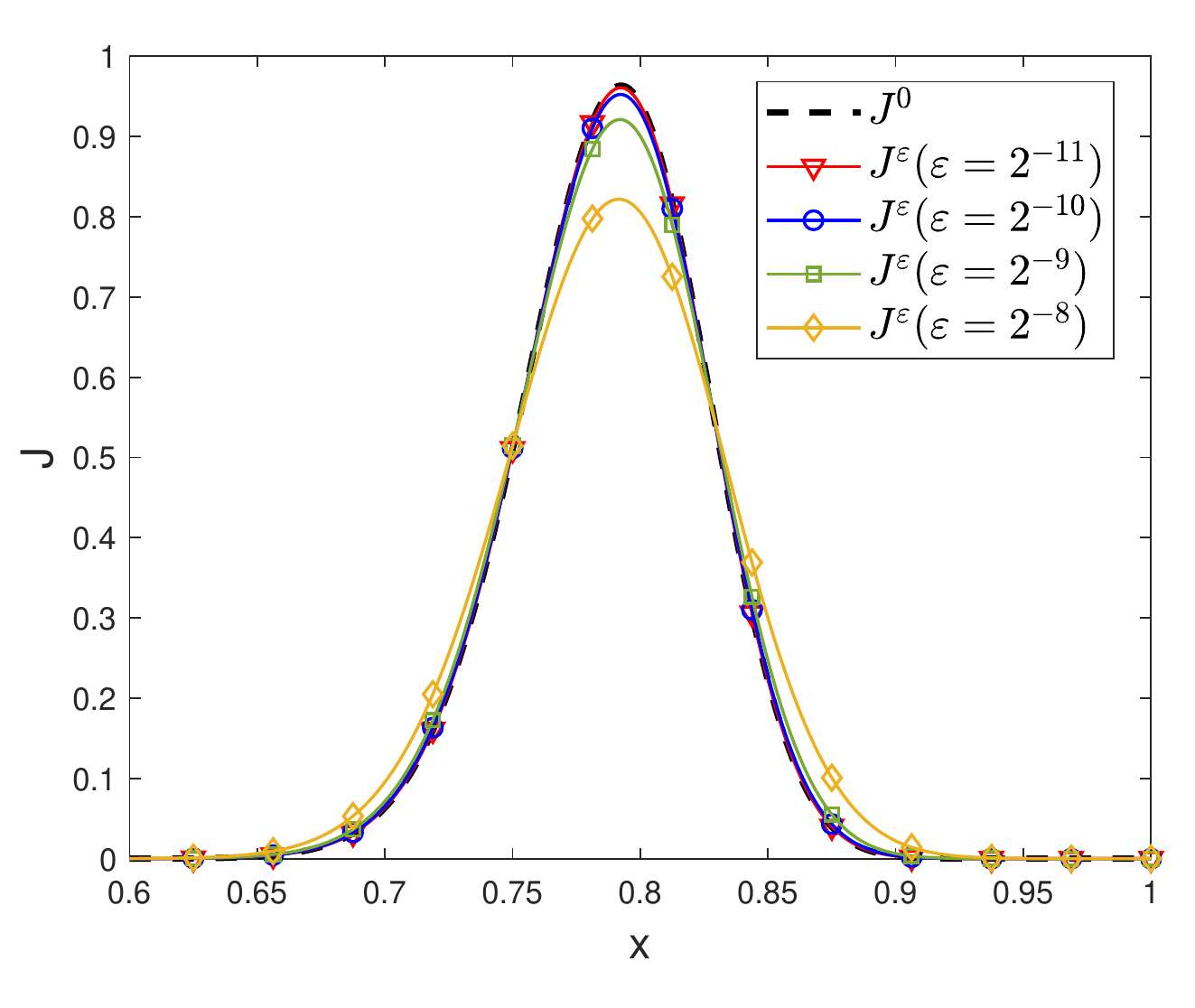}}

  \subfloat[]{\label{fig:Liouville_conv}\includegraphics[width=0.45\textwidth]{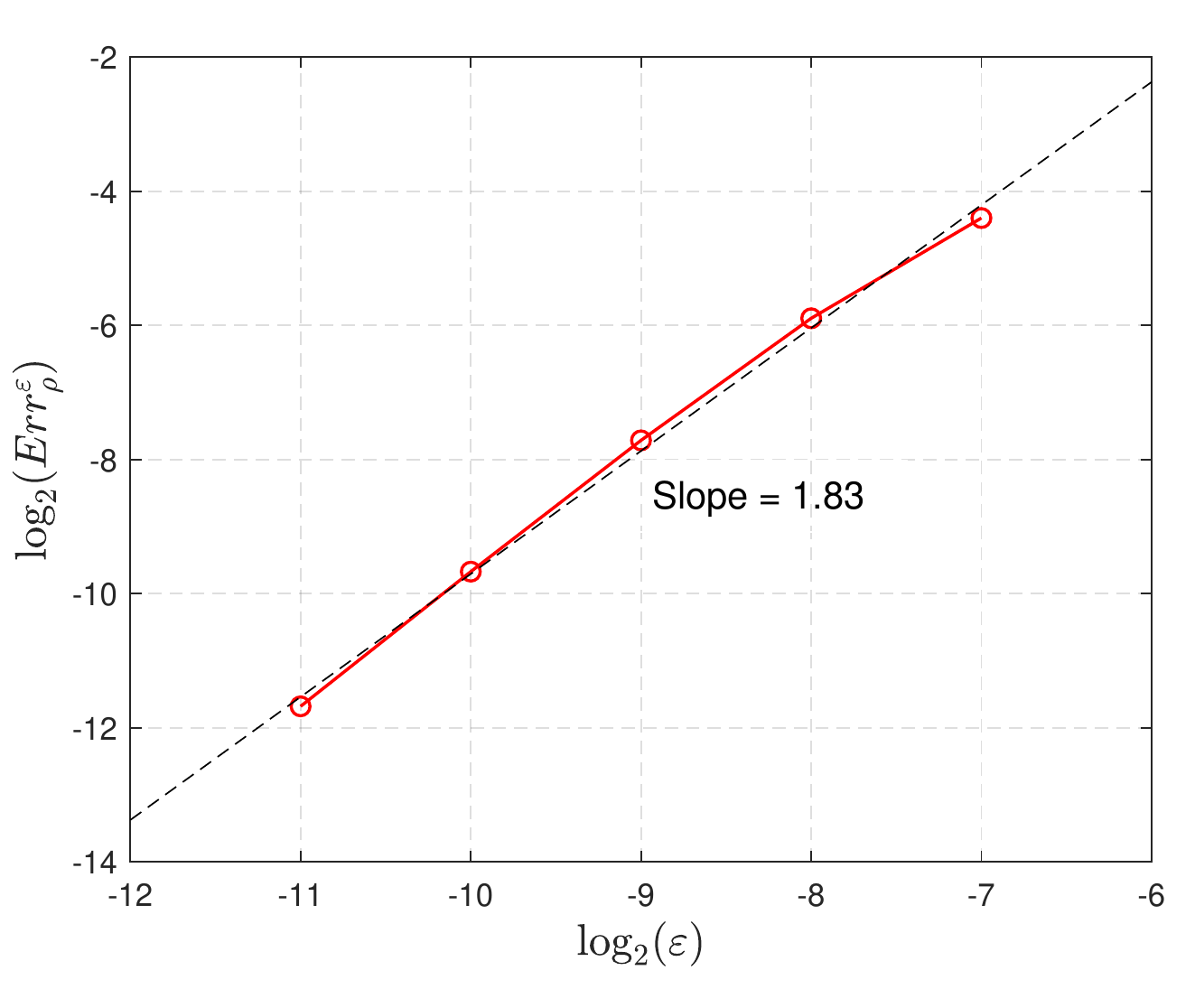}}
  \subfloat[]{\label{fig:Liouville_conv_J}\includegraphics[width=0.45\textwidth]{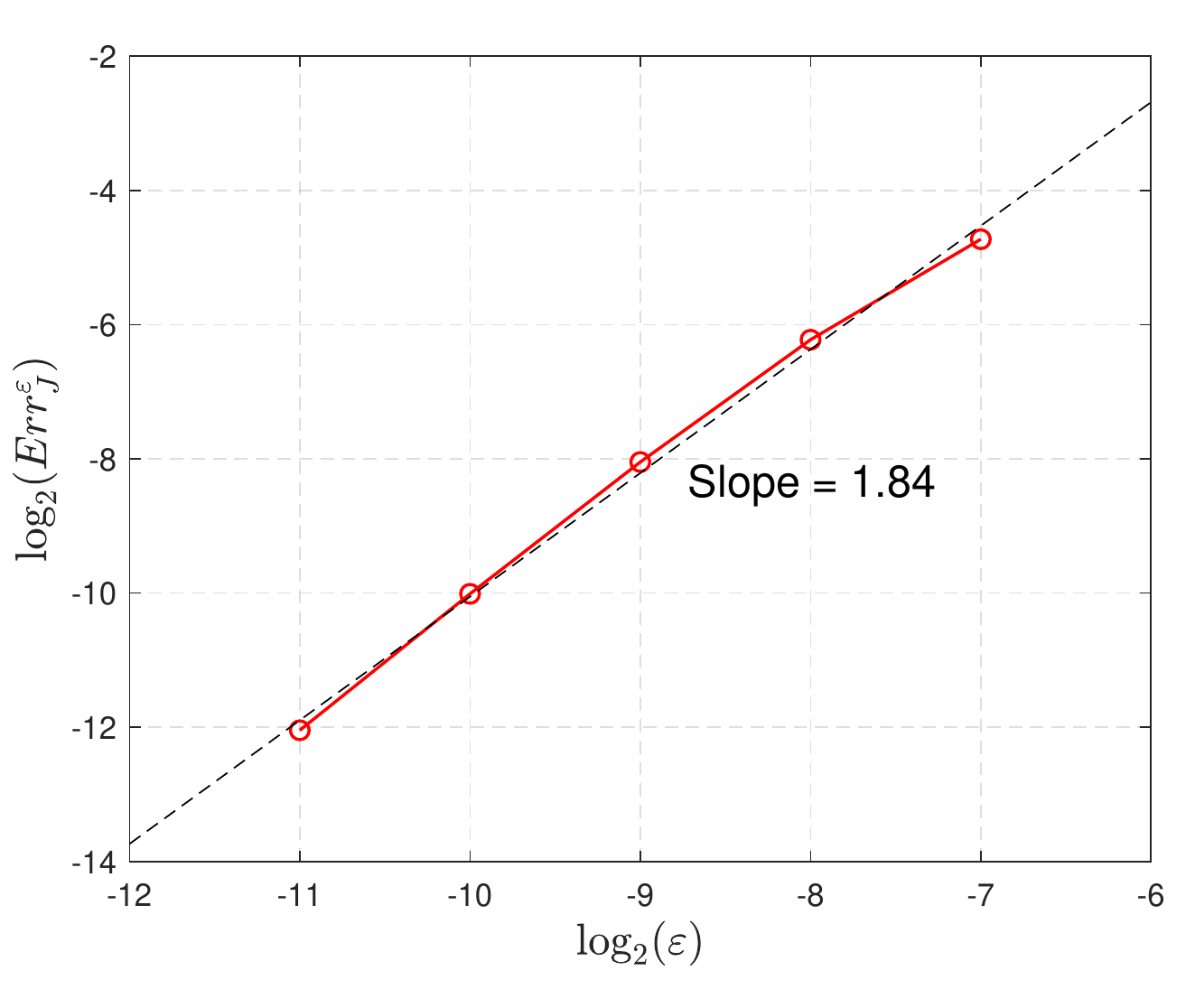}}
  \caption{ The plots (a)-(b) compare particle density $\rho^{\varepsilon}$ with $\rho^0$ and current density $J^{\varepsilon}$ with $J^0$ at $T = 0.5$ for different $\varepsilon$. The plots (c)-(d) show the errors $Err^{\varepsilon}_\rho$ and $Err^{\varepsilon}_J$ as a function of $\varepsilon$. The decay rate suggests that both errors are of $O(\varepsilon^2)$. The mass $m_0$~\cref{eqn:m0_x_t} is $t$-dependent and potential $V_0 = 0$.}
\end{figure}

In the second example, we consider a more practical setting. The effective mass and the external potential terms are selected to resemble a resonant tunneling diode, adopted from~\cite{ScSc:2020,ShCaTa:2006}, as are shown in~\cref{fig:setup_ex3}. We choose $L = 2$ and $T=0.5$. The effective mass $m_0$ is set to be
\begin{equation}\label{eqn:mass_ex3}
m_0(x) =
\begin{cases}
1 - 0.5\exp\left( 2^{-6}\left( \frac{4}{0.25^2} - \frac{1}{(0.75-x)(x-0.5)} \right) \right) \,, \quad 0.5<x<0.75 \\
1 - 0.5\exp\left( 2^{-6}\left( \frac{4}{0.25^2} - \frac{1}{(1.25-x)(x-1)} \right) \right) \,, \quad 1<x<1.25 \\
0 \,, \quad \mathrm{otherwise}\,.
\end{cases}
\end{equation}
and the potential $V_0$ is
\begin{equation}\label{eqn:potential_ex3}
V_0(x) =
\begin{cases}
\exp\left( 2^{-6}\left( \frac{4}{0.25^2} - \frac{1}{(0.75-x)(x-0.5)} \right) \right) \,, \quad 0.5<x<0.75 \\
\exp\left( 2^{-6}\left( \frac{4}{0.25^2} - \frac{1}{(1.25-x)(x-1)} \right) \right) \,, \quad 1<x<1.25 \\
0 \,, \quad \mathrm{otherwise}\,.
\end{cases}
\end{equation}
Note that in the original paper~\cite{ScSc:2020,ShCaTa:2006}, both terms are piece-wise constants. The discontinuity is beyond what we analyze in our paper and we smooth the transitions.

We use the same initial data as the first example, and set the spatial size $\Delta x = 2^{-10}$ and the frequency step $\Delta k =  2^{-10}$ in the computation of Liouville equation. We set the rescaled Planck constant $\varepsilon = 2^{-n}$ in VMSE and the discretization is chosen to be
\begin{equation}
\Delta t = 2^{-1.5n-3}, \quad \Delta x = 2^{-n-2}\,,
\end{equation}
that resolves the scales. In~\cref{fig:contour_ex3} we show the solution to the transport equation~\cref{eqn:liouville_potential}. In~\cref{fig:rho_ex3} we compare $\rho^0$ and $\rho^\varepsilon$ with different $\varepsilon$, and in~\cref{fig:J_ex3} we compare $J^0$ and $J^\varepsilon$. In~\cref{fig:err_rho_ex3} and~\cref{fig:err_J_ex3}, we show the convergence of $\text{Err}^{\varepsilon}_\rho$ and $\text{Err}^{\varepsilon}_J$ as a function of $\varepsilon$. According to the plot, the errors still decay at a rate of $O(\varepsilon^2)$.

\begin{figure}[tbhp]
  \centering
  \subfloat[]{\label{fig:setup_ex3}\includegraphics[width=0.45\textwidth]{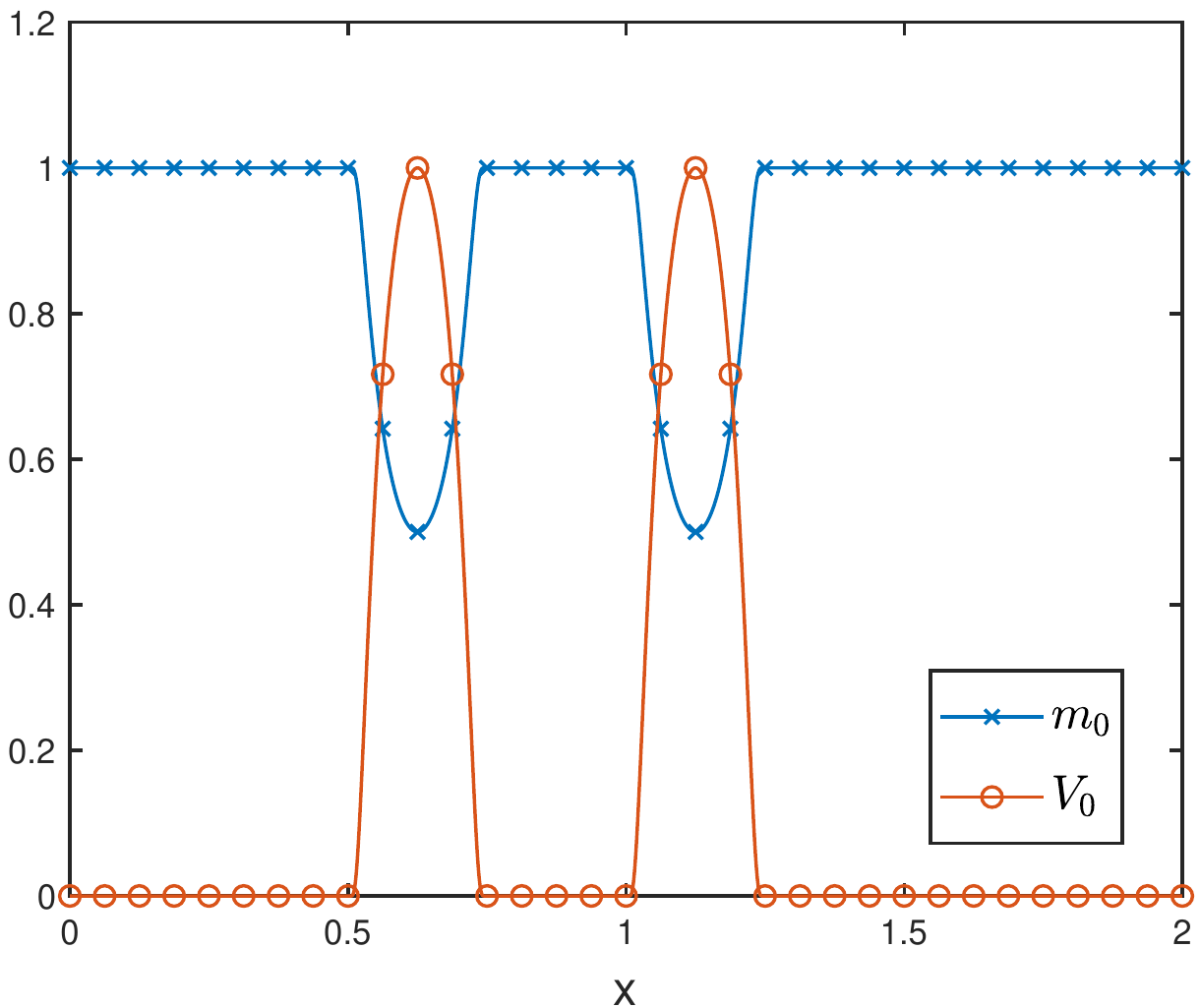}}
  \caption{The diagram of the potential $V_0(x)$~\cref{eqn:potential_ex3} and the effective mass $m_0(x)$~\cref{eqn:mass_ex3} of the resonant tunneling diode.}
\end{figure}

\begin{figure}[tbhp]
  \centering
  \subfloat[$t = 0.05$]{
  \includegraphics[width=0.45\textwidth]{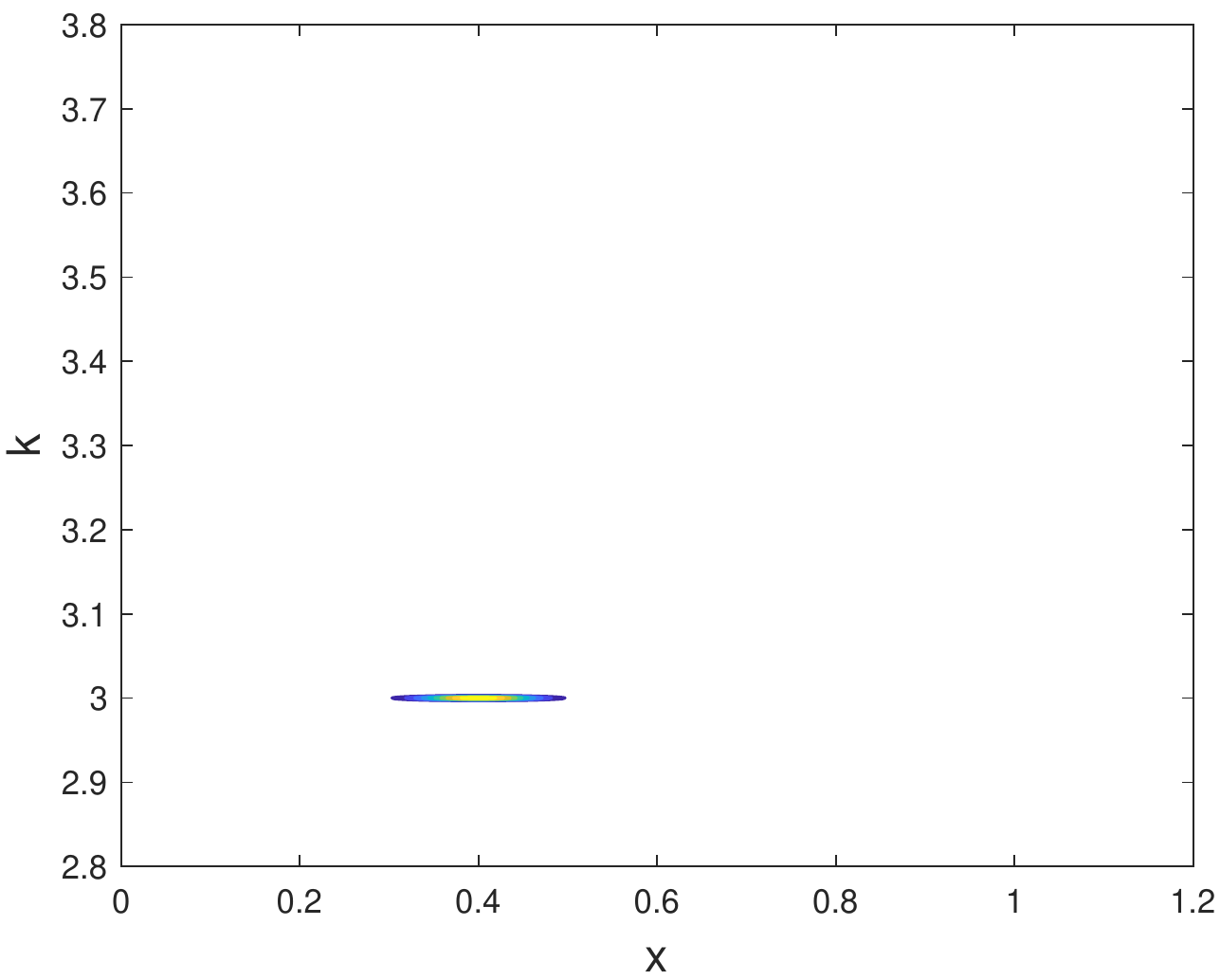}
  \includegraphics[width=0.45\textwidth]{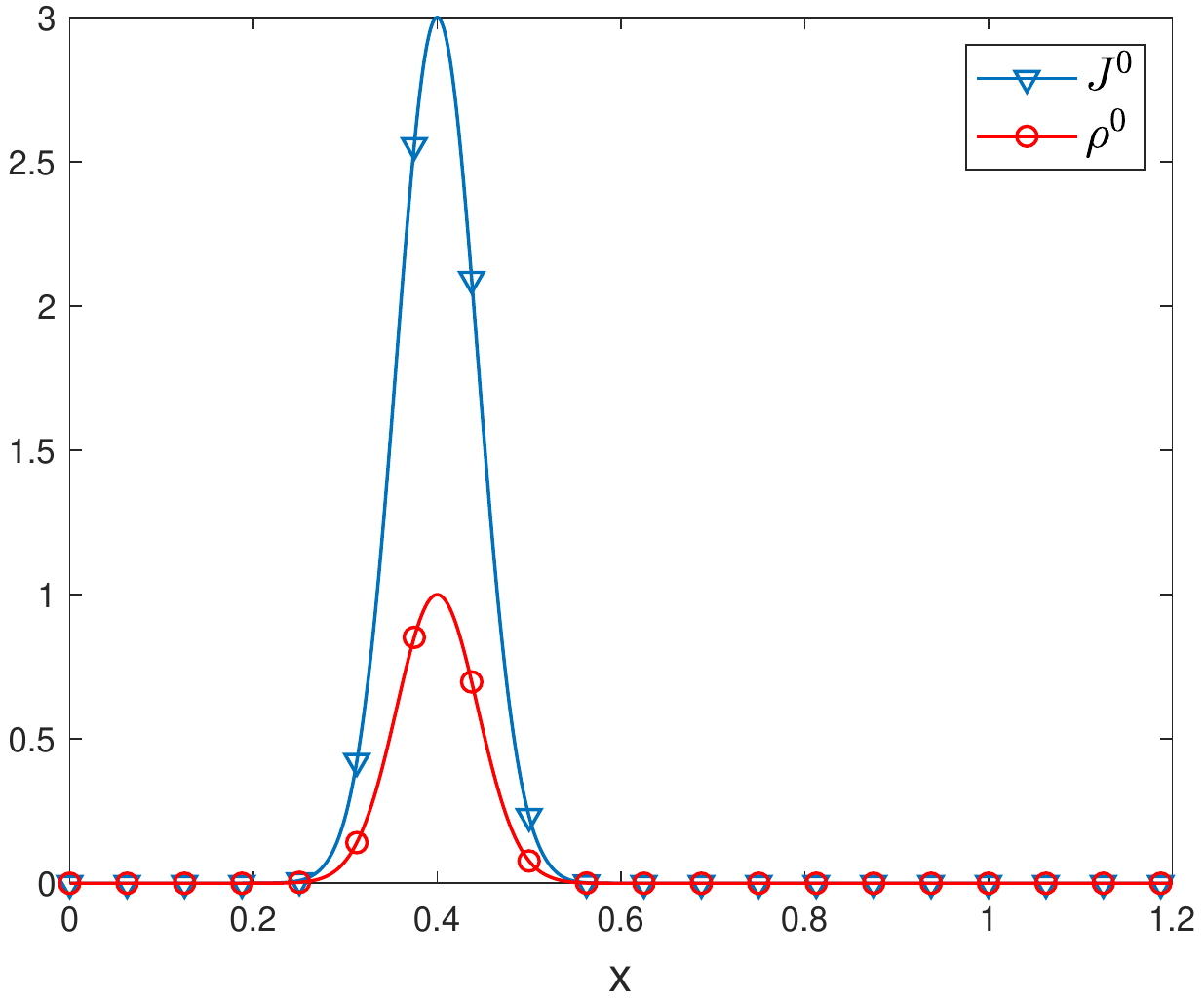}
  }

  \subfloat[$t = 0.15$]{
  \includegraphics[width=0.45\textwidth]{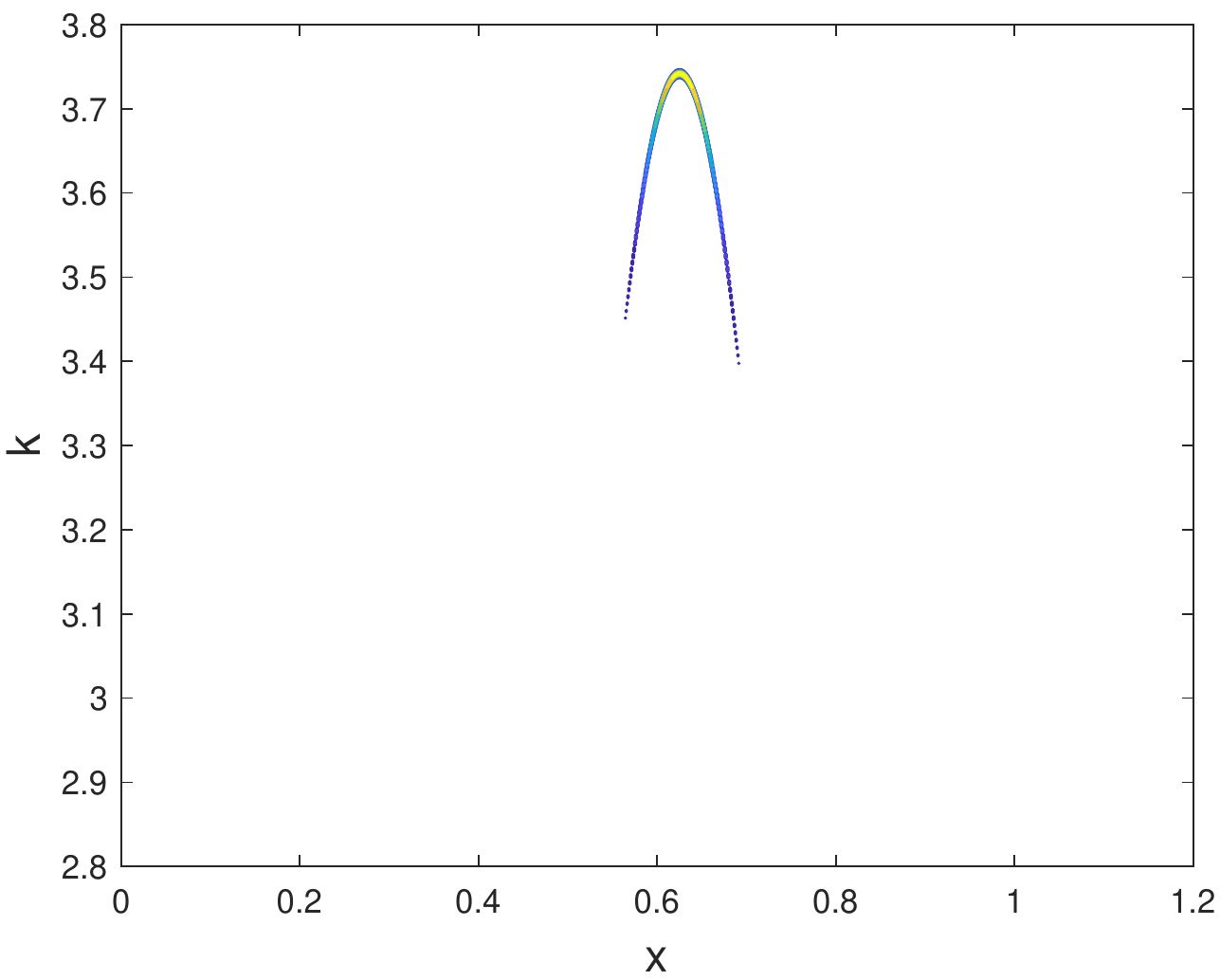}
  \includegraphics[width=0.45\textwidth]{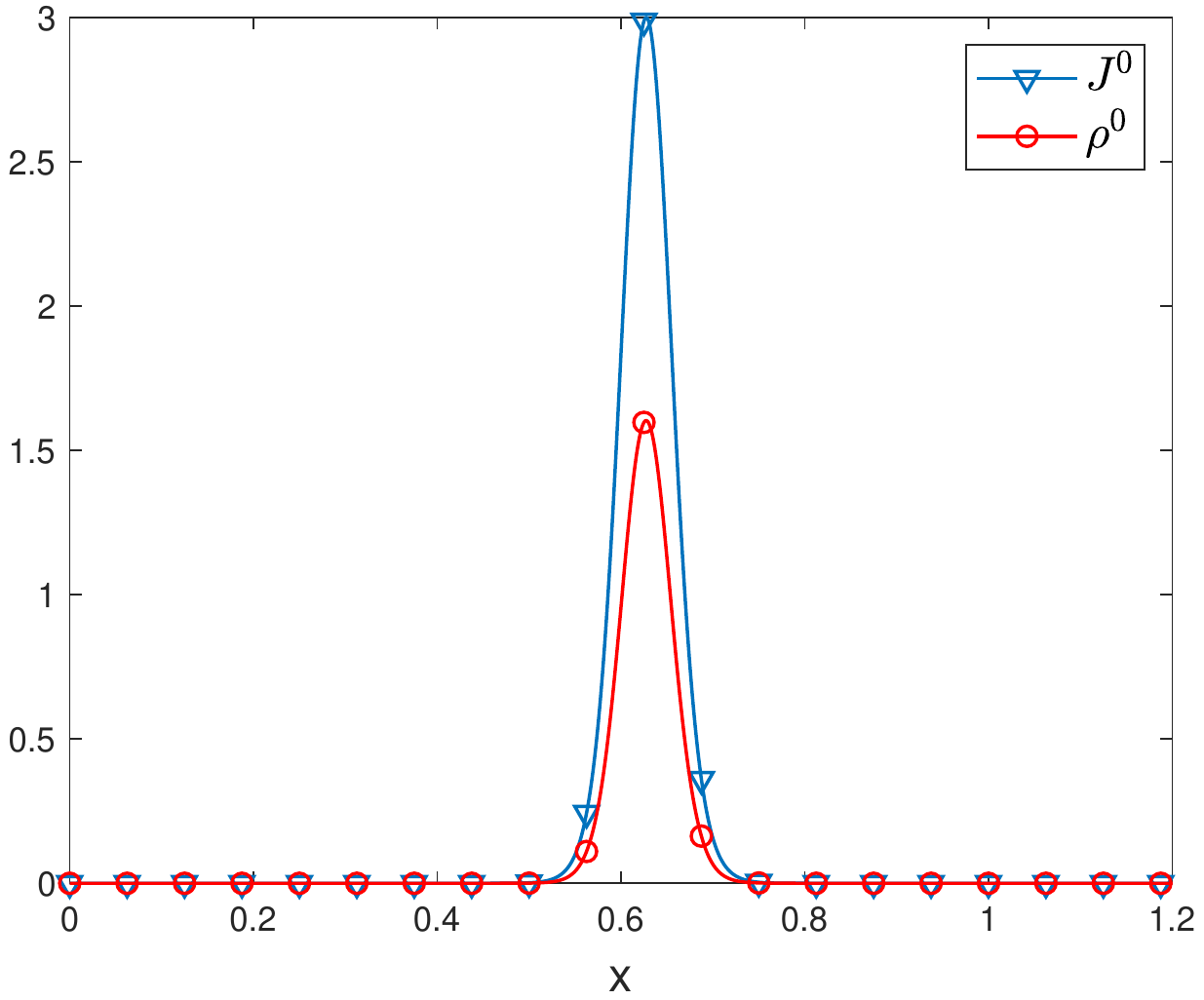}
  }

  \subfloat[$t = 0.4$]{
  \includegraphics[width=0.45\textwidth]{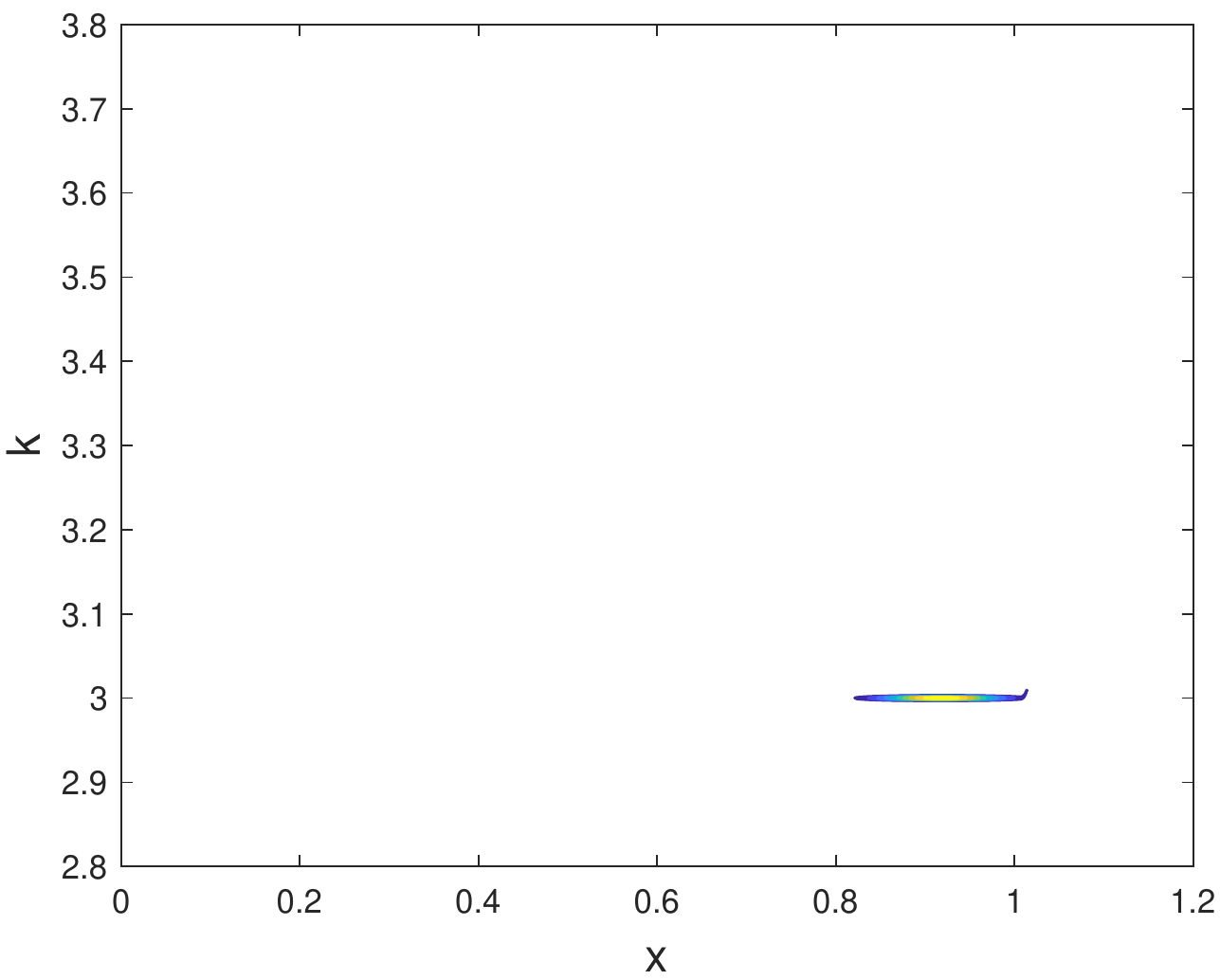}
  \includegraphics[width=0.45\textwidth]{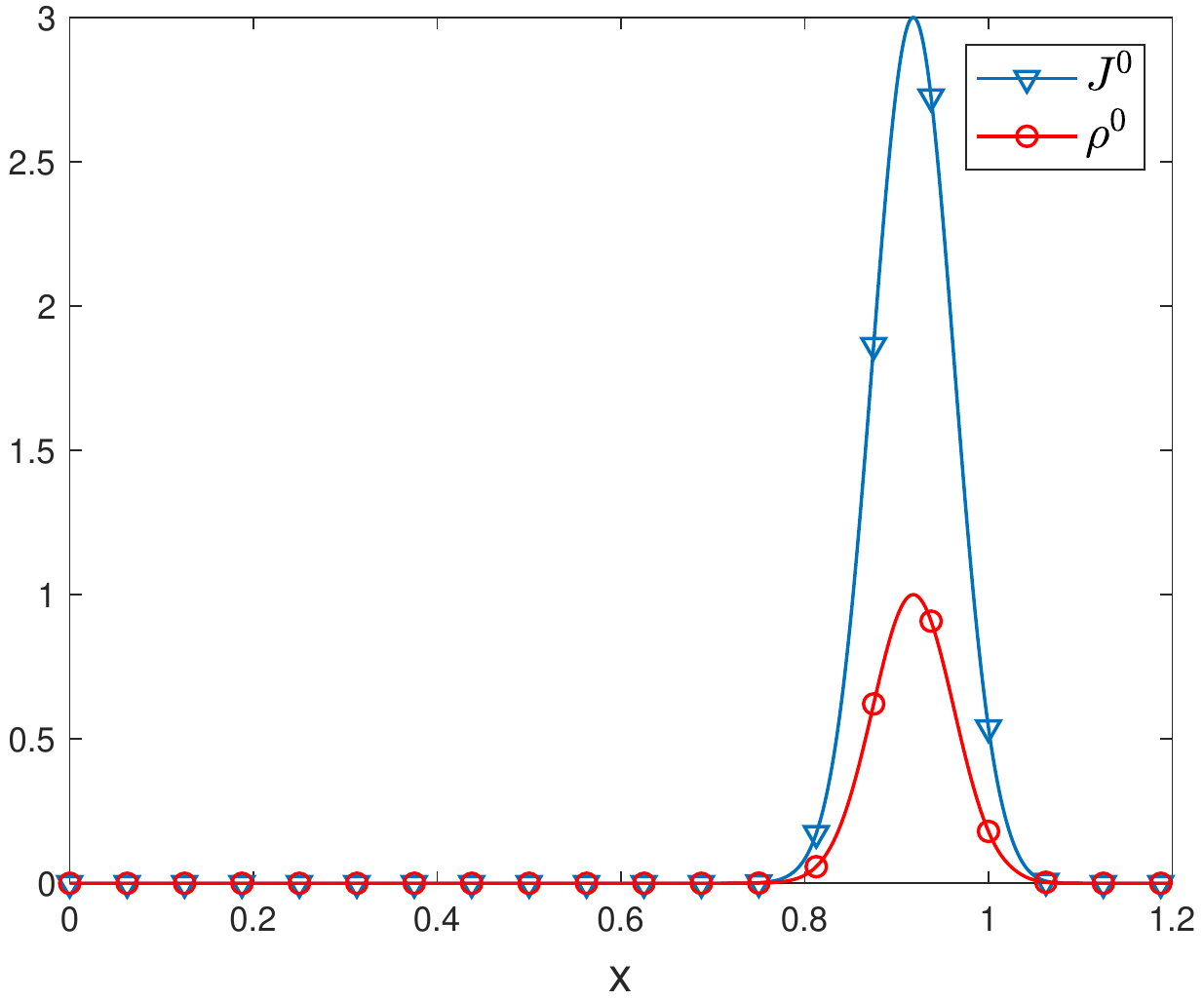}
  }
  \caption{The left column shows the contour of $W^0$ in phase space and the right column shows the particle density $\rho^0 = \int W^0 \rmd k$ and current density $J^0 = m_0 \int  k W^0 \rmd k$. The mass $m_0$~\cref{eqn:mass_ex3} and potential $V_0$~\cref{eqn:potential_ex3}.}
  \label{fig:contour_ex3}
\end{figure}

\begin{figure}[tbhp]
  \centering
  \subfloat[]{\label{fig:rho_ex3}\includegraphics[width=0.45\textwidth]{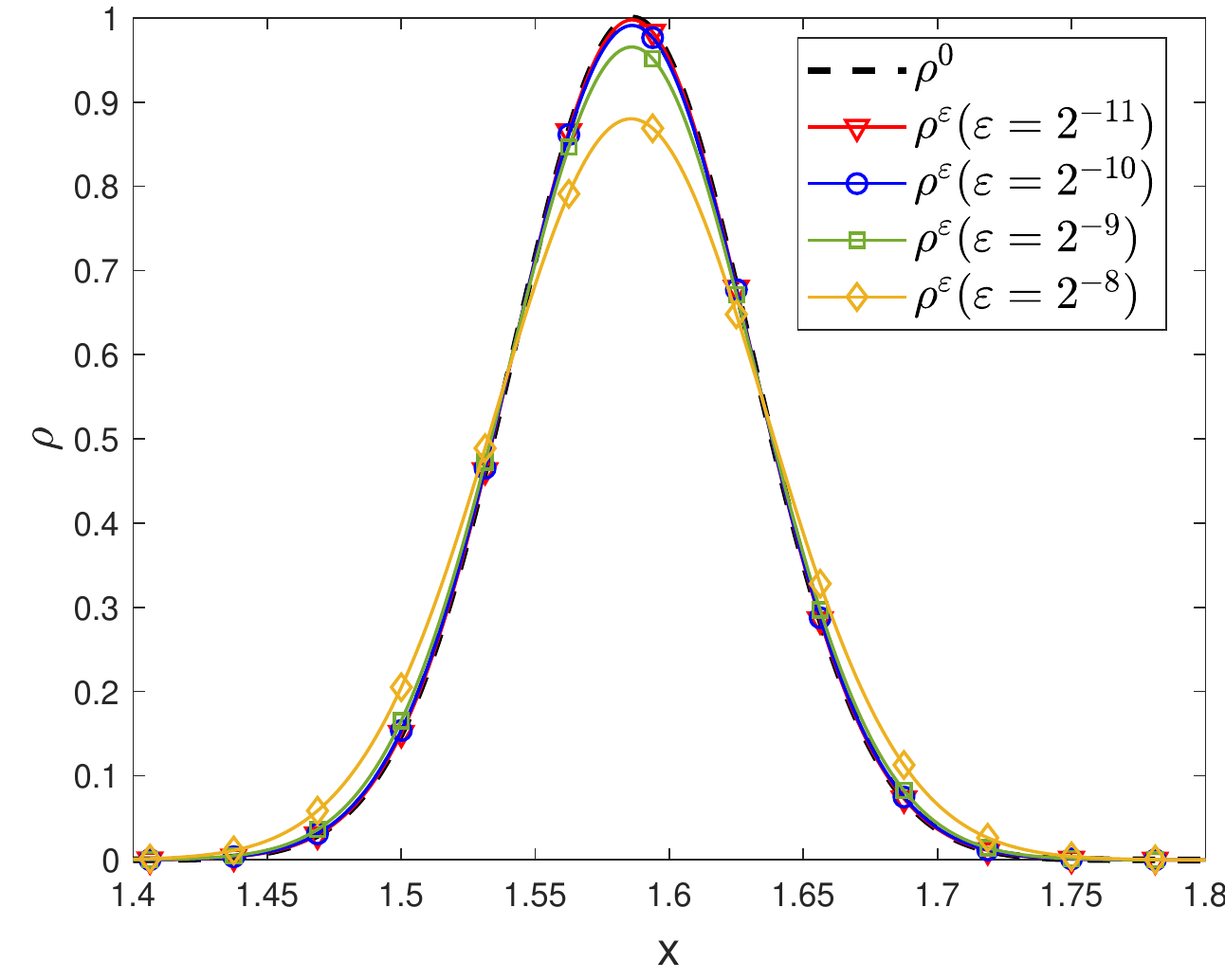}}
  \subfloat[]{\label{fig:J_ex3}\includegraphics[width=0.45\textwidth]{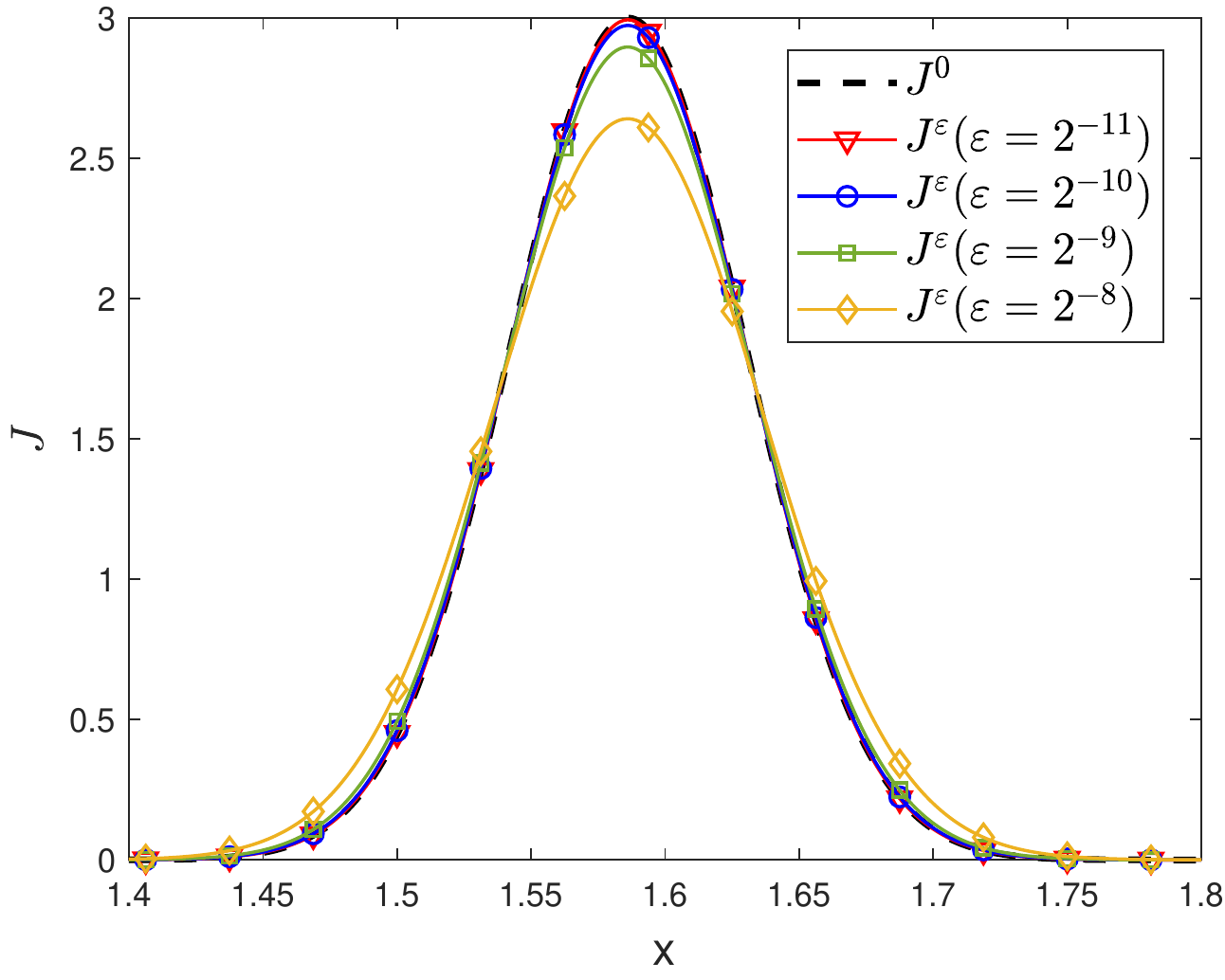}}

  \subfloat[]{\label{fig:err_rho_ex3}\includegraphics[width=0.45\textwidth]{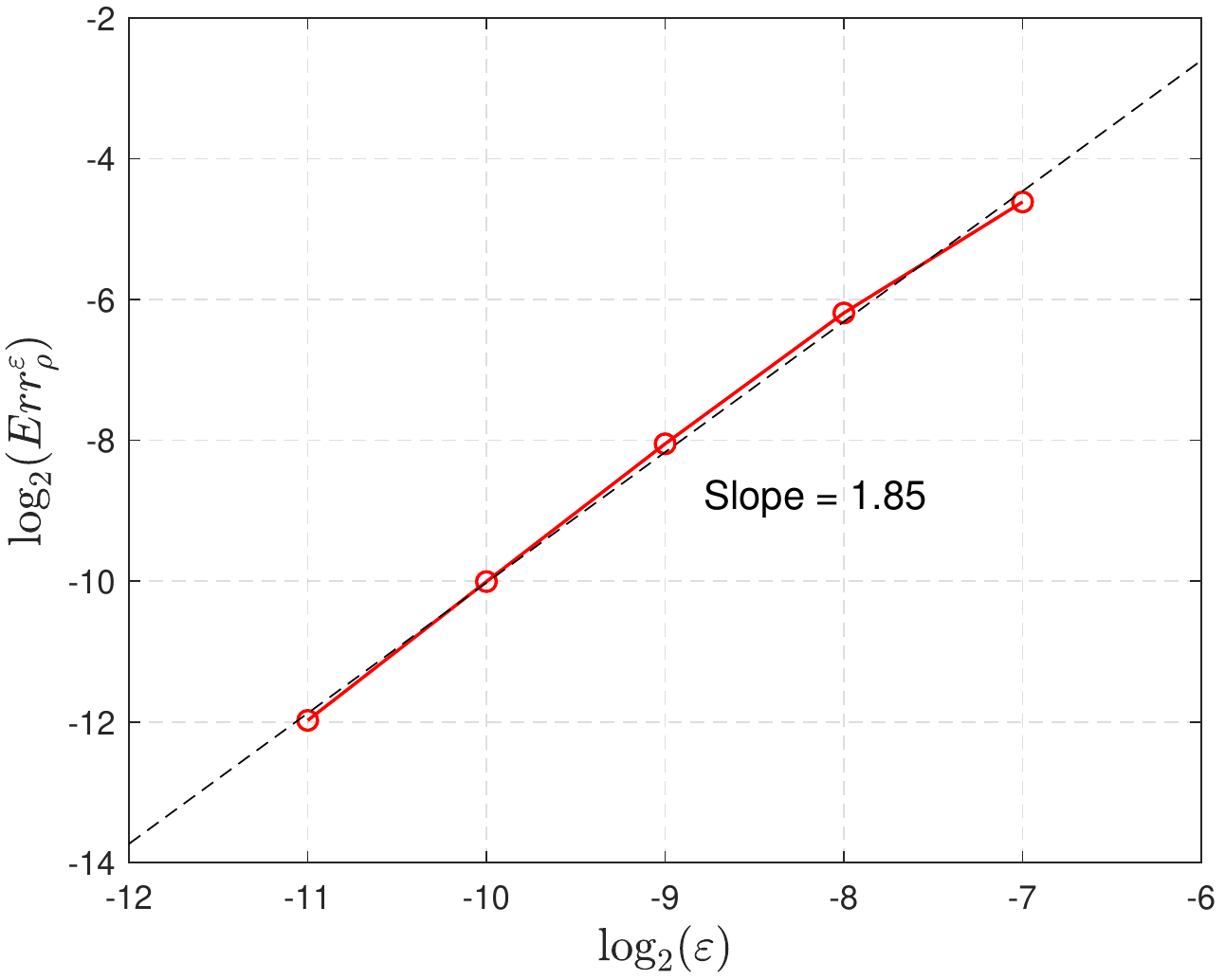}}
  \subfloat[]{\label{fig:err_J_ex3}\includegraphics[width=0.45\textwidth]{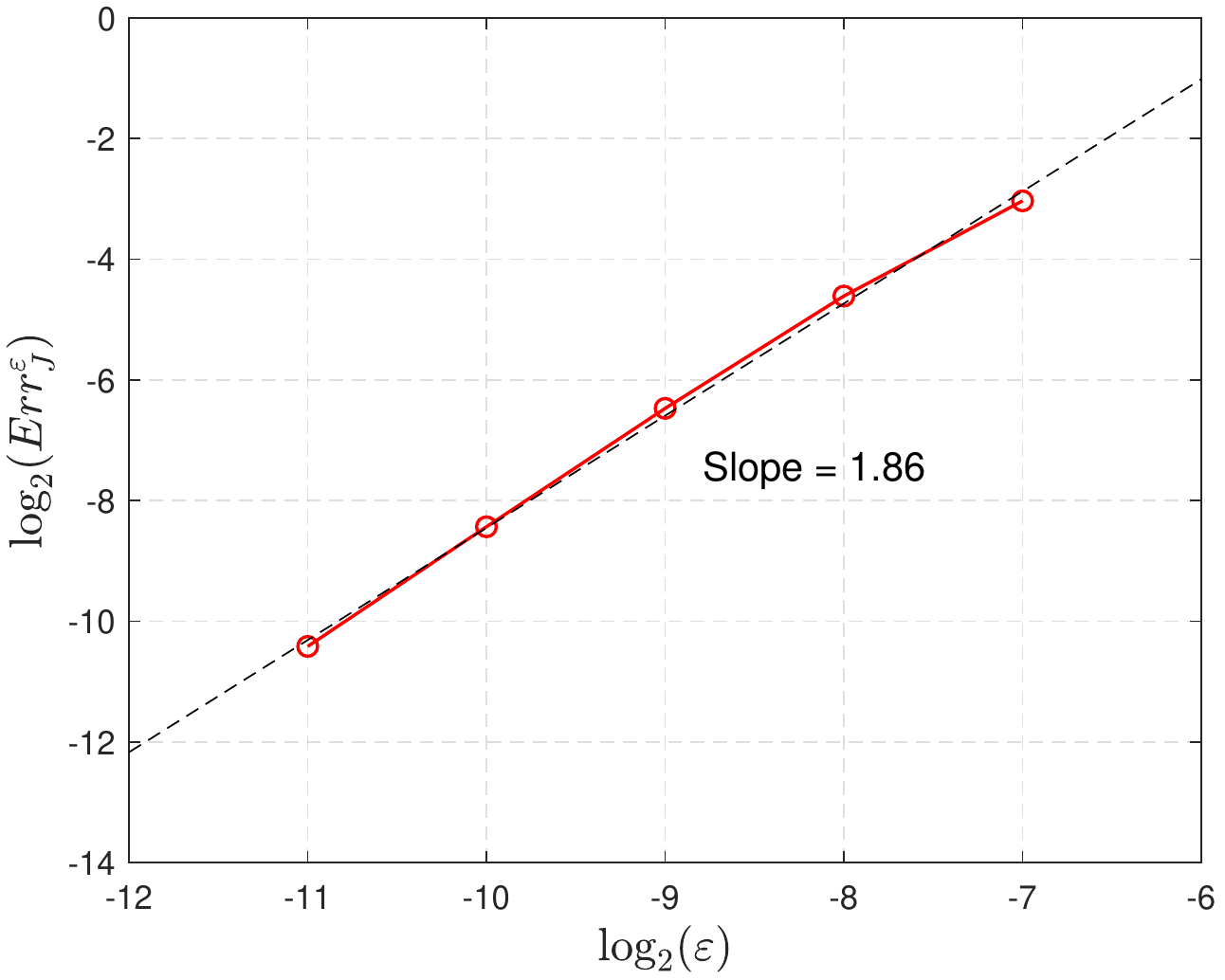}}
  \caption{ The plots (a)(b) compare particle density $\rho^{\varepsilon}$ with $\rho^0$ and current density $J^{\varepsilon}$ with $J^0$ at $T = 0.5$ for different $\varepsilon$. The plots (c)(d) show the errors $Err^{\varepsilon}_\rho$ and $Err^{\varepsilon}_J$ as a function of $\varepsilon$. The decay rate suggests that both errors are of $O(\varepsilon^2)$. The mass $m_0$~\cref{eqn:mass_ex3} and potential $V_0$~\cref{eqn:potential_ex3}.}
\end{figure}

\subsection{Illustration of~\cref{thm:W_rand}}
We show the numerical results on illustrating~\cref{thm:W_rand} here, namely, we will compute VMSE with small $\varepsilon$ and random mass, and compare the numerical results, when taking expectation values, with that of the limiting radiative transfer equation.

\subsubsection{Numerical setup}
As a set-up, we take the computational domain to be $\Omega = [0,L]\times[0,T]$, and set the correlation function to be:
\begin{equation}\label{eqn:correlation_num}
R(t,x) = \mathbb{E}[m_1(s,z)m_1(t+s,x+z)] = D^2\exp(-t/a-x/b)\,,
\end{equation}
where $a>0, b>0$ and $D^2$ is the variance of $m_1$.

We choose the initial data to have a Gaussian form:
\begin{equation}\label{eqn:initialSchr}
\ue_I(x) = \exp\left(-A(x-x_0)^2 + \frac{\ri}{\varepsilon}p_0x\right)\,.
\end{equation}
The periodic boundary conditions are imposed
\begin{equation}
\ue(t,0) = \ue(t,L), \quad \partial_x\ue(t,0) = \partial_x\ue(t,L)\,.
\end{equation}

Correspondingly the transport equation~\cref{eqn:RTE} has the initial data:
\begin{equation}\label{eqn:initialRTE}
W_I^0(x,k) = \exp(-2A(x-x_0)^2)\delta(k-p_0)\,,
\end{equation}
and it is equipped with periodic conditions:
\begin{equation}
W(t,0,k) = W(t,L,k),\quad \text{for } t>0 \text{ and all }k\in \bbr.
\end{equation}

Similar to the previous subsection, $L$ is set to be large enough and the periodic boundary condition plays minimum role.

The computation of the limiting radiative transfer equation is rather straightforward. Due to the form of the correlation function~\cref{eqn:correlation_num}, one has
\begin{equation}
\hat{R}(\omega,p) = \frac{4abD^2}{(1+a^2\omega^2)(1+b^2p^2)}\,.
\end{equation}
Since $m_0$ is a deterministic slow-varying function, the equation composes of two transport terms, which we use a fifth-order WENO scheme~\cite{JiSh:1996efficient}, and a collision operator,  which we apply the trapezoidal rule to approximate.

There are more numerical difficulties regarding the computation of VMSE. The challenge is two-folded: dealing with the randomness, and resolving the high oscillation. To handle the randomness, we perform the Karhunen-Lo\'eve expansion by setting~\cite{loeve}
\begin{equation}\label{eqn:KLseries}
m_1(t/\varepsilon,x/\varepsilon) = D\sum_{i,j=1}^{\infty}\sqrt{\lambda_i^{\varepsilon}\sigma_j^{\varepsilon}}\psi_i^{\varepsilon}(t)\phi_j^{\varepsilon}(x)\xi_{ij},
\end{equation}
where $\xi_{ij}$ are i.i.d. random variables with
\[
\mathbb{E}[\xi_{ij}] = 0\,,\quad \mathbb{E}[\xi_{ij}^2] = 1, \quad\forall i,j = 1,2,\cdots\,.
\]
The form of $\xi$ depends on the field, and we numerically use either uniformly distributed random variable or Gaussian random variable. $\lambda_i^{\varepsilon}$ and $\sigma_j^{\varepsilon}$ are descending eigenvalues corresponding eigenfunctions $\psi_i^{\varepsilon}$ and $\phi_j^{\varepsilon}$:
\begin{equation}
\begin{aligned}
\int_0^{T}e^{-\frac{|t-s|}{a\varepsilon}}\psi_i^{\varepsilon}(s)\rmd s = \lambda_i^{\varepsilon} \psi_j^{\varepsilon}(t)\,,\quad \int_0^{L}e^{-\frac{|x-z|}{b\varepsilon}}\phi_j^{\varepsilon}(z)\rmd z = \sigma_j^{\varepsilon} \phi_j^{\varepsilon}(x)\,.
\end{aligned}
\end{equation}

For the particular form of $R$ defined in~\cref{eqn:correlation_num}, it is shown in~\cite{Xi:2010} that
\begin{equation}
\begin{aligned}
&\lambda_i^{\varepsilon} = \frac{2a\varepsilon}{1+a^2\varepsilon^2w_i^2},\quad \sigma_j^{\varepsilon} = \frac{2b\varepsilon}{1+b^2\varepsilon^2v_j^2}\,,\\
&\psi_i^{\varepsilon}(t) =
\begin{cases}
\sin(w_i (t-T/2))\bigg/\sqrt{\frac{T}{2}-\frac{\sin(w_i T)}{2w_i}},\quad \text{if }i\text{ is even}\,,\\
\cos(w_i (t-T/2))\bigg/\sqrt{\frac{T}{2}+\frac{\sin(w_i T)}{2w_i}},\quad \text{if }i\text{ is odd}\,,
\end{cases}\\
&\phi_j^{\varepsilon}(x) =
\begin{cases}
\sin(v_j (x-L/2))\bigg/\sqrt{\frac{L}{2}-\frac{\sin(v_j L)}{2v_j}},\quad \text{if }j\text{ is even}\,,\\
\cos(v_j (x-L/2))\bigg/\sqrt{\frac{L}{2}+\frac{\sin(v_j L)}{2v_j}},\quad \text{if }j\text{ is odd}\,.
\end{cases}
\end{aligned}
\end{equation}
where $w_i$ and $v_j$ are solutions to
\begin{equation}
\begin{aligned}
&\begin{cases}
a\varepsilon w_i + \tan(w_i\frac{T}{2}) = 0, \quad \text{for even }i\,,\\
1 - a\varepsilon w_i \tan(w_i\frac{T}{2}) = 0, \quad \text{for odd }i\,,
\end{cases}\\
&\begin{cases}
b\varepsilon v_j + \tan(v_j\frac{L}{2}) = 0, \quad \text{for even }j\,,\\
1 - b\varepsilon v_j \tan(v_j\frac{L}{2}) = 0, \quad \text{for odd }j\,.
\end{cases}
\end{aligned}
\end{equation}

Numerically we perform Monte Carlo, that is to sample a large number of $N$ configurations of $\xi_{ij}$ which give rise to $N$ configuration of $m_1$. For these deterministic $m_1$, we compute the deterministic VMSE, and take the ensemble mean and variance in the end.

For Schr\"odinger equation, the Crank-Nicolson and spectral method are applied as in the previous section with the scales resolved: $\Delta x = O(\varepsilon)$ and $\Delta t = o(\varepsilon)$. Note that $m_1$ is already deterministic for each Monte Carlo sample.

Numerically to illustrate Theorem~\cref{thm:W_rand}, we mainly compare the macroscopic quantities. In particular we will compare the particle density and the current density, that is to compare
\begin{equation}\label{eqn:density_ue}
\rho^0(t,x) = \int W^0(t,x,k)\rmd{k}\,, \qquad \mathbb{E}[\rho^{\varepsilon}(t,x)] = \mathbb{E}[|u^{\varepsilon}(t,x)|^2]\approx \frac{1}{N}\sum_{i = 1}^{N} |u_i^{\varepsilon}(t,x)|^2\,,
\end{equation}
and
\begin{equation}\label{eqn:J_ue}
\begin{aligned}
J^0(t,x) &= \int m_0(t,x)k W^0(t,x,k)\rmd{k}\,, \\
\mathbb{E}[J^{\varepsilon}(t,x)] &= \mathbb{E}[\varepsilon\text{Im}\left(m^\varepsilon(t,x)\overline{\ue(t,x)}\nabla_x\ue(t,x)\right)]\\
&\approx \frac{1}{N}\sum_{i = 1}^{N} \varepsilon\text{Im}\left(m_i^\varepsilon(t,x)\overline{\ue_i(t,x)}D^\mathrm{s}_x\ue_i(t,x)\right)\,.
\end{aligned}
\end{equation}

\subsubsection{Numerical examples}
We demonstrate two numerical examples: one to illustrates~\cref{thm:W_rand} where VMSE has the right scaling, and in the second example we use the wrong perturbation scaling simply to observe the differnce.

In the first example, we set $\Omega = [0,1.625]\times[0,0.4]$, and the parameters in $R(t,x)$ (defined in~\cref{eqn:correlation_num}) are $a = b = 100$. For the initial data~\cref{eqn:initialSchr} and~\cref{eqn:initialRTE}, we take $A = 2^8$ and $x_0 = 0.3$, and set  $m_0 = 1$. To compute RTE, we set $\Delta x = \Delta k = 2^{-10}$ and $\Delta t = 2^{-12}$. To compute VMSE, we set $\varepsilon = 2^{-n}$ and we use the discretization:
\begin{equation}\label{eqn:stepSchr}
\Delta t = 2^{-1.2n-3}, \quad \Delta x = 2^{-n-2}\,.
\end{equation}
The KL series is truncated at $\nkl$ finite terms with
\begin{equation}\label{eqn:KLtruncate}
\sqrt{(\lambda_i^{\varepsilon}\sigma_j^{\varepsilon})_{\nkl}/\lambda_1^{\varepsilon}\sigma_1^{\varepsilon}}<2^{-9}\,.
\end{equation}
As $i$ and $j$ increase, the oscillations in the associated eigenfunctions $\phi$ and $\psi$ also increase, but the choice of $\Delta x, \Delta t$ ensures that these oscillations are resolved. For $\varepsilon = 2^{-6}, 2^{-8}, 2^{-10}$ respectively, $\nkl=663, 3157, 27968$ to ensure~\cref{eqn:KLtruncate}. $10000$ Monte Carlo samples are used in total.

In~\cref{fig:mass_transport1} we show the solution to the transport equation~\cref{eqn:RTE} at three specific time for $D=1.5$ and $p_0=1.5$.

\begin{figure}[tbhp]
  \centering
  \subfloat[$t = 0.1167$]{
  \includegraphics[width=0.45\textwidth]{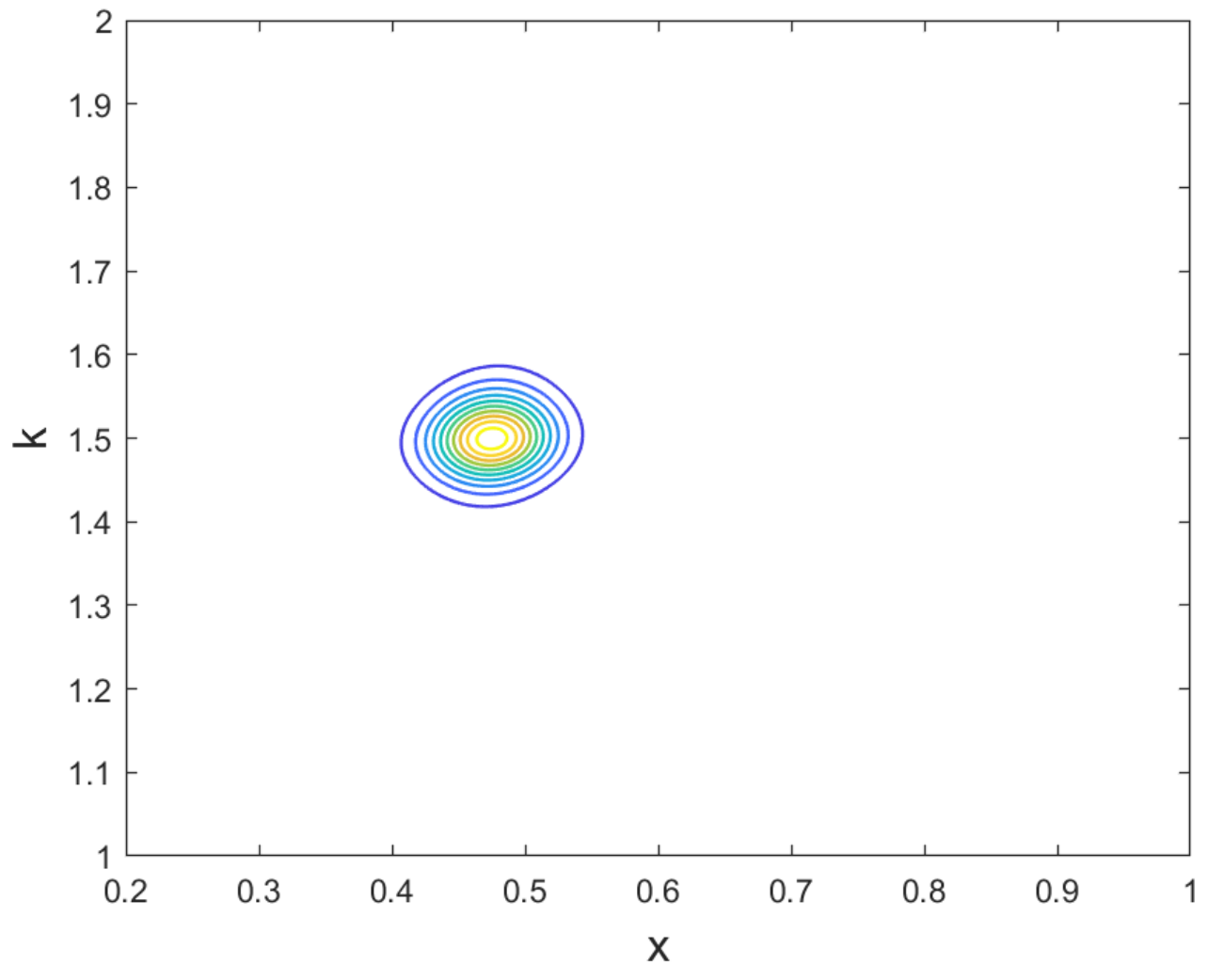}
  \includegraphics[width=0.45\textwidth]{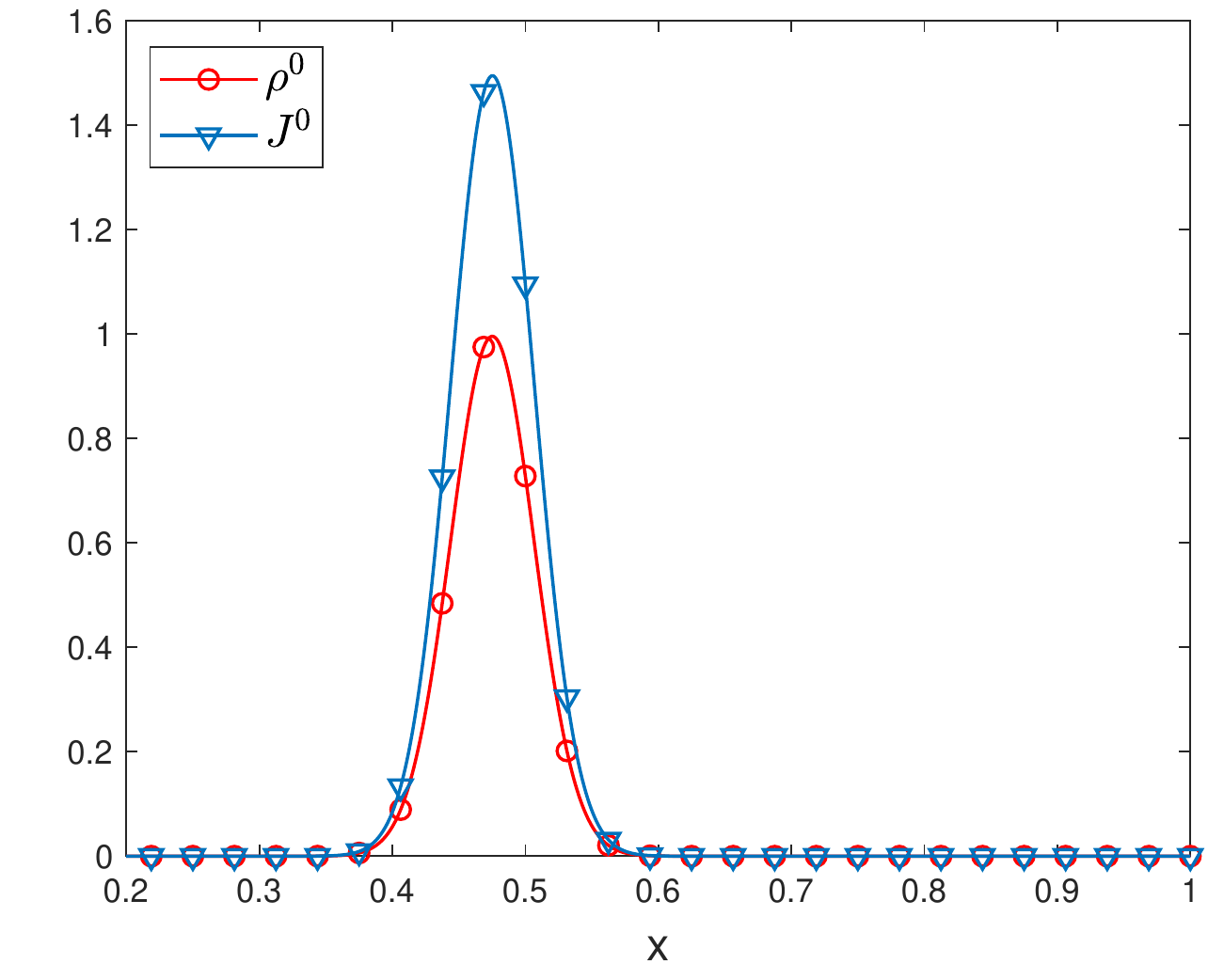}
  }

  \subfloat[$t = 0.2335$]{
  \includegraphics[width=0.45\textwidth]{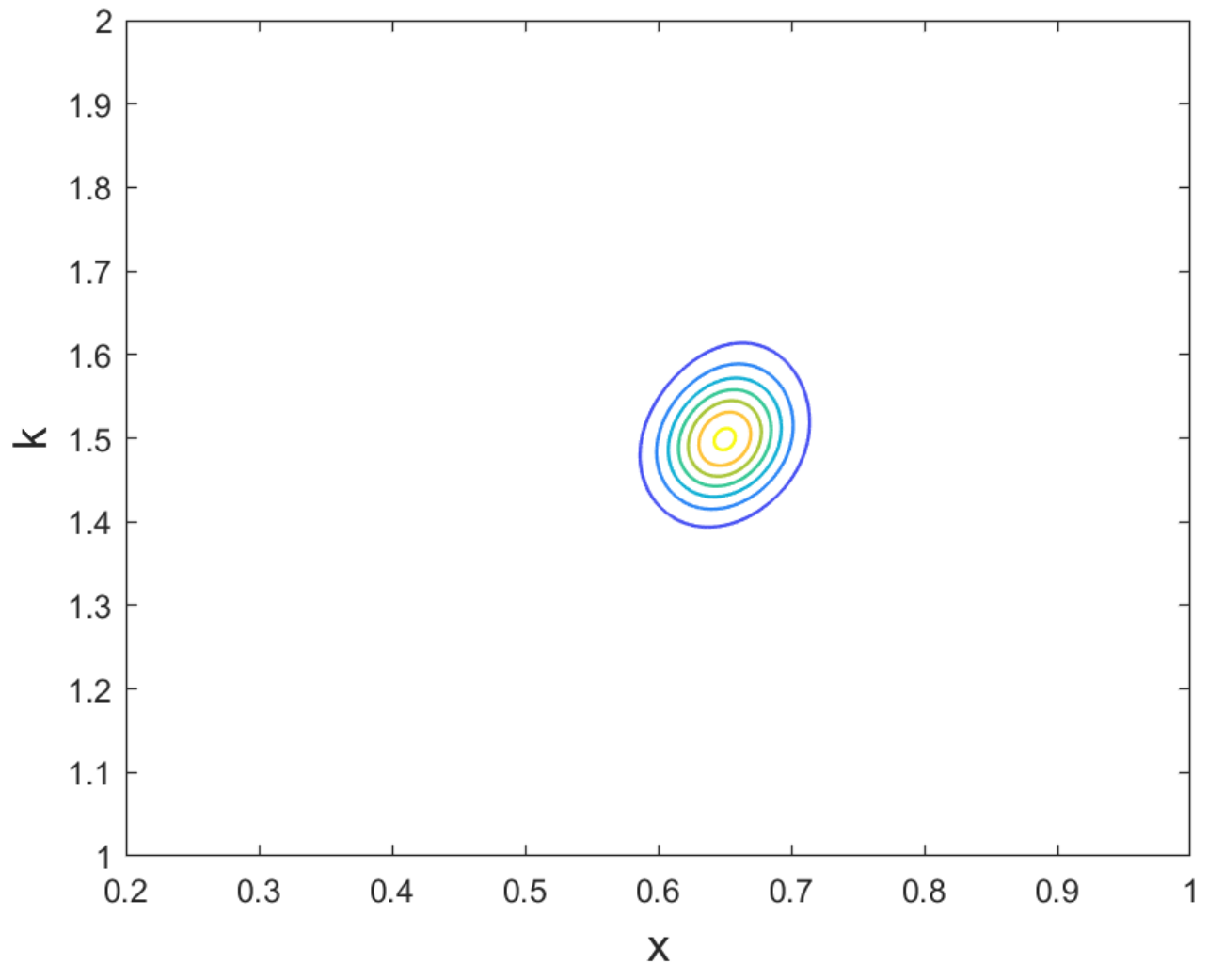}
  \includegraphics[width=0.45\textwidth]{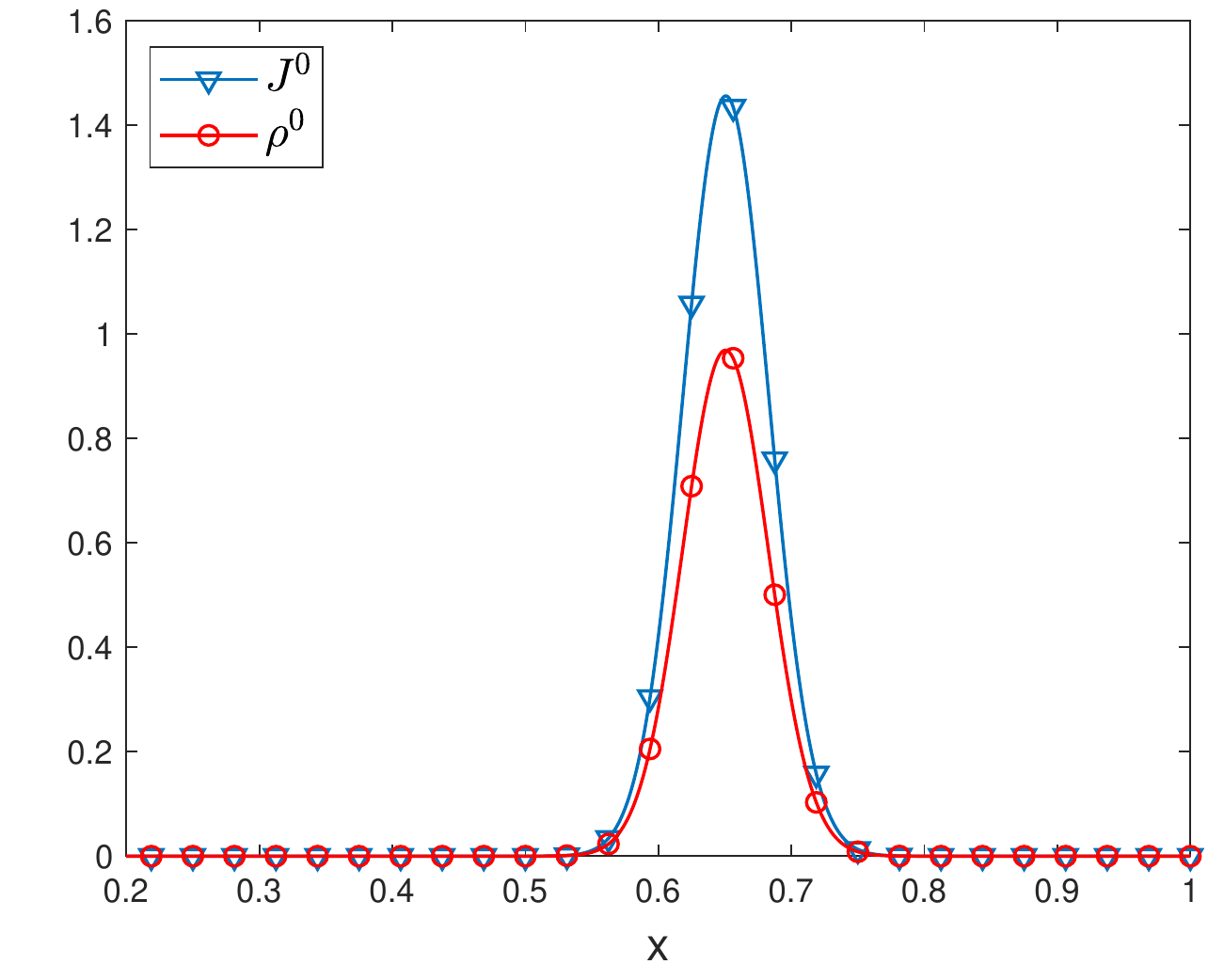}
  }

  \subfloat[$t = 0.3500$]{
  \includegraphics[width=0.45\textwidth]{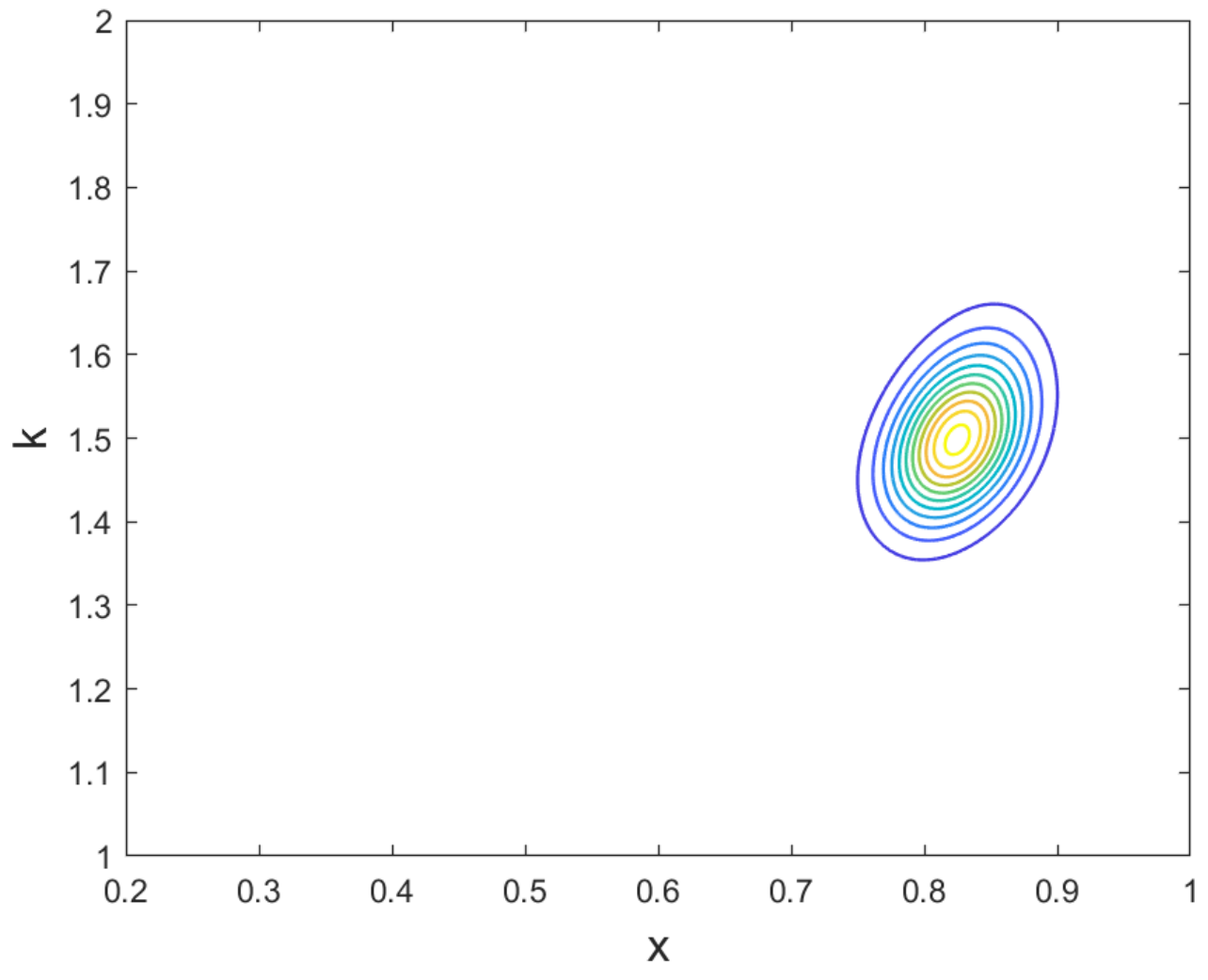}
  \includegraphics[width=0.45\textwidth]{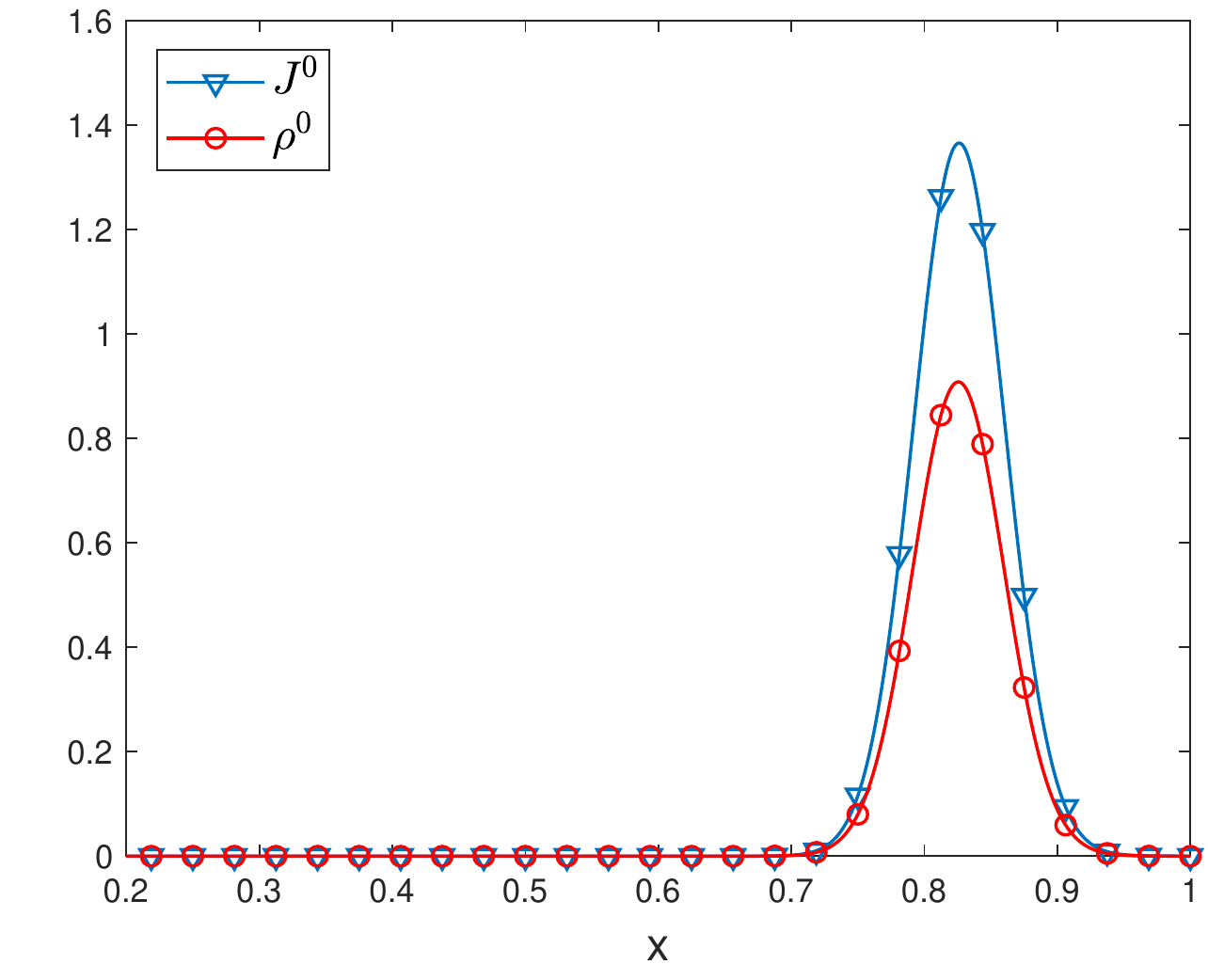}
  }
  \caption{The left column shows the contour of $W^0$ in phase space and the right column shows the particle density $\rho^0 = \int W^0 \rmd k$ and the current density $J^0 = m_0\int k W^0 \rmd k$.}
  \label{fig:mass_transport1}
\end{figure}

In~\cref{fig:mass_RTE_vs_Schr1_Dp} we show that for different pairs of $(D,p_0)$, the numerical solution to RTE and numerical solution to VMSE are rather close for $\varepsilon = 2^{-10}$.
\begin{figure}[htbp]
  \centering
  \subfloat[]{\includegraphics[width=0.45\textwidth]{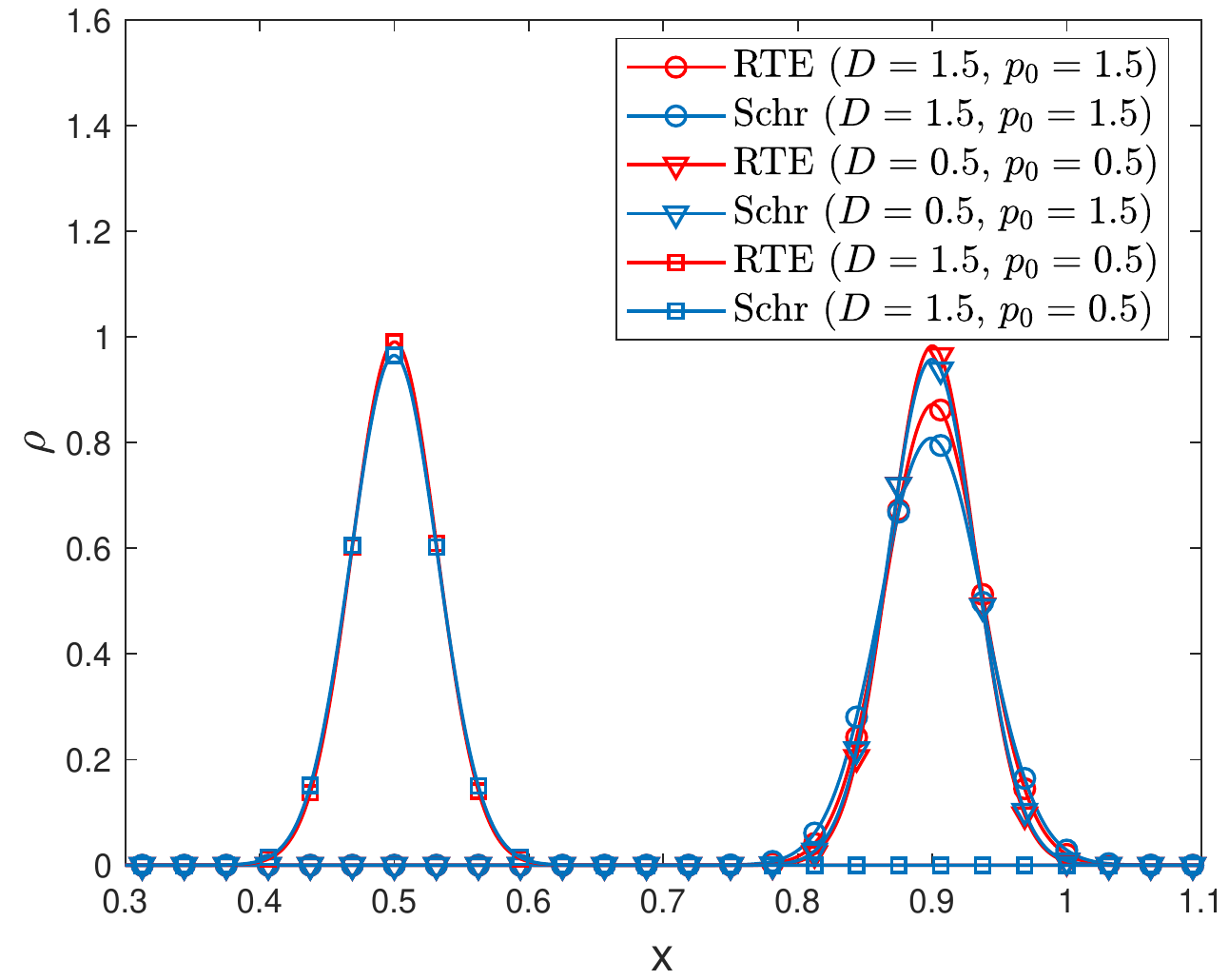}}
  \subfloat[]{\includegraphics[width=0.45\textwidth]{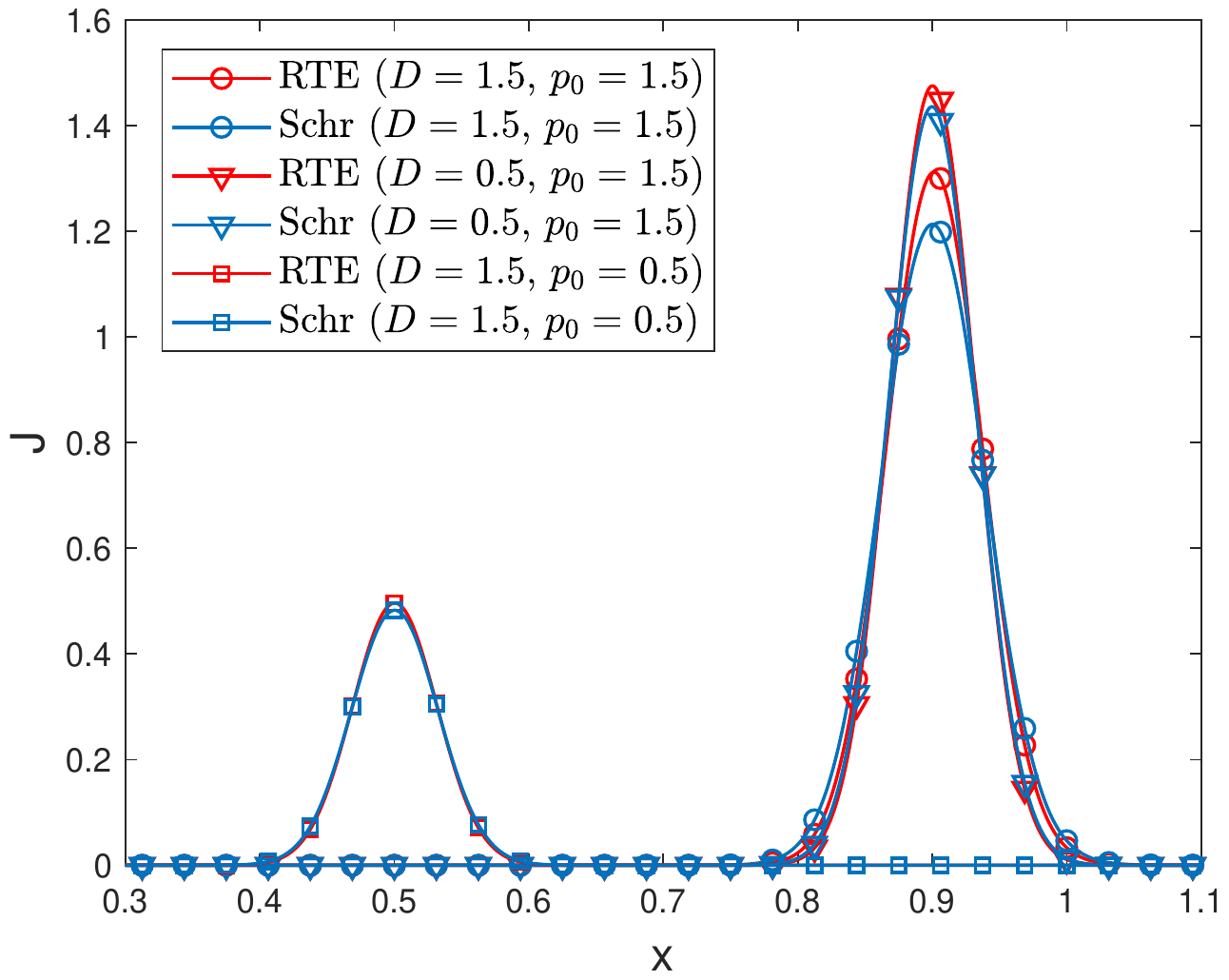}}
  \caption{The plot (a) shows the particle density $\mathbb{E}[\rho^{\varepsilon}]$ ($\varepsilon = 2^{-10}$) and $\rho^0 = \int W^0 \rmd k$ at $t = 0.4$ with different $(D,p_0)$ pairs. The plot (b) shows the current density $\mathbb{E}[J^{\varepsilon}]$ ($\varepsilon = 2^{-10}$) and $J^0 = m_0 \int k W^0 \rmd k$ at $t = 0.4$ with different $(D,p_0)$ pairs.}
  \label{fig:mass_RTE_vs_Schr1_Dp}
\end{figure}

It is fairly straightforward to observe the convergence of VMSE to RTE as $\varepsilon\to0$. Such convergence can also be quantified. Define the error:
\begin{equation}\label{eqn:err_rho}
\text{Err}^{\varepsilon}_\rho = \int_{\bbr} |\rho^0 - \mathbb{E}[\rho^{\varepsilon}]| \rmd x\,, \quad
\text{Err}^{\varepsilon}_J = \int_{\bbr} |J^0 - \mathbb{E}[J^{\varepsilon}]| \rmd x
\end{equation}
In~\cref{fig:mass_RTE_vs_Schr1} we compare $\rho^0$ and $\mathbb{E}[\rho^\varepsilon]$, $J^0$ and $\mathbb{E}[J^\varepsilon]$ with different $\varepsilon$, fixing $D = 1.5$ and $p_0= 1.5$. In~\cref{fig:mass_density_conv1}, we show the convergence of $\text{Err}^{\varepsilon}_\rho$ and $\text{Err}^{\varepsilon}_J$ as a function of $\varepsilon$ for both Gaussian and uniform distributed variable $\xi_{ij}$. According to the plot, the error decays at a rate of $O(\varepsilon)$ -- this is stronger than our ansatz where $W^{(1)}$ is assumed to be at the order of $\sqrt{\varepsilon}$. This suggests that $\mathbb{E}[W^{(1)})]$ is of higher order than $\sqrt{\varepsilon}$, but is not yet proved. We also note that although we do not have theoretical result on the convergence of $J^\varepsilon$, it is nevertheless observed numerically.

\begin{figure}[tbhp]
  \centering
  \subfloat[]{\label{fig:mass_RTE_vs_Schr1}
  \includegraphics[width=0.45\textwidth]{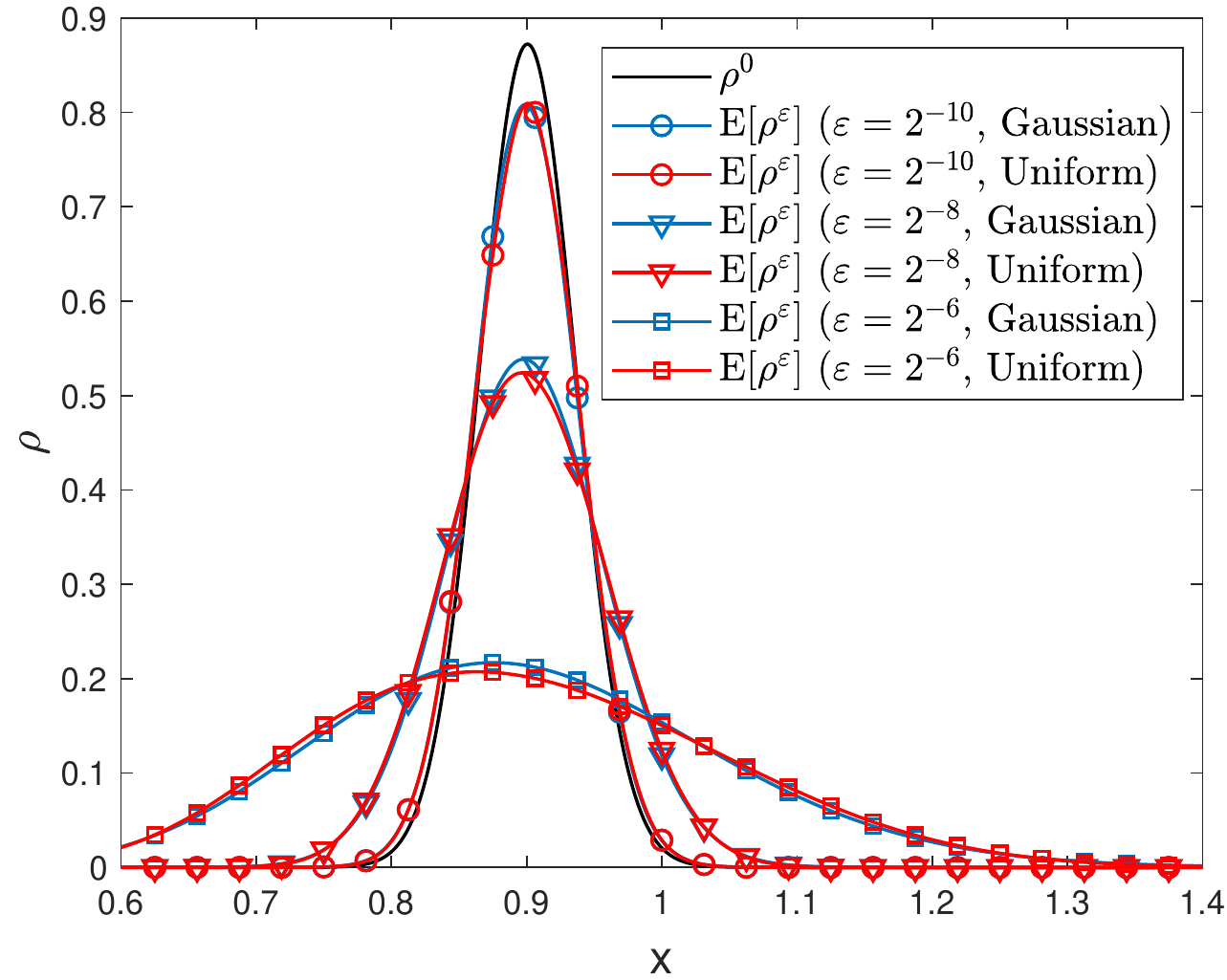}
  \includegraphics[width=0.45\textwidth]{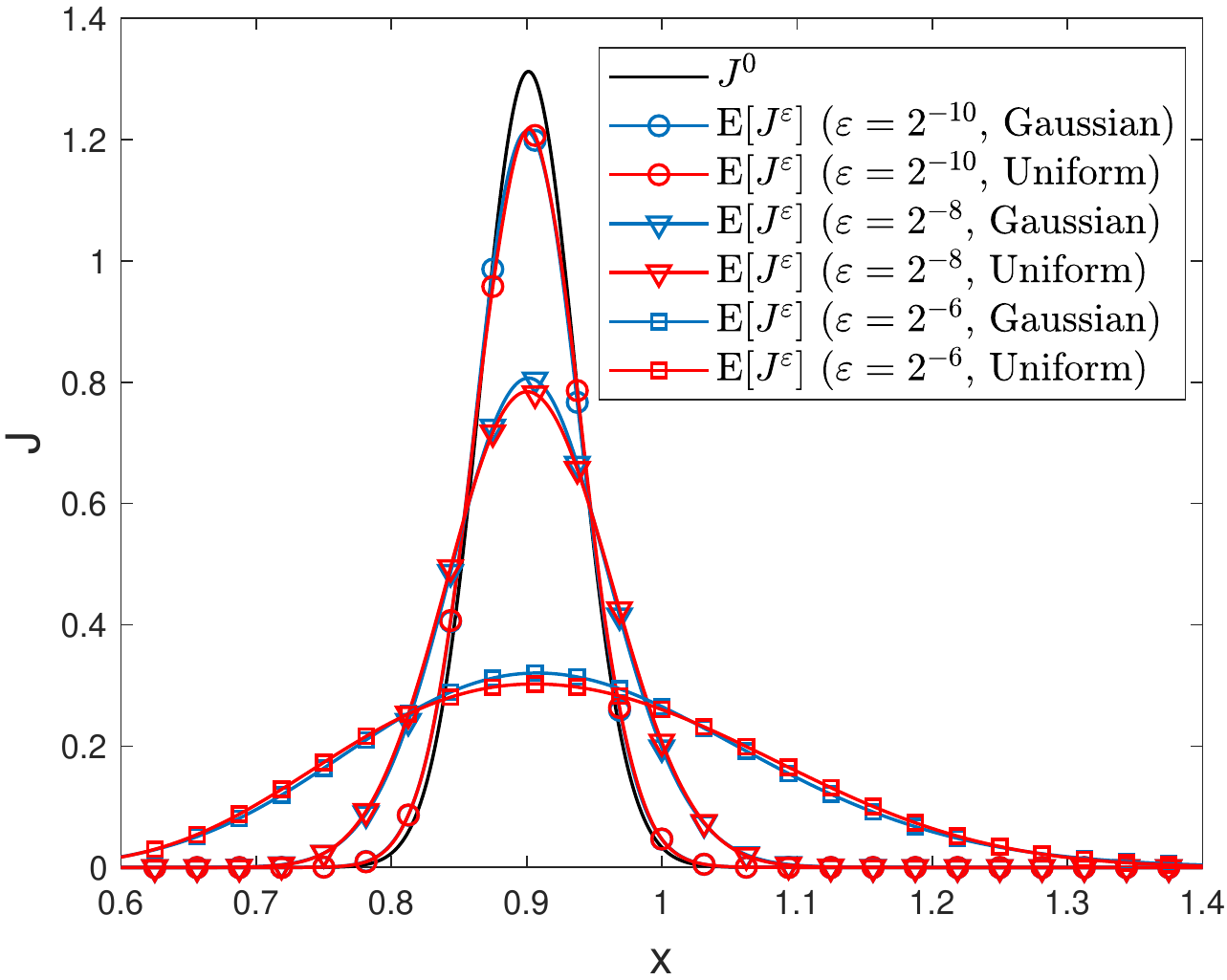}
  }

  \subfloat[]{\label{fig:mass_density_conv1}
  \includegraphics[width=0.45\textwidth]{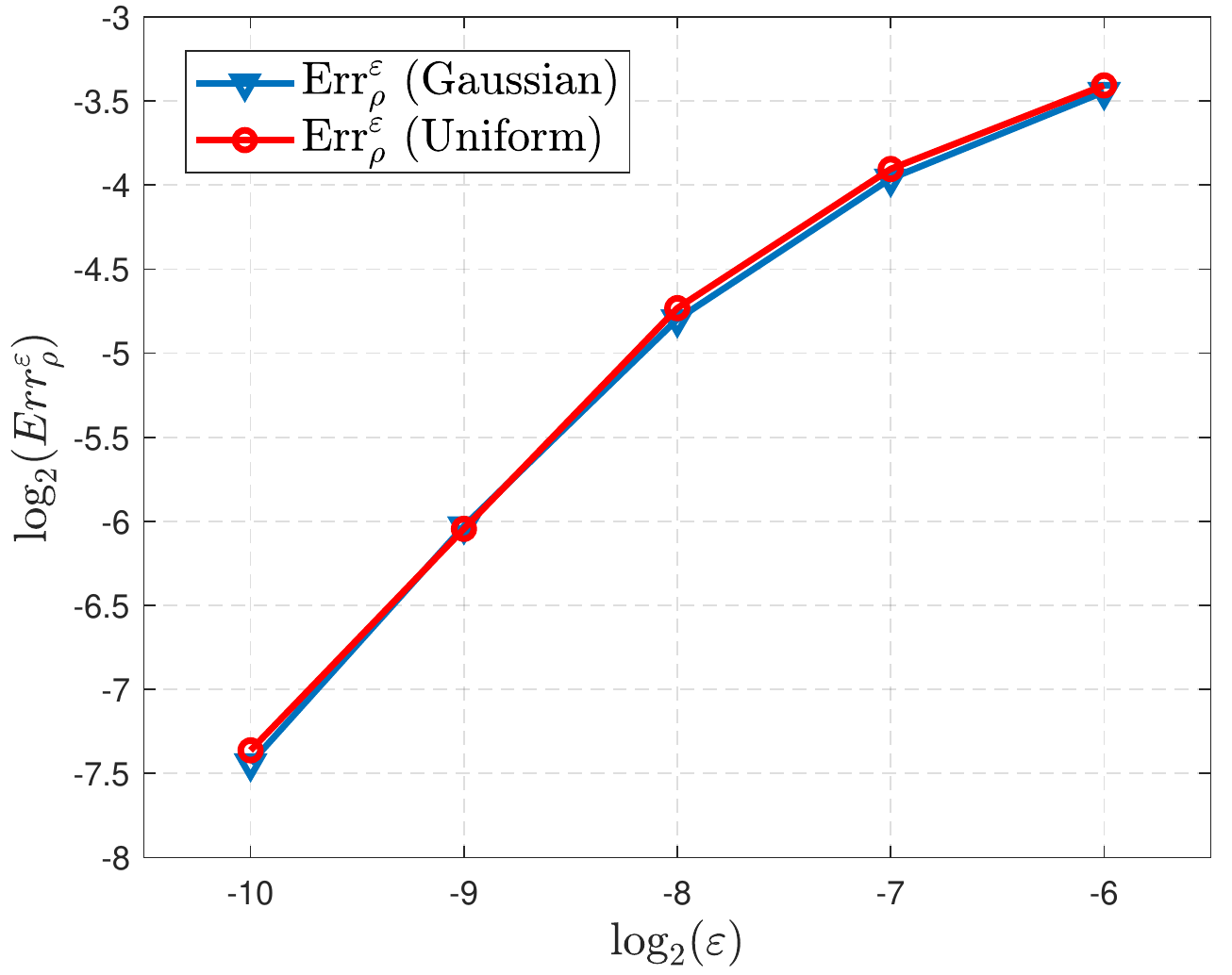}
  \includegraphics[width=0.45\textwidth]{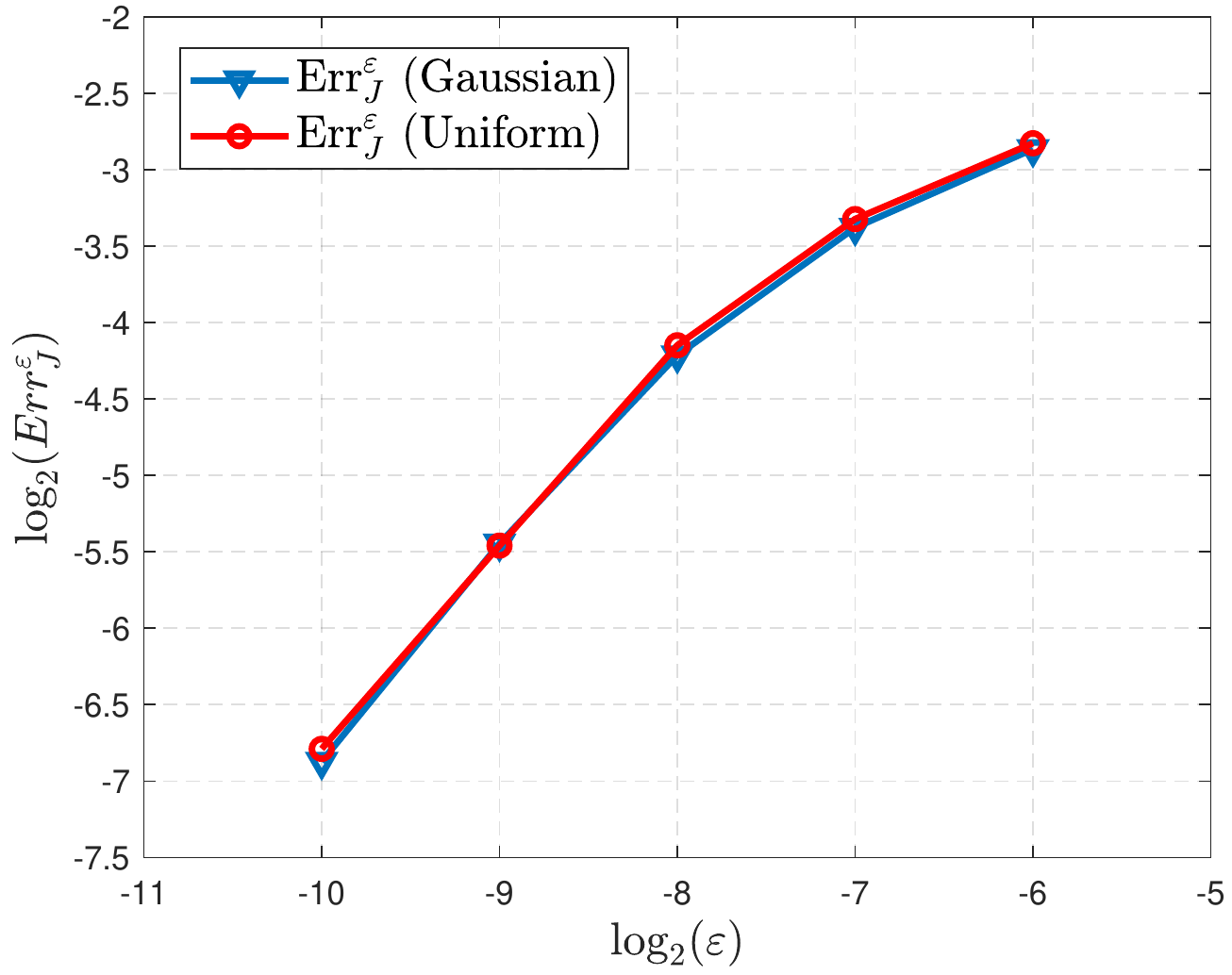}
  }
  \caption{(a) The plot compares particle density $\mathbb{E}[\rho^{\varepsilon}]$ (current density $\mathbb{E}[J^\varepsilon]$, \emph{respectively}) with $\rho^0$ ($J^0$, \emph{respectively}), defined in~\cref{eqn:density_ue} and~\cref{eqn:J_ue} at $t = 0.4$ for different $\varepsilon$ and different random distribution of $\xi_{ij}$. (b) The plot shows the $L^1$-error~\cref{eqn:err_rho} as a function of $\varepsilon$. Both Gaussian and Uniform distributions are used to sample $\xi_{ij}$. The decay rate suggests that $Err^{\varepsilon}_\rho$ and $Err^{\varepsilon}_J$ are both of $O(\varepsilon)$.}
\end{figure}

Although we do not derive the equation for the standard deviation, we do numerically investigate the statistics of $\rho$ and $J$. In particular, we set $\xi_{ij}$'s to be Gaussian random variables, and we plot, in~\cref{fig:mass_density_std1} and~\cref{fig:mass_J_std1} the standard deviation $\sigma[\rho^{\varepsilon}]$ and $\sigma[J^{\varepsilon}]$ for different $\varepsilon$, and in~\cref{fig:covariance} and~\cref{fig:covariance_J}, the covariance at $t=0.4$. The two quantities are defined as follows:
\begin{equation*}
\begin{aligned}
&\sigma[\rho^{\varepsilon}(t,x)] \approx \sqrt{ \frac{1}{N-1} \sum_{i = 1}^{N} (\rho_i^{\varepsilon}(t,x) - \mathbb{E}[\rho^{\varepsilon}(t,x)])^2}\,,\\
&\sigma[J^{\varepsilon}(t,x)] \approx \sqrt{ \frac{1}{N-1} \sum_{i = 1}^{N} (J_i^{\varepsilon}(t,x) - \mathbb{E}[J^{\varepsilon}(t,x)])^2}\,,
\end{aligned}
\end{equation*}
and
\begin{equation}
\begin{aligned}
&\text{Cov}(\rho^{\varepsilon}(t,x),\rho^{\eps}(t,y))
\approx  \frac{1}{N-1} \sum_{i = 1}^{N} (\rho_i^{\varepsilon}(t,x) - \mathbb{E}[\rho^{\varepsilon}(t,x)])(\rho_i^{\varepsilon}(t,y) - \mathbb{E}[\rho^{\varepsilon}(t,y)])\,,\\
&\text{Cov}(J^{\varepsilon}(t,x),J^{\eps}(t,y))
\approx  \frac{1}{N-1} \sum_{i = 1}^{N} (J_i^{\varepsilon}(t,x) - \mathbb{E}[J^{\varepsilon}(t,x)])(J_i^{\varepsilon}(t,y) - \mathbb{E}[J^{\varepsilon}(t,y)])\,.
\end{aligned}
\end{equation}
In the computation we set $D = 1.5$ in the correlation function~\cref{eqn:correlation} and $p_0 = 1.5$ in the initial data. Numerically we observe that with smaller $\varepsilon$ we have high standard deviation at the wave-packet center. We leave the mathematical justification to the future research.
\begin{figure}[htbp]
  \centering
  \subfloat[]{\label{fig:mass_density_std1}\includegraphics[width=0.45\textwidth]{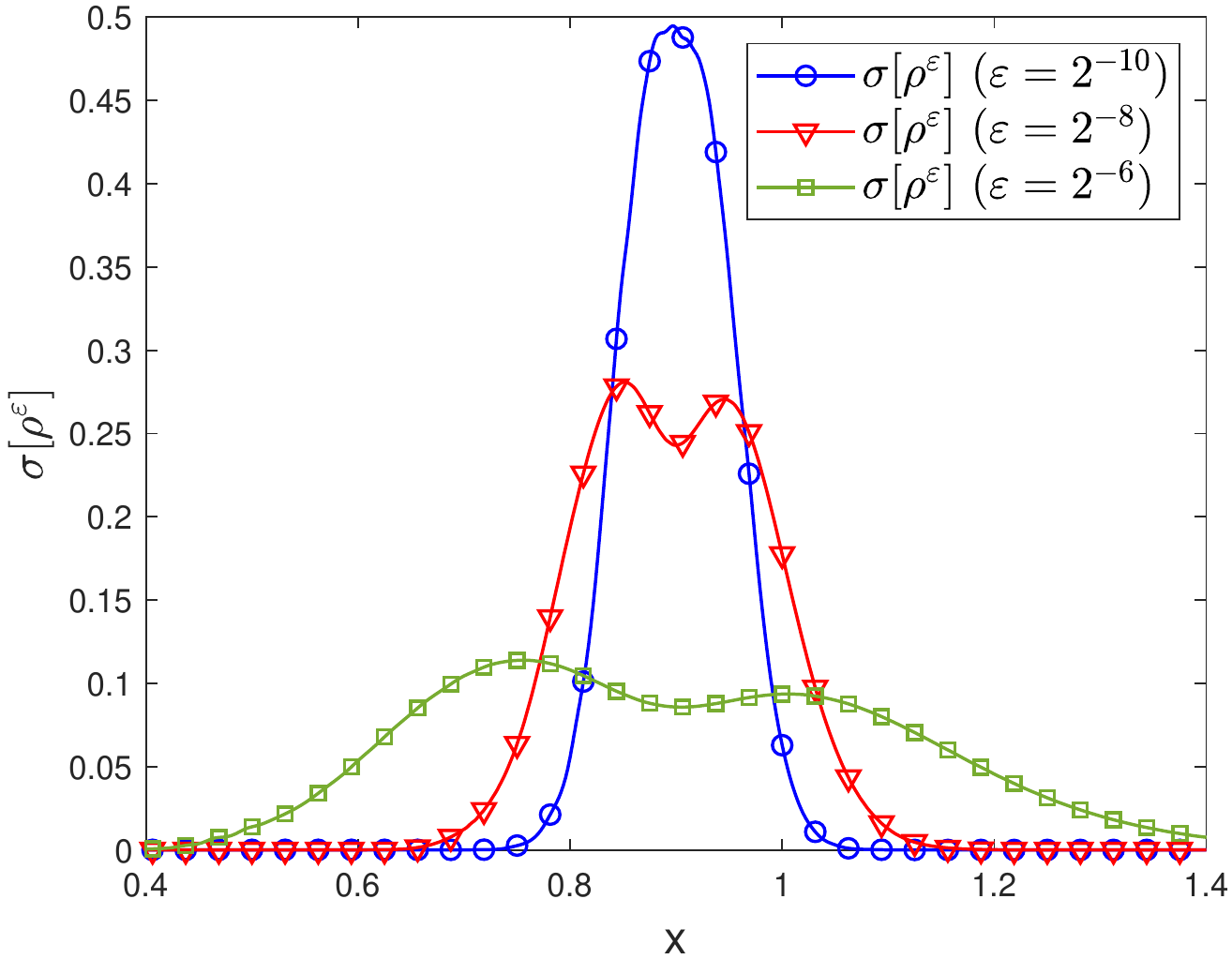}}
  \subfloat[]{\label{fig:mass_J_std1}\includegraphics[width=0.45\textwidth]{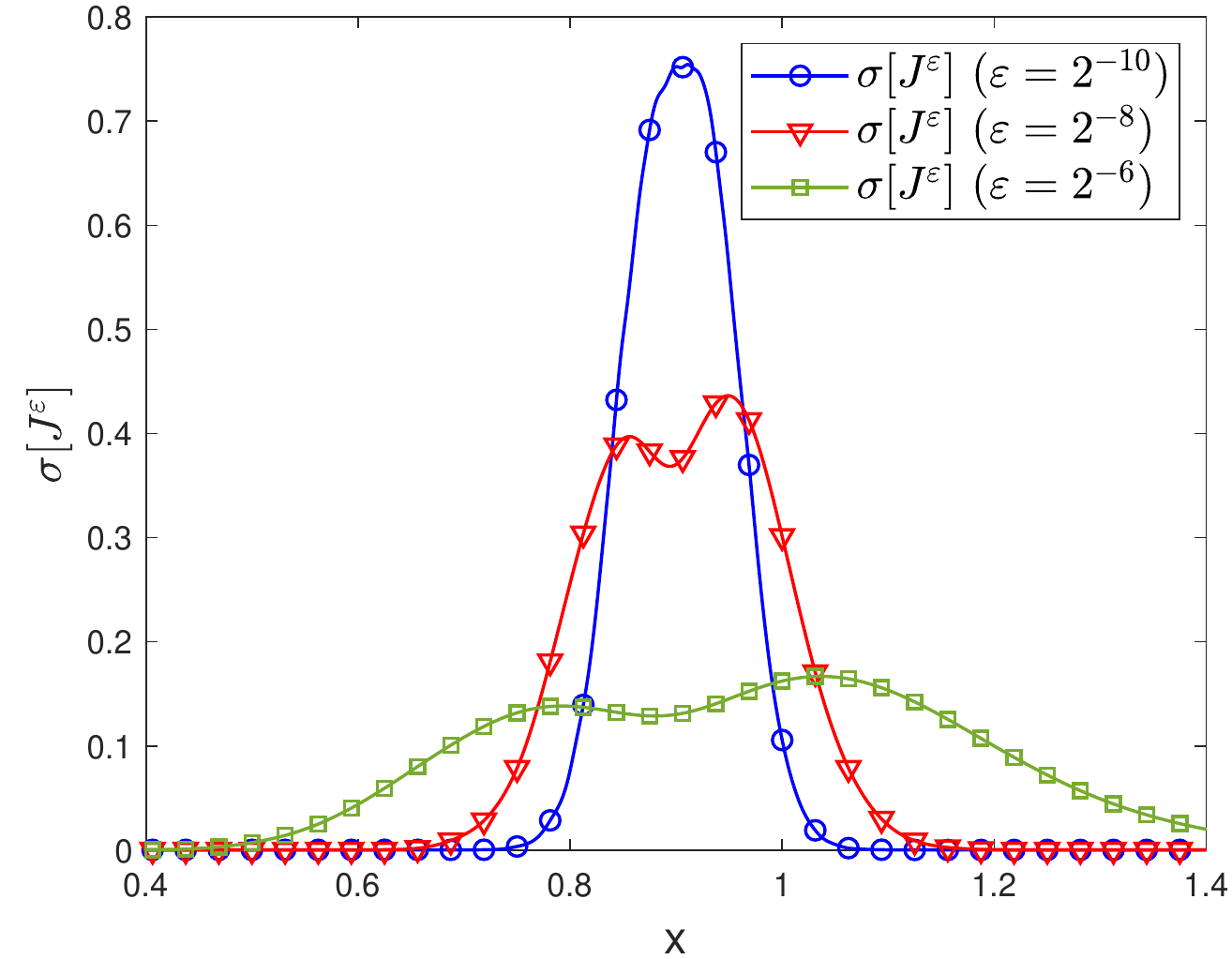}}

  \subfloat[]{\label{fig:covariance}\includegraphics[width=0.3\textwidth]{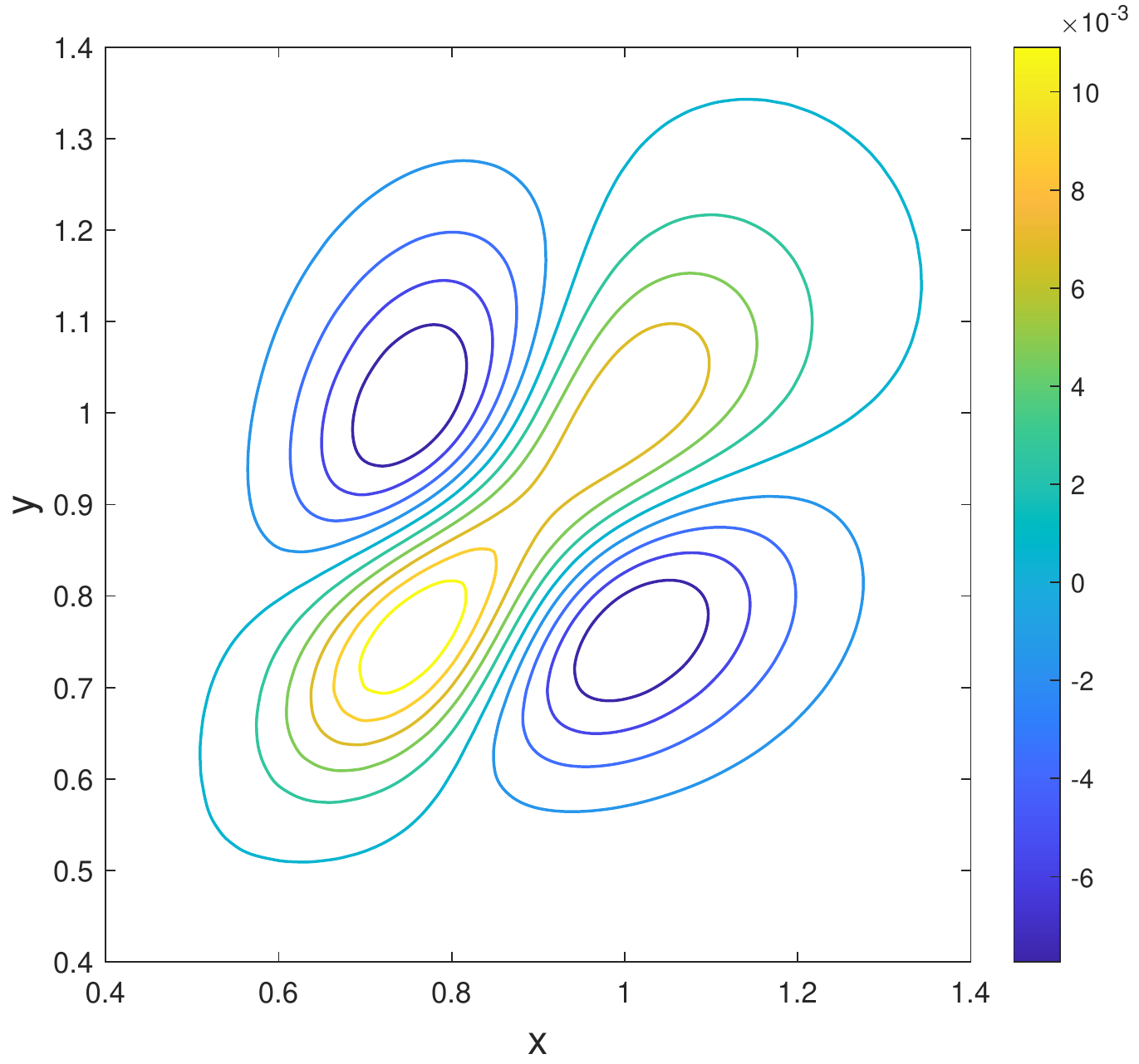}
  \includegraphics[width=0.3\textwidth]{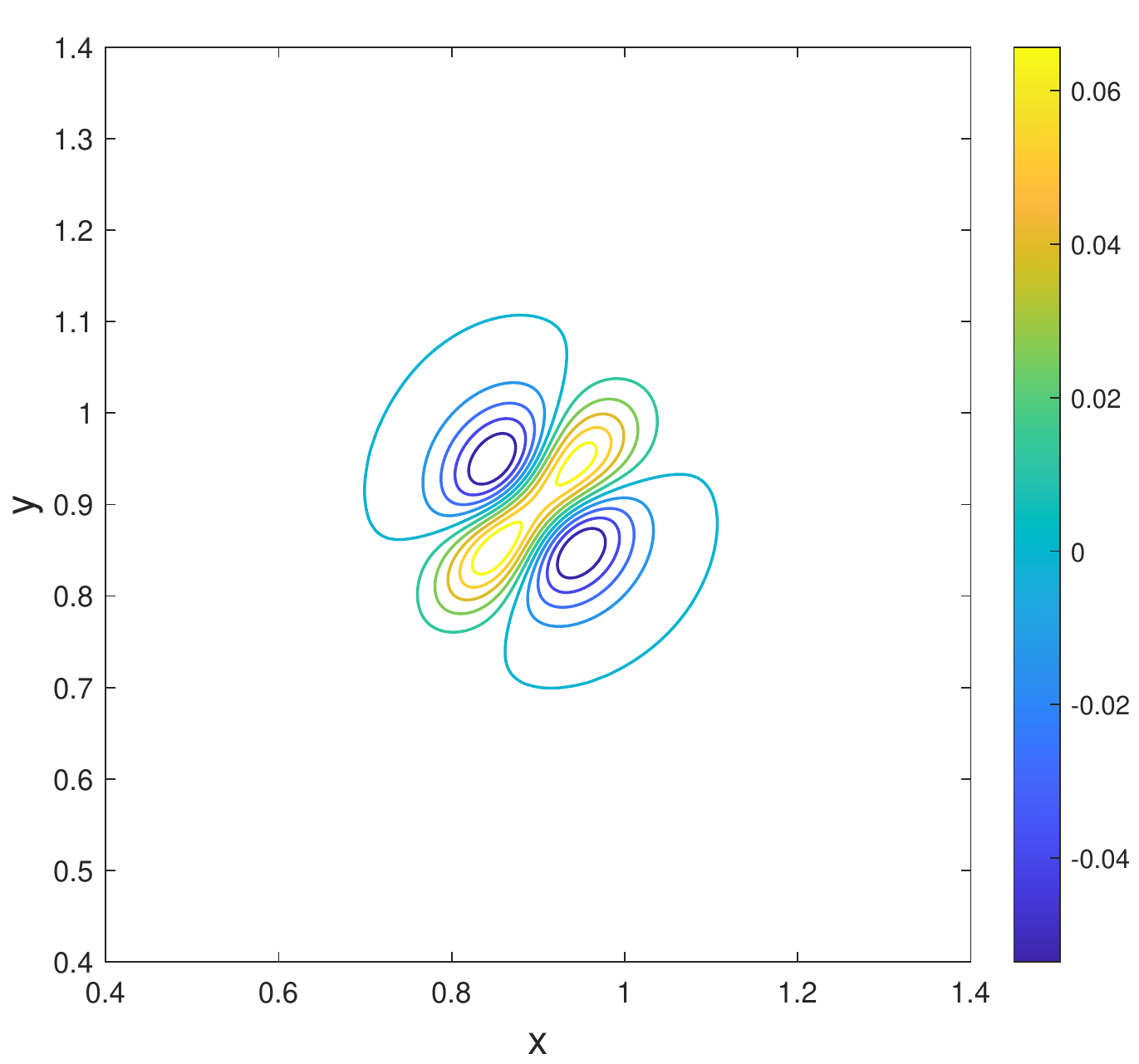}
  \includegraphics[width=0.3\textwidth]{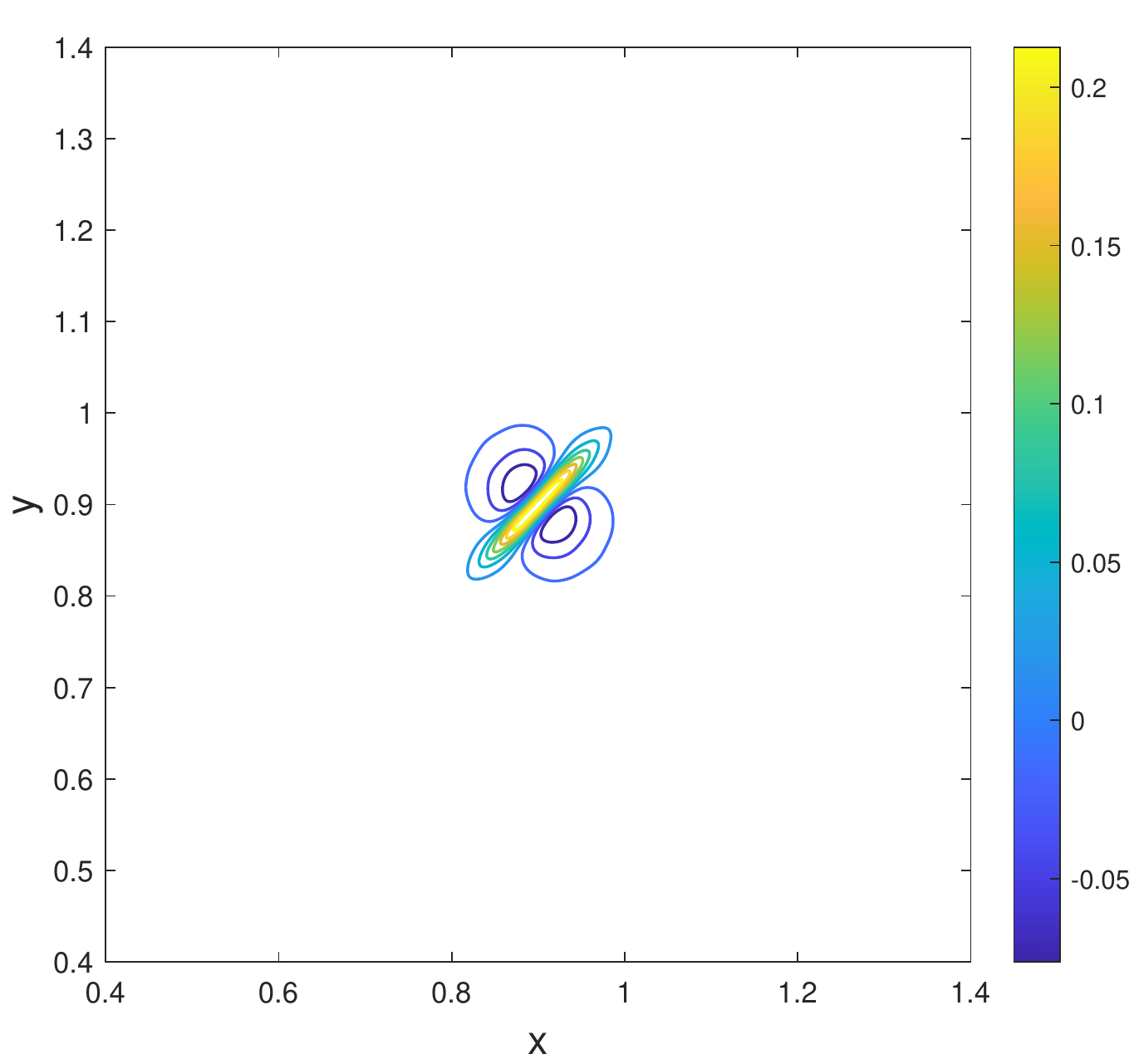}
  }

  \subfloat[]{\label{fig:covariance_J}\includegraphics[width=0.3\textwidth]{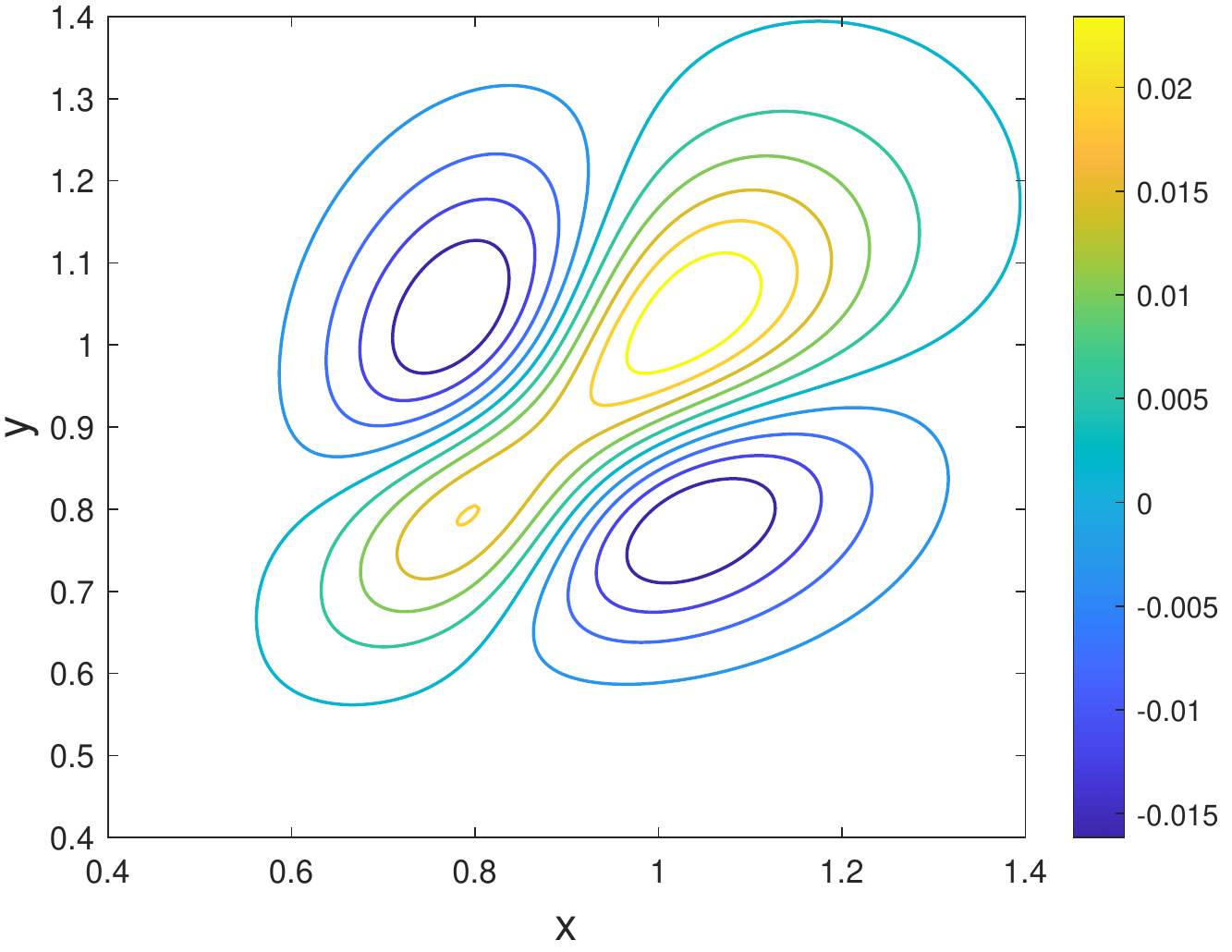}
  \includegraphics[width=0.3\textwidth]{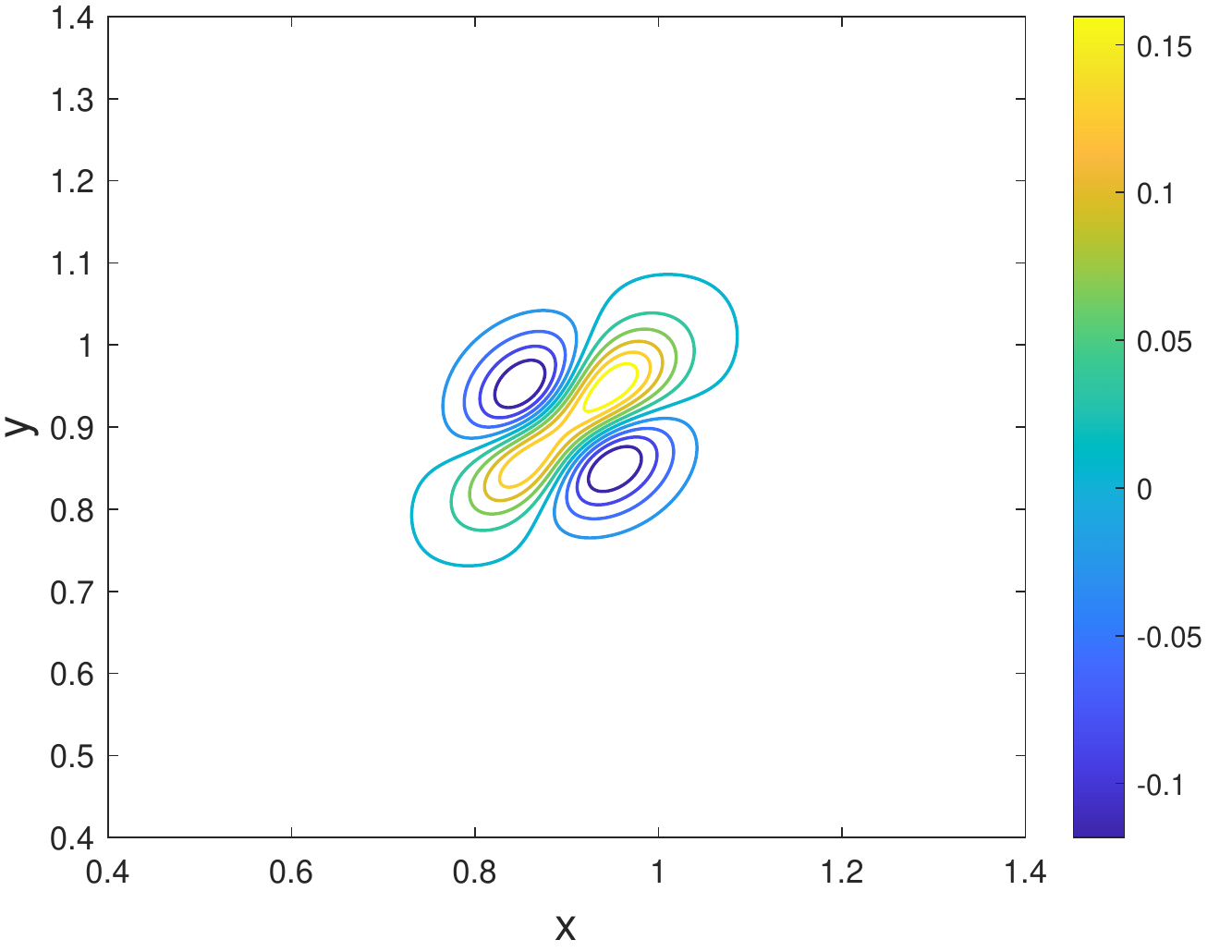}
  \includegraphics[width=0.3\textwidth]{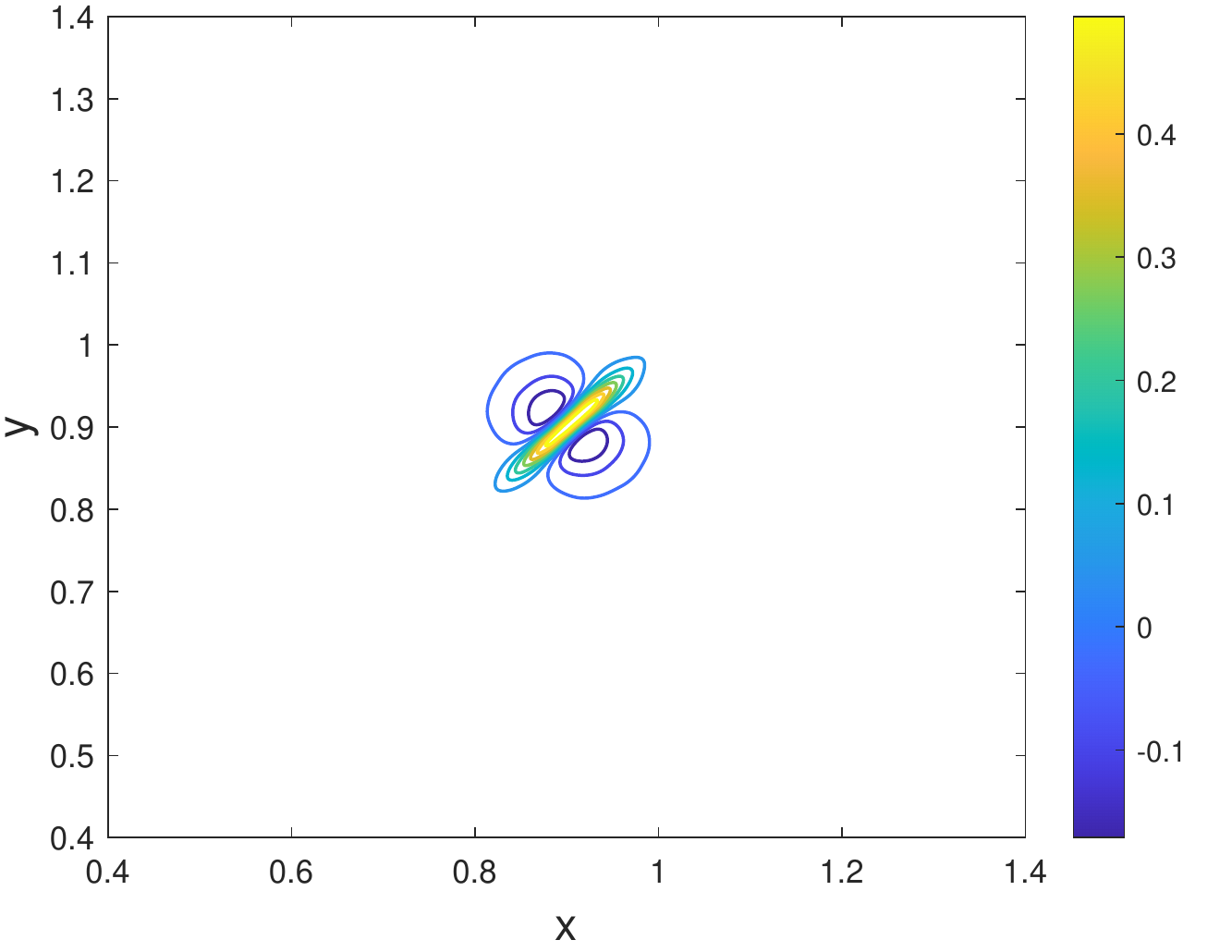}
  }
  \caption{ The plots (a)(b) show the standard deviation of particle density $\sigma[\rho^{\varepsilon}]$ and current density $\sigma[J^{\varepsilon}]$ at $t = 0.4$ for different $\varepsilon$. The plot (c) from left to right show the covariance $\text{Cov}(\rho^{\varepsilon}(x),\rho^{\eps}(y))$ at $t = 0.4$ for $\varepsilon = 2^{-6}\,,2^{-8}\,,2^{-10}$, respectively. The plot (d) shows the covariance $\text{Cov}(J^{\varepsilon}(x),J^{\varepsilon}(y))$ at $t = 0.4$ for the same $\varepsilon$ as (c). The random variables $\xi_{ij}$'s are chosen to be standard Gaussian random variables.}
\end{figure}

Finally we compare the CPU time of computing the limiting RTE and the reference Schr\"odinger equation. With the discretization mentioned above, it takes $5.8\times10^3$s to compute the RTE. In~\cref{tbl:time} we list the cost of solving the Schr\"odinger equation. It suggests that for $\varepsilon<2^{-9}$ one should switch to computing RTE as the limit for numerical efficiency. We consider $N = 10,000$ is big enough to have an accurate approximation of the statistical quantities.
\begin{table}
	\centering
	\begin{tabular}{c | c | c | c | c | c}
		\hline \hline
		$- \log_2 \varepsilon$					& 6   & 7 & 8 & 9 & 10         \\
        \hline
		CPU Time (s)                & $2.96\times 10^3$ & $7.74\times 10^3$ & $2.47\times 10^4$ & $4.25\times 10^5$ & $2.01\times 10^6$  \\
		\hline \hline
	\end{tabular}
    \caption{CPU time for computing 10,000 samples for VMSE with different $\varepsilon$ and $\xi_{ij}$'s being Gaussian random variables.}
\label{tbl:time}
\end{table}

In the second example, we purposely choose $m_1$ not to have the correct scaling as what we use in the derivation. In the derivation, we need the random perturbation to be have the order of $O(\sqrt{\varepsilon})$. This is a very typical scaling for the Schr\"odinger equation with random potential that leads to radiative transfer limit. Different scales may lead to different limits, as seen in~\cite{GuRy:2016,GuRy:2017,BaKoRy:2010}. For the VMSE, one would also expect $O(\sqrt{\varepsilon})$ to be also critical. In the following, we consider the VMSE with $O(\varepsilon)$ scale in random perturbation
\begin{equation}\label{eqn:mass_eps}
m^{\varepsilon}(t,x) = m_0(t,x) + \varepsilon m_1(\teps,\xeps)\,,
\end{equation}
and the VMSE with $O(\varepsilon^{0.4})$ scale in random mass
\begin{equation}\label{eqn:mass_eps_2_5}
m^{\varepsilon}(t,x) = m_0(t,x) + \varepsilon^{0.4} m_1(\teps,\xeps)\,.
\end{equation}
Here $m_1(t,x)$ is taken to be Gaussian random field with correlation function~\cref{eqn:correlation_num}. In~\cref{fig:mass_RTE_vs_Schr_eps}, we compare $\rho^0$ in RTE limit and $\mathbb{E}[\rho^\varepsilon]$ of VMSE with mass~\cref{eqn:mass_eps}, fixing $D = 1.5$ and $p_0= 1.5$. VMSE is computed using 10000 Monte Carlo samples and $\nkl=663, 3157, 27968$ for $\varepsilon = 2^{-6}, 2^{-8}, 2^{-10}$ respectively to ensure~\cref{eqn:KLtruncate}. It can be seen that the scattering produced by random perturbation is smaller in the limit $\varepsilon\rightarrow0$, which is also indicated in the standard deviation~\cref{fig:mass_density_std_eps}.
\begin{figure}[tbhp]
  \centering
  \subfloat[]{\label{fig:mass_RTE_vs_Schr_eps}\includegraphics[width=0.45\textwidth]{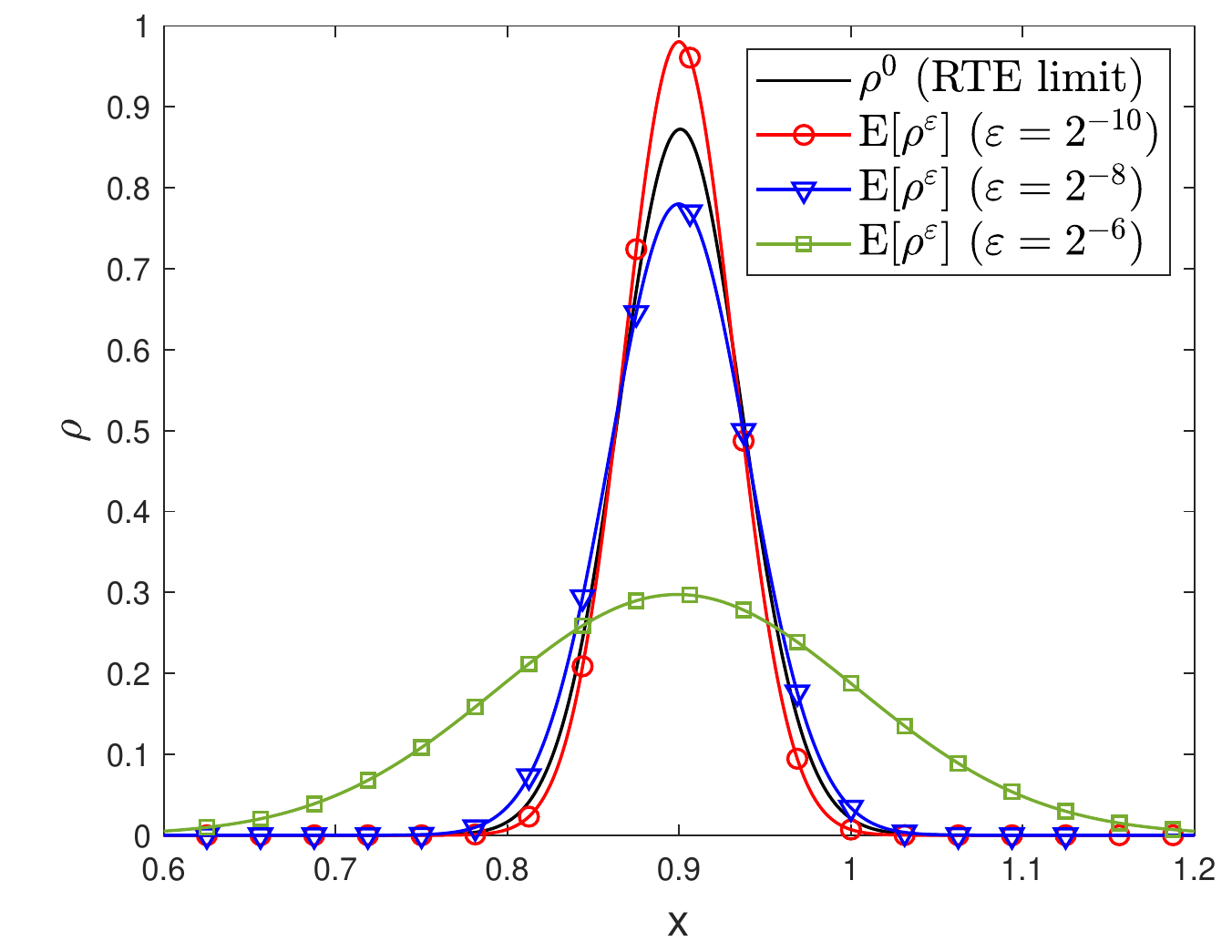}}
  \subfloat[]{\label{fig:mass_density_std_eps}\includegraphics[width=0.45\textwidth]{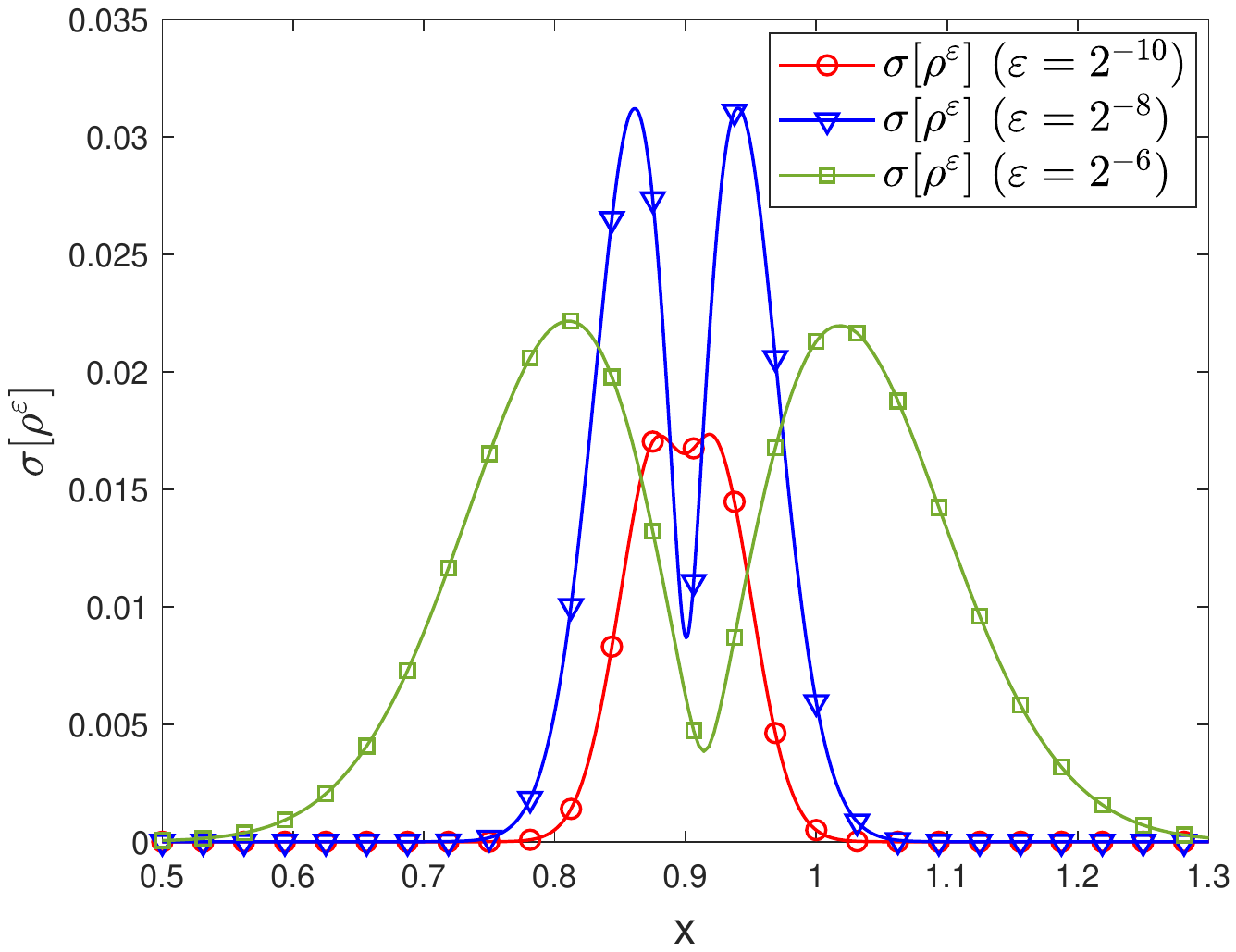}}
  \caption{(a) The plot compares particle density $\mathbb{E}[\rho^{\varepsilon}]$ for different $\varepsilon$ with $\rho^0$, defined in~\cref{eqn:density_ue} at $t = 0.4$. (b) The plot shows the standard deviation of particle density $\sigma[\rho^{\varepsilon}]$ at $t = 0.4$ for different $\varepsilon$. The random perturbation is of $O(\varepsilon)$ scale~\cref{eqn:mass_eps} with Gaussian $\xi_{ij}$ in the two plots.}
\end{figure}

In \Cref{fig:mass_RTE_vs_Schr_eps_2_5}, we plot $\mathbb{E}[\rho^\varepsilon]$ of VMSE with mass~\cref{eqn:mass_eps_2_5}, fixing $D = 1.5$ and $p_0= 1.5$. VMSE is computed in the same way as before. The standard deviations are plotted in~\cref{fig:mass_density_std_eps_2_5}. A much larger random scattering is now observed for this scale.
\begin{figure}[tbhp]
  \centering
  \subfloat[]{\label{fig:mass_RTE_vs_Schr_eps_2_5}\includegraphics[width=0.45\textwidth]{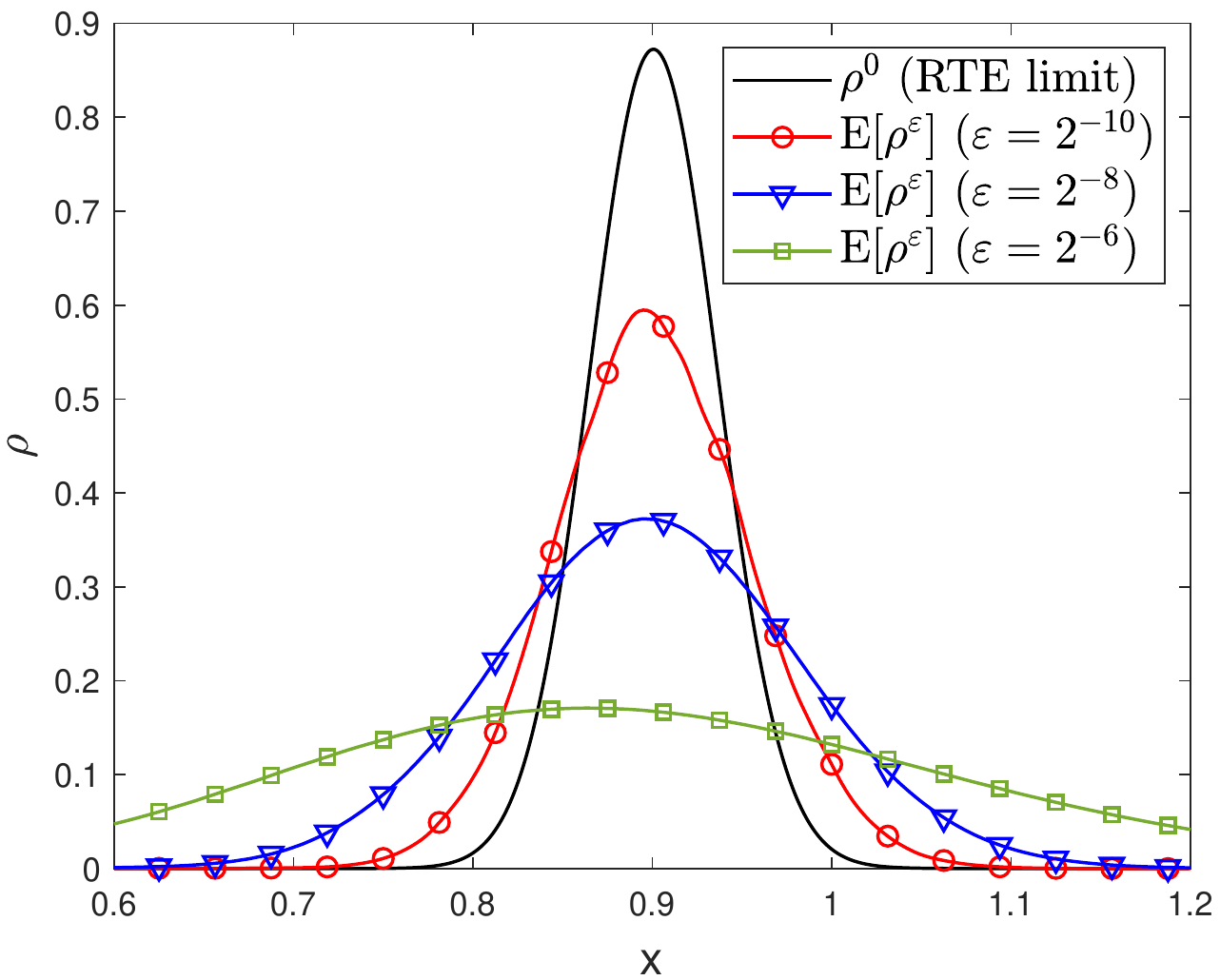}}
  \subfloat[]{\label{fig:mass_density_std_eps_2_5}\includegraphics[width=0.45\textwidth]{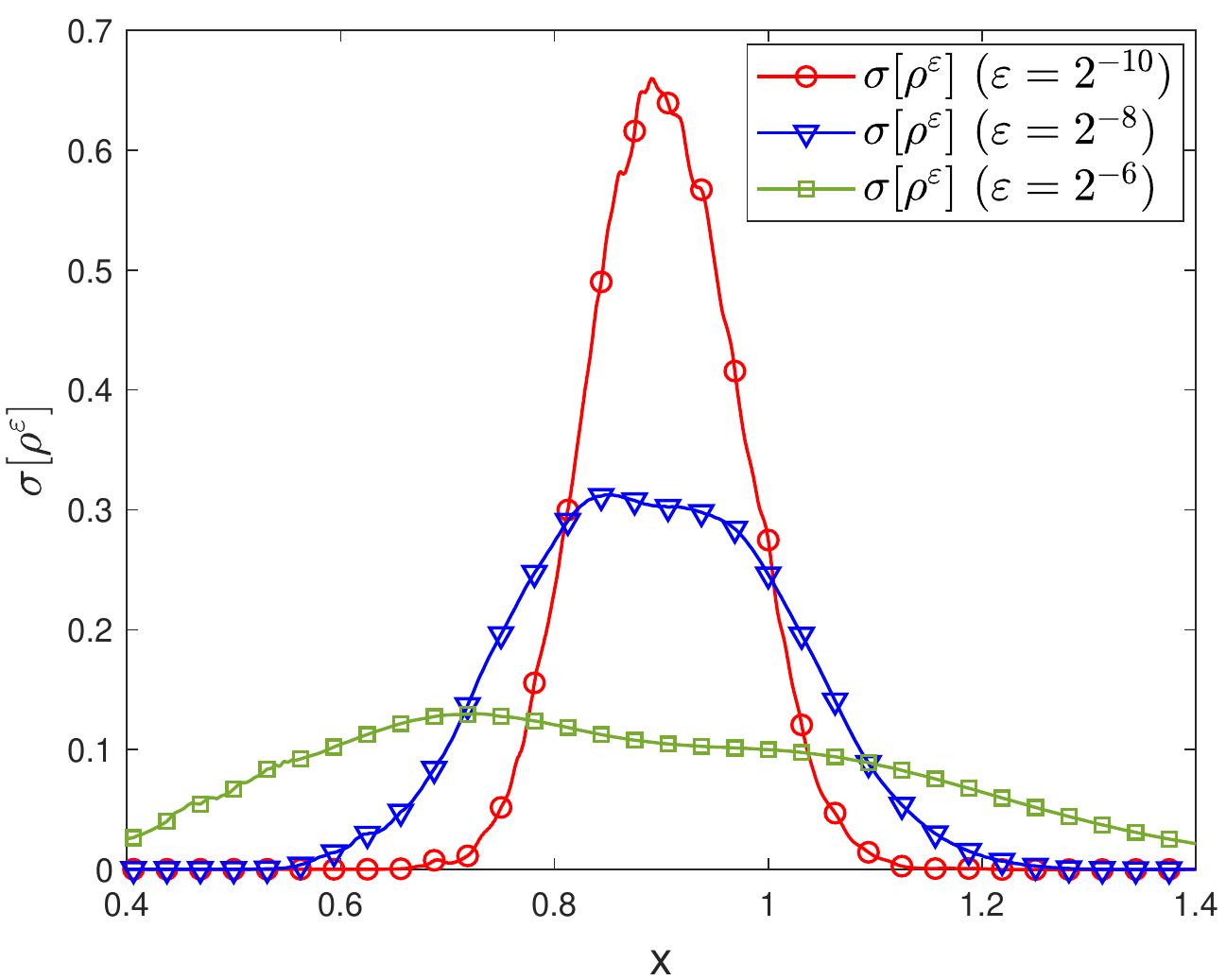}}
  \caption{(a) The plot compares particle density $\mathbb{E}[\rho^{\varepsilon}]$ for different $\varepsilon$ with $\rho^0$, defined in~\cref{eqn:density_ue} at $t = 0.4$. (b) The plot shows the standard deviation of particle density $\sigma[\rho^{\varepsilon}]$ at $t = 0.4$ for different $\varepsilon$. The random perturbation is of $O(\varepsilon^{0.4})$ scale~\cref{eqn:mass_eps_2_5} with Gaussian $\xi_{ij}$ in the two plots.}
\end{figure}

\section{Conclusion}\label{sec:conclusion}
In this paper, we systematically derived the radiative transfer equation for the solution to the varying-mass Schr\"odinger equation (VMSE) with random heterogeneities. In specific, we consider VMSE in the classical regime (the rescaled Planck constant $\varepsilon\ll 1$), and expand the corresponding Wigner equation to proper orders to obtain the asymptotic limit. We verify the derivation by numerically computing both VMSE and radiative transfer equations, and showing that the two solutions agree well.

\section*{Acknowledgment}
The research of S.C. and Q.L. are supported in part by NSF under DMS-1619778, DMS-1750488, and Wisconsin Data Science Initiatives. X.Y. was partially supported by the NSF grant DMS 1818592. Q.L. would like to thank Josselin Garnier and Guillaume Bal for valuable discussions.

\newpage

\begin{appendix}

\section{Proof of~\cref{lem:W}}\label{appendix}

The proof of~\cref{eqn:wignereqn} is direct computation using the VMSE~\cref{eqn:effect_mass_schr_m} and integration by parts. Notice that
\begin{equation}
\partial_t \we = \frac{1}{(2\pi)^d}\int_{\bbr^d} e^{\ri ky}\partial_t \ue(t,y)\uec(t,y)\rmd{y} +\frac{1}{(2\pi)^d}\int_{\bbr^d} e^{\ri ky} \ue(t,y)\partial_t\uec(t,y)\rmd{y}\,,
\end{equation}
we have, plugging in~\cref{eqn:effect_mass_schr_m}:
\begin{equation}
\begin{aligned}\label{eqn:partial_t_W}
\partial_t \we =& \frac{\ri\varepsilon}{2(2\pi)^d}\int_{\bbr^d} e^{\ri ky}\nabla_x\cdot\left(m_0\left(t,x-\epsh y\right)\nabla_x \ue\left(t,x-\epsh y\right)\right)\uec\left(t,x+\epsh y\right) \rmd y\\
              &-\frac{\ri\varepsilon}{2(2\pi)^d}\int_{\bbr^d} e^{\ri ky}\nabla_x\cdot\left(m_0\left(t,x+\epsh y\right)\nabla_x \uec\left(t,x+\epsh y\right)\right)\ue\left(t,x-\epsh y\right) \rmd y\\
              :=& \frac{\ri\varepsilon}{2(2\pi)^d}M_1-\frac{\ri\varepsilon}{2(2\pi)^d}M_2\,.
\end{aligned}
\end{equation}
Noting that $t$ serves as a parameter and doesn't affect the derivation, we suppress the $t$-dependence in the following.

Since the two terms $M_1$ and $M_2$ are conjugate with $y\to -y$ for the second term, we only study the first one. With integration by parts:
\begin{equation}\label{eqn:ibp}
\begin{aligned}
M_1=&\frac{2}{\eps}\int_{\mathbb{R}^d}\left[\nabla_y(e^{\ri ky})\cdot\nabla_x \ue\left(x-\epsh y\right)\right]m_0\left(x-\epsh y\right) \uec\left(x+\epsh y\right)\rmd y \\
&+\frac{2}{\eps}\int_{\mathbb{R}^d}e^{\ri ky}m_0\left(x-\epsh y\right)\nabla_x \ue\left(x-\epsh y\right) \cdot\nabla_y\uec\left(x+\epsh y\right)\rmd y\\
:=&I_1+I_2\,.
\end{aligned}
\end{equation}

We treat the $I_1$ and $I_2$ respectively in the following. Perform integration by parts again to $I_1$
\begin{equation}\label{eqn:ibp_2}
    \begin{aligned}
     I_1=&\frac{4}{\eps^2}\int_{\mathbb{R}^d} \Delta_y (e^{\ri ky})m_0\left(x-\epsh y\right) \left[\ue\left(x-\epsh y\right)\uec\left(x+\epsh y\right)\right]\rmd y\\
     &+\frac{4}{\eps^2}\int_{\mathbb{R}^d} \nabla_y (e^{\ri ky})\cdot\nabla_y m_0\left(x-\epsh y\right) \left[\ue\left(x-\epsh y\right)\uec\left(x+\epsh y\right)\right]\rmd y\\
     &+\frac{2}{\eps}\int_{\mathbb{R}^d}\left[\nabla_y(e^{\ri ky})\cdot\nabla_x \uec\left(x+\epsh y\right)\right]m_0\left(x-\epsh y\right) \ue\left(x-\epsh y\right)\rmd y\\
     :=&I_{11}+I_{12}+I_{13}\,.
    \end{aligned}
\end{equation}
Note that $I_1$ and the last term $I_{13}$ can be combined so that a complete $x-$gradient of $\ue\left(x-\epsh y\right)\uec\left(x+\epsh y\right)$ is available, namely one arrives at a formula for $I_1$
\begin{equation}\label{eqn:ibp_3}
    \begin{aligned}
     I_1=&\frac{1}{2}I_1+\frac{1}{2}(I_{11}+I_{12}+I_{13}) = \frac{1}{2}(I_1+I_{13})+\frac{1}{2}(I_{11}+I_{12})\\
     =&\frac{2}{\eps^2}\int_{\mathbb{R}^d} \Delta_y (e^{\ri ky})m_0\left(x-\epsh y\right) \left[\ue\left(x-\epsh y\right)\uec\left(x+\epsh y\right)\right]\rmd y\\
     +&\frac{2}{\eps^2}\int_{\mathbb{R}^d} \nabla_y e^{\ri ky}\cdot\nabla_y m_0\left(x-\epsh y\right) \left[\ue\left(x-\epsh y\right)\uec\left(x+\epsh y\right)\right]\rmd y\\
     +&\frac{1}{\eps}\int_{\mathbb{R}^d}\nabla_y(e^{\ri ky})\cdot\nabla_x \left[\ue\left(x-\epsh y\right)\uec\left(x+\epsh y\right)\right]m_0\left(x-\epsh y\right) \rmd y\,.
    \end{aligned}
\end{equation}

For $I_2$ in \cref{eqn:ibp}, integration by parts against $\nabla_x\ue\left(x-\epsh y\right)$ produces
\begin{equation}\label{eqn:ibp_4}
    \begin{aligned}
     I_2 = &\frac{4}{\eps^2}\int_{\mathbb{R}^d}\left[\nabla_y(e^{\ri ky})\cdot\nabla_y\uec\left(x+\epsh y\right)\right]  m_0\left(x-\epsh y\right) \ue\left(x-\epsh y\right) \rmd y \\
     +&\frac{4}{\eps^2}\int_{\mathbb{R}^d}e^{\ri ky}\left[\nabla_y m_0\left(x-\epsh y\right)\cdot\nabla_y\uec\left(x+\epsh y\right)\right]\ue\left(x-\epsh y\right) \rmd y\\
     +&\int_{\mathbb{R}^d}e^{\ri ky}m_0\left(x-\epsh y\right)\ue\left(x-\epsh y\right) \Delta_x\uec\left(x+\epsh y\right)\rmd y:=I_{21}+I_{22}+I_{23}\,.
    \end{aligned}
\end{equation}
On the other hand, integration by parts against $\nabla_y\uec\left(x+\epsh y\right)$ gives
\begin{equation}\label{eqn:ibp_5}
    \begin{aligned}
     I_2 = &\frac{4}{\eps^2}\int_{\mathbb{R}^d}\left[\nabla_y(e^{\ri ky})\cdot\nabla_y\ue\left(x-\epsh y\right)\right]  m_0\left(x-\epsh y\right) \uec\left(x+\epsh y\right) \rmd y \\
     +&\frac{4}{\eps^2}\int_{\mathbb{R}^d}e^{\ri ky}\left[\nabla_y m_0\left(x-\epsh y\right)\cdot\nabla_y\ue\left(x-\epsh y\right)\right] \uec\left(x+\epsh y\right) \rmd y\\
     +&\int_{\mathbb{R}^d}e^{\ri ky}m_0\left(x-\epsh y\right)\uec\left(x+\epsh y\right) \Delta_x\ue\left(x-\epsh y\right)\rmd y:=I'_{21}+I'_{22}+I'_{23}\,.
    \end{aligned}
\end{equation}
Note that $I_{21}$ and $I'_{21}$ can be combined after another integration by parts
\begin{equation}\label{eqn:ibp_6}
    \begin{aligned}
     I_{21} + I'_{21}= &-\frac{4}{\eps^2}\int_{\mathbb{R}^d}\nabla_y(e^{\ri ky})\cdot\nabla_y m_0\left(x-\epsh y\right)  \left[\ue\left(x-\epsh y\right)\uec\left(x+\epsh y\right)\right] \rmd y\\
     -&\frac{4}{\eps^2}\int_{\mathbb{R}^d}m_0\left(x-\epsh y\right) \Delta_y(e^{\ri ky})\left[\ue\left(x-\epsh y\right)\uec\left(x+\epsh y\right)\right]   \rmd y\,.
    \end{aligned}
\end{equation}
$I_{22}$ and $I'_{22}$ can be combined similarly
\begin{equation}\label{eqn:ibp_7}
    \begin{aligned}
     I_{22} + I'_{22} = &-\frac{4}{\eps^2}\int_{\mathbb{R}^d}\nabla_y(e^{\ri ky})\cdot\nabla_y m_0\left(x-\epsh y\right)  \left[\ue\left(x-\epsh y\right)\uec\left(x+\epsh y\right)\right] \rmd y\\
     -&\frac{4}{\eps^2}\int_{\mathbb{R}^d}e^{\ri ky}\Delta_ym_0\left(x-\epsh y\right)\left[\ue\left(x-\epsh y\right)\uec\left(x+\epsh y\right)\right]    \rmd y\,.
    \end{aligned}
\end{equation}
Hence using~\cref{eqn:ibp_4}-\cref{eqn:ibp_7} and the trick in~\cref{eqn:ibp_3}, one derives the formula for $I_2$ in~\cref{eqn:ibp}
\begin{equation}\label{eqn:ibp_8}
    \begin{aligned}
     I_2 = & \frac{1}{4}(I_{21}+I_{22}+I_{23})+\frac{1}{2}I_2+\frac{1}{4}(I'_{21}+I'_{22}+I'_{23}) \\
     = & \left(\frac{1}{4}I_{23}+\frac{1}{2}I_2 +\frac{1}{4}I'_{23} \right) + \frac{1}{4}(I_{21}+I'_{21}) + \frac{1}{4}(I_{22}+I'_{22})\\
     = & \frac{1}{4}\int_{\mathbb{R}^d}e^{\ri ky}m_0\left(x-\epsh y\right)\Delta_x\left[\ue\left(x-\epsh y\right) \uec\left(x+\epsh y\right)\right]\rmd y\\
     -&\frac{2}{\eps^2}\int_{\mathbb{R}^d}\nabla_y(e^{\ri ky})\cdot\nabla_y m_0\left(x-\epsh y\right)  \left[\ue\left(x-\epsh y\right)\uec\left(x+\epsh y\right)\right] \rmd y\\
     -&\frac{1}{\eps^2}\int_{\mathbb{R}^d} m_0\left(x-\epsh y\right)\Delta_y(e^{\ri ky})\left[\ue\left(x-\epsh y\right)\uec\left(x+\epsh y\right)\right]   \rmd y\\
     -&\frac{1}{\eps^2}\int_{\mathbb{R}^d}e^{\ri ky}\Delta_y m_0\left(x-\epsh y\right)\left[\ue\left(x-\epsh y\right)\uec\left(x+\epsh y\right)\right]    \rmd y\,.
    \end{aligned}
\end{equation}
Finally from~\cref{eqn:ibp_3} and~\cref{eqn:ibp_8}, one gets
\begin{equation}\label{eqn:ibp_9}
    \begin{aligned}
    M_1=&\frac{1}{4}\int_{\bbr^d} e^{\ri ky} m_0\left(x-\epsh y\right) \Delta_x\left[ \ue\left(x-\epsh y\right) \uec\left(x+\epsh y\right)\right]\rmd y\\
    &+\frac{1}{\varepsilon}\int_{\bbr^d} m_0\left(x-\epsh y\right) \nabla_y (e^{\ri ky})\cdot \nabla_x\left[ \ue\left(x-\epsh y\right) \uec\left(x+\epsh y\right)\right]\rmd y\\
    &+\frac{1}{\varepsilon^2}\int_{\bbr^d}  m_0\left(x-\epsh y\right)\Delta_y (e^{\ri ky}) \left[\ue\left(x-\epsh y\right) \uec\left(x+\epsh y\right)\right]\rmd y\\
    &-\frac{1}{\varepsilon^2}\int_{\bbr^d}  e^{iky}\Delta_y m_0\left(x-\epsh y\right) \left[\ue\left(x-\epsh y\right) \uec\left(x+\epsh y\right)\right]\rmd y\\
    :=& T_1+T_2+T_3+T_4\,.
\end{aligned}
\end{equation}

All the $T_i$ terms can be explicitly expressed by the Wigner transform~\cref{eqn:wigner_transform}. In particular:
\begin{equation}\label{eqn:uu_to_W_2}
\begin{aligned}
T_1 &= \int_{\bbr^d} e^{\ri px} \frmm(p) \Delta_x \we\left(x,k-\epsh p\right) \rmd p\,,\\
T_2 & = \int_{\bbr^d} e^{\ri px} \frmm(p) ik\cdot\nabla_x \we\left(x,k-\epsh p\right) \rmd p\,,\\
T_3 & = \int_{\bbr^d} -|k|^2 e^{\ri px} \frmm(p)  \we\left(x,k-\epsh p\right) \rmd p\,,\\
T_4 & = \eps^2 \int_{\bbr^d} |p|^2 e^{\ri px} \frmm(p) \we\left(x,k-\epsh p\right) \rmd p\,.
\end{aligned}
\end{equation}

We use $T_1$ as an example to show this. Recalling:
\[
\Delta_x\we(x,k) = \frac{1}{(2\pi)^d}\int e^{\ri ky}\Delta_x\left[\ue\left(x-\frac{\varepsilon}{2}y\right)\uec\left(x+\frac{\varepsilon}{2}y\right)\right]\rmd{y}\,,
\]
we have
\begin{equation}\label{eqn:u_to_W1}
    \begin{aligned}
&\int_{\bbr^d} e^{\ri px} \frmm(p) \Delta_x \we\left(x,k-\epsh p\right) \rmd p\\
=&\frac{1}{(2\pi)^d}\int \int\int e^{\ri px} e^{-\ri pz}m_0(z)e^{\ri (k-\frac{\varepsilon}{2}p) y}\Delta_x \left[\ue\left(x-\frac{\varepsilon}{2}y\right)\uec\left(x+\frac{\varepsilon}{2}y\right)\right] \rmd z\rmd p\rmd y\,,\\
=&\int_{\bbr^d} e^{\ri ky} m_0\left(x-\epsh y\right) \Delta_x\left[ \ue\left(x-\epsh y\right) \uec\left(x+\epsh y\right)\right]\rmd y=\frac{1}{4}T_1\,,
\end{aligned}
\end{equation}
where we used the fact that
\begin{equation}
\delta(x) = \invpi \int_{\bbr^d} e^{\ri xz} \rmd z\,,\quad\text{and}\quad \invpi \int\int f(x) e^{\ri xz} \rmd z\rmd x = f(0)\,.
\end{equation}

Using~\cref{eqn:uu_to_W_2}, we get
\begin{equation}\label{eqn:ibp_s1}
    \begin{aligned}
    M_1=&\frac{1}{4}\int_{\bbr^d} e^{\ri px} \frmm(p) \Delta_x \we\left(x,k-\epsh p\right) \rmd p\\ &+\frac{1}{\varepsilon}\int_{\bbr^d} e^{\ri px} \frmm(p) \ri k\cdot\nabla_x \we\left(x,k-\epsh p\right) \rmd p\\
    &+\int_{\bbr^d} |p|^2 e^{\ri px} \frmm(p) \we\left(x,k-\epsh p\right) \rmd p\\
    &-\frac{1}{\varepsilon^2}\int_{\bbr^d} |k|^2 e^{\ri px} \frmm(p) \we\left(x,k-\epsh p\right) \rmd p\,.
\end{aligned}
\end{equation}
By the conjugate argument, one gets, setting $p\to -p$:
\begin{equation}\label{eqn:ibp_s2}
    \begin{aligned}
    M_2=&\frac{1}{4}\int_{\bbr^d} e^{\ri px} \frmm(p) \Delta_x \we\left(x,k+\epsh p\right) \rmd p\\
    &-\frac{1}{\varepsilon}\int_{\bbr^d} e^{\ri px} \frmm(p) \ri k\cdot\nabla_x \we\left(x,k+\epsh p\right) \rmd p\\
    &+\int_{\bbr^d} |p|^2 e^{\ri px} \frmm(p) \we\left(x,k+\epsh p\right) \rmd p\\
    &-\frac{1}{\varepsilon^2}\int_{\bbr^d} |k|^2 e^{\ri px} \frmm(p) \we\left(x,k+\epsh p\right) \rmd p\,.
\end{aligned}
\end{equation}
Finally, substitute~\cref{eqn:ibp_s1} and~\cref{eqn:ibp_s2} into~\cref{eqn:partial_t_W}, and we arrive at the Wigner equation in~\cref{eqn:wignereqn}.

\end{appendix}

\newpage
\bibliographystyle{siamplain}
\bibliography{PWE}
\end{document}